\documentclass[11pt,reqno]{amsart}

\usepackage[utf8]{inputenc}
\usepackage[margin=1.25in]{geometry}
\usepackage{hyperref}
\usepackage{appendix}
\usepackage{amsfonts}
\usepackage{amsthm}
\usepackage{amssymb}
\usepackage{stmaryrd} 
\usepackage{amsmath}
\usepackage{amsthm}
\usepackage[dvipsnames]{xcolor}
\usepackage{mathrsfs}
\usepackage{lipsum} % for filler text

\theoremstyle{plain}
\newtheorem{theorem}{Theorem}[section]
\newtheorem{corollary}[theorem]{Corollary}
\newtheorem{lemma}[theorem]{Lemma}
\newtheorem{Proposition}[theorem]{Proposition}

\newtheorem{Notation}[theorem]{Notation}

\newtheorem{Definition}[theorem]{Definition}

\newtheorem{fact}[theorem]{Fact}

\newtheorem{Claim}[theorem]{Claim}
\newtheorem{Conjecture}[theorem]{Conjecture}

\newtheorem{reduction}[theorem]{Reduction}

\theoremstyle{remark}

\newtheorem{remark}[theorem]{Remark}

\usepackage{biblatex} %Imports biblatex package
\numberwithin{equation}{section}
\title[Repeated singular values and decoupled estimates]{Repeated singular values of a random symmetric matrix and decoupled singular value estimates}

\author{Yi HAN}
\address{Department of Mathematics, Massachusetts Institute of Technology, Cambridge, MA
}
\email{hanyi16@mit.edu}

\begin{document}

\begin{abstract} Let $A_n$ be a random symmetric matrix with Bernoulli $\{\pm 1\}$ entries. For any $\kappa>0$ and two real numbers $\lambda_1,\lambda_2$ with a separation $|\lambda_1-\lambda_2|\geq \kappa n^{1/2}$ and both lying in the bulk $[-(2-\kappa)n^{1/2},(2-\kappa)n^{1/2}]$, we prove a joint singular value estimate
$$
\mathbb{P}(\sigma_{min}(A_n-\lambda_i I_n)\leq\epsilon n^{-1/2};i=1,2)\leq C\epsilon^2+2e^{-cn}.
$$ For general subgaussian distribution and a mesoscopic separation $|\lambda_1-\lambda_2|\geq \kappa n^{-1/2+\sigma},\sigma>0$ we prove the same estimate with $e^{-cn}$ replaced by an exponential type error. This means that extreme behaviors of the least singular value at two locations can essentially be decoupled all the way down to the exponential scale when the two locations are separated. As a corollary, we prove that all the singular values of $A_n$ in $[\kappa n^{1/2},(2-\kappa)n^{1/2}]$ are distinct with probability $1-e^{-cn}$, and with high probability the minimal gap between these singular values has order at least $n^{-3/2}$. This justifies, in a strong quantitative form, a conjecture of Vu up to $(1-\kappa)$-fraction of the spectrum for any $\kappa>0$.

\end{abstract}

\maketitle

\section{Introduction}

Let $A_n$ be an $n\times n$ random symmetric matrix whose upper diagonal entries $(A_{ij})_{1\leq i\leq j\leq n}$ are i.i.d. mean 0 variance 1 random variables. This model is also called a Wigner matrix, and its empirical eigenvalue density is known to converge to the Wigner semicircle law \cite{wigner1958distribution}.

In this work we investigate the extreme behavior of eigenvalues of $A_n$ around a fixed point, and more importantly on the correlation (and decoupling) of eigenvalue behavior for two eigenvalues of $A_n$ around separated locations in the spectrum. As a corollary, we investigate certain extreme behavior of singular values of $A_n$.

One of the initial motivations for this paper arises from Conjecture 8.5 of \cite{vu2021recent}: 
  \begin{Conjecture}\label{conjecture1.6}
  let $M_n^{sym}$ be a random symmetric matrix with Bernoulli $\{\pm 1\}$ entries. Is it true that with probability $1-o(1)$, all the singular values of $M_n^{sym}$ are distinct?
\end{Conjecture}

As $M_n^{sym}$ is a symmetric matrix, this conjecture is equivalent to asking if, with high probability, there does not exist any pair of two eigenvalues of $M_n^{sym}$ whose sum is zero.

Several problems of a similar flavor have been investigated in the past few years. A first natural problem is, are the eigenvalues of $M_n^{sym}$ distinct? An affirmative answer was obtained in \cite{tao2017random}, and a quantitative lower bound for the eigenvalue gap was obtained in \cite{nguyen2017random}. This was later generalized to sparse random matrices \cite{MR4164838}, \cite{lopatto2021tail} and to some other models with dependence \cite{christoffersen2024eigenvalue}. In \cite{campos2024least} it was proven that the eigenvalues of $M_n^{sym}$ are distinct with probability $1-e^{-cn}$. The universality of minimal eigenvalue gap distribution for certain Wigner matrices was proven in \cite{bourgade2021extreme}. Another closely related problem is, are the singular values of a random matrix with i.i.d. entries distinct? This question and some generalized versions of it are recently positively answered in two concurrent works \cite{han2025simplicity} and \cite{christoffersen2025gaps}.

Despite all the recent progress, no positive answer seems to have been reached for Conjecture \ref{conjecture1.6}. The main technical challenge is that all the previous works use in a crucial way the Cauchy interlacing theorem, and the fact that we can remove a row and column with the same index from a random symmetric matrix such that the removed column is independent from the remaining $(n-1)$-principal minor. If we consider the singular value of a random symmetric matrix, one can of course consider a linearized version $\begin{pmatrix}
    0&M_n^{sym}\\M_n^{sym}&0
\end{pmatrix}$ whose nonnegative eigenvalues coincide with the singular values of $M_n^{sym}$, but unfortunately we cannot apply the column removal trick to this linearized matrix due to lack of independence. In this paper we propose a new method to deduce much stronger information for the spectrum of $M_n^{sym}$ and, as a corollary, prove Conjecture \ref{conjecture1.6} in a strong quantitative form with a high probability estimate for the smallest gap between singular values, for an $(1-\kappa)$-fraction of the support of the singular values and for any $\kappa>0$.

The main results of this paper are the following two-point singular value estimates. For a square matrix $A$ we always denote by $\sigma_{min}(A)$ the least singular value of $A$.

\begin{theorem}\label{Theorem1.1} Let $\zeta$ be a mean 0, variance 1 random variable with a subgaussian tail. We also assume that $\zeta$ has a finite Log-Sobolev constant.

Let $A_n$ be an $n\times n$ symmetric random matrix with upper-diagonal entries $\{A_{i,j}\}_{1\leq i\leq j\leq n}$ being independent and identically distributed with distribution $\zeta$, and $A_{ij}=A_{ji}.$

Fix two positive numbers $\kappa>0$ and $\Delta>0$, and consider two locations $\lambda_1,\lambda_2\in[-(2-\kappa)\sqrt{n},(2-\kappa)\sqrt{n}]$ such that $|\lambda_1-\lambda_2|\geq \Delta n^{1/2}$. 

Then there exist constants $C>0,c>0$ depending only on the random variable $\zeta$ and the constants $\kappa,\Delta$ such that for any $\delta_1,\delta_2\geq 0$, we have the following singular value estimate jointly at $\lambda_1$ and $\lambda_2$:
\begin{equation}\label{gasagagga} 
    \mathbb{P}\left(\sigma_{min}(A_n-\lambda_1 I_n)\leq\delta_1 n^{-1/2},\sigma_{min}(A_n-\lambda_2 I_n)\leq\delta_2 n^{-1/2}\right)\leq C\delta_1\delta_2+2e^{-cn}. 
\end{equation}

\end{theorem}

Theorem \ref{Theorem1.1} is a two-point version of the recent breakthrough of minimal singular value estimate for Wigner matrices \cite{campos2024least}. The main result of \cite{campos2024least} is that, under the assumption of Theorem \ref{Theorem1.1} (without assuming Log-Sobolev constant for $\zeta$), we have for any $\epsilon\geq 0$,
\begin{equation}\label{onelocationbounds}
\mathbb{P}(\sigma_{min}(A_n)\leq\epsilon n^{-1/2})\leq C\epsilon+2e^{-cn}.
\end{equation} This result is a culmination of a line of progress on lower tails of singular values including \cite{vershynin2014invertibility} and \cite{jain2022smallest}, and on the invertibility problem of a random symmetric matrix from \cite{costello2006random}, \cite{nguyen2012inverse}, \cite{vershynin2014invertibility}, \cite{ferber2019singularity}, \cite{campos2022singularity},\cite{campos2025singularity}. When the entries of $A_n$ have a continuous density and satisfy certain estimates on Fourier decay, the same estimates of least singular values was obtained much earlier in \cite{erdHos2010wegner} and without the $e^{-cn}$ error. In another direction, when $\epsilon$ is large, say $\epsilon\geq n^{-c}$, another set of methods, namely the fixed energy universality \cite{bourgade2016fixed}, has been able to match the distribution of $\sigma_{min}(A_n)$ with the Gaussian case, and yield the bound $\mathbb{P}(\sigma_{min}(A_n)\leq\epsilon n^{-1/2})\leq \epsilon+o(1)$ which gives the correct leading term in $\epsilon$.

The main interest in our results, Theorem \ref{Theorem1.1} and \ref{Theorem1.2}, lie in the case when $\delta_1,\delta_2$ are very small, say much smaller than $n^{-1/2}$, and the numerical constant $C$ can possibly be very large, so the result can be stated as a decoupling of eigenvalue events up to a (large) multiplicative constant. The universality of joint singular value distributions at several locations can possibility be revealed via the recent technique of multi-resolvent local law \cite{cipolloni2024eigenvector} and other developments. However, none of these seems to be able to go beyond  the range $\delta_1,\delta_2\geq n^{-c}$ for even a small $c>0$ at several locations, which is the goal of this paper.

Since a Bernoulli $\{\pm 1\}$ variable has a finite Log-Sobolev constant, Theorem \ref{Theorem1.1} immediately applies to symmetric Bernoulli matrices. Theorem \ref{Theorem1.1} states that, as long as the two locations $\lambda_1,\lambda_2$ are in the bulk of the semicircle law and are separated by a macroscopic distance $\Delta n^{1/2}$, then we can completely decouple the least singular value events around the two locations $\lambda_1,\lambda_2$. The resulting estimate has the form of the product of the two one-point probability estimates \eqref{onelocationbounds} at $\lambda_1,\lambda_2$ respectively.

It is also very natural to consider $\lambda_1,\lambda_2$ with a mesoscopic separation $|\lambda_1-\lambda_2|\geq n^{-1/2+c}$.
   In Theorem \ref{Theorem1.2}, we show that the main results of Theorem \ref{Theorem1.1} remain true at this generality, with the mild difference that we replace the exponential error $e^{-cn}$ by an exponential type error. Also, in this result we do not need to assume $\zeta$ has a finite Log-Sobolev constant, and any subgaussian random variable suffices.

    \begin{theorem}\label{Theorem1.2}(Mesoscopic distance, without assuming LSI) Let $\zeta$ be a mean 0, variance 1 random variable with a subgaussian tail. Let $A_n$ be an $n\times n$ symmetric matrix with upper diagonal entries $\{A_{i,j}\}_{1\leq i\leq j\leq n}$ being i.i.d. with the distribution $\zeta$, and $A_{ij}=A_{ji}$.

Fix positive numbers $\kappa>0,\Delta>0$ and $\sigma\in(0,1]$.
Consider two fixed locations $\lambda_1,\lambda_2\in[-(2-\kappa)\sqrt{n},(2-\kappa)\sqrt{n}]$ such that $|\lambda_1-\lambda_2|\geq \Delta n^{\sigma-1/2}$.

Then there exist two constants $c>0,C>0$ depending on $\zeta$ and on $\kappa,\Delta,\sigma$ such that for any $\delta_1,\delta_2\geq 0$, we have the joint estimate
\begin{equation}
    \mathbb{P}\left(\sigma_{min}(A_n-\lambda_i I_n)\leq\delta_i n^{-1/2},\quad\text{ for each }i=1,2\right)\leq C\delta_1\delta_2+2e^{-cn^{\sigma/2}}.
\end{equation}

\end{theorem}

\begin{remark}
    In May 2024 the author of this paper uploaded to arXiv an unpublished preprint \cite{han2024small} which claims to prove Theorem \ref{Theorem1.1} and \ref{Theorem1.2} for a very restrictive class of entry distribution of $A_n$. Namely, \cite{han2024small} essentially requires that the entries of $A_n$ have a non-vanishing Gaussian component. In this paper, we completely remove this restriction, yet we will reuse many of the computations from \cite{han2024small}. Therefore, we merge the main body of the preprint \cite{han2024small} to the current manuscript, and we have done a lot of rewriting to correct various errors and to make the proof clearer and more precise.
\end{remark}

\begin{remark}(On the various parametric restrictions) We discuss the various restrictions on the parameters in Theorem \ref{Theorem1.1} and \ref{Theorem1.2}.

First, both theorems assume that $\lambda_1,\lambda_2$ lie in the bulk of the semicircle law. This stems from the use of local semicircle law in Section \ref{chap3chap3chap3}. Although local semicircle law also holds at the spectral edge \cite{tao2010random} \cite{erdHos2012rigidity}, the proof of Lemma \ref{lemma4.11} requires a non-vanishing density and thus does not extend to the edge case. Thus in the edge case the whole estimate deteriorates.

Second, in Theorem \ref{Theorem1.1} the assumption that $\zeta$ has a finite Log-Sobolev constant is technical, and is only designed for the use of Proposition \ref{proposition4.7}. As this is satisfied by most good probability distributions such as the Bernoulli distribution, no generality is lost.

The assumptions in Theorem \ref{Theorem1.2} deserve further discussion.\begin{enumerate}
    \item 
 The assumption $|\lambda_1-\lambda_2|\geq\Delta^{\sigma-\frac{1}{2}},\sigma>0$ is clearly necessary for our decoupling estimates. When $|\lambda_1-\lambda_2|\lesssim n^{-1/2}$, the local eigenvalue statistics around $\lambda_1$ and $\lambda_2$ should not be asymptotically independent and instead are highly correlated. \item When $|\lambda_1-\lambda_2|\lesssim n^{1/2}$, there are solid reasons to believe that the error term $e^{-cn^{\sigma/2}}$ in Theorem \ref{Theorem1.2}  cannot be improved to $e^{-cn}$ by the current technique, because there are typically only $n^{\sigma}$ eigenvalues between $\lambda_1$ and $\lambda_2$. This can be more clearly understood from Lemma \ref{lemma10.444} and Remark \ref{remarkweaks}. \item However, it is still reasonable to believe that we can improve the error from $e^{-cn^{\sigma/2}}$ to $e^{-cn^\sigma}$ for $\sigma\in(0,1)$. The latter error arises from the use of local semicircle law in Theorem \ref{theorem4.3}, which could be suboptimal in certain respects.
\end{enumerate}

Theorem \ref{Theorem1.2} raises an interesting question as to the exact extent of decoupling of local eigenvalue statistics around $\lambda_1$ and $\lambda_2$. That is, to what exponential scale on $\delta_1,\delta_2$ does the decoupling remain valid, when $n^{-1/2}\lesssim|\lambda_1-\lambda_2|\lesssim n^{1/2}$?

Finally, it appears that our method completely breaks down when $|\lambda_1-\lambda_2|\lesssim n^{-1/2}$, and very different methods might be needed to produce any effective estimate there.
\end{remark}

\begin{remark}
   A natural question is to generalize Theorem \ref{Theorem1.1} and \ref{Theorem1.2} to $d\geq 3$ locations $\lambda_1\cdots\lambda_d$. Some computations for $d\geq 3$ are presented in the preprint \cite{han2024small}, but we believe that the case $d\geq 3$ still requires considerable new ideas. While most parts of this paper can work for $d\geq 3$, the proof of Lemma \ref{smalllemma2.5bound} seems to work only for $d=2$. The constructions in Section \ref{verification1} also become overly complicated when $d\geq 3$.

\end{remark}

\begin{remark}(Comparison to the i.i.d. case) In a recent work \cite{han2025simplicity} the author considered the analogous problem of Theorem \ref{Theorem1.1} and \ref{Theorem1.2} for a random matrix with i.i.d. entries, and the proof involves an enhanced version of the program of Rudelson and Vershynin \cite{rudelson2008littlewood}. The main result can be stated as follows: if $A$ is an $n\times n$ random matrix with i.i.d. real subgaussian entries and $\lambda_1,\lambda_2\in\mathbb{R}$, then 
$$
\mathbb{P}(\sigma_{min}(A-\lambda_1I_n)\leq\epsilon,\sigma_{min}(A-\lambda_2I_n)\leq\epsilon)\leq \frac{C\sqrt{n}}{|\lambda_1-\lambda_2|}\epsilon^2+2e^{-cn},
$$ and if we suppose that the entries of $A$ have an i.i.d. real and complex component, then for any $\lambda_1,\lambda_2\in\mathbb{C}$ we have $$
\mathbb{P}(\sigma_{min}(A-\lambda_1I_n)\leq\epsilon,\sigma_{min}(A-\lambda_2I_n)\leq\epsilon)\leq \frac{Cn}{|\lambda_1-\lambda_2|^2}\epsilon^4+2e^{-cn}.
$$ Compared to Theorem \ref{Theorem1.1} and \ref{Theorem1.2}, these estimates do not capture the decoupling of correlations when $|\lambda_1-\lambda_2|\gg n^{-1/2}$ but has an interesting advantage that they are valid also when $|\lambda_1-\lambda_2|\ll n^{-1/2}$. This discrepancy arises from the different proof techniques, and we conjecture that (1) the type of bounds in \cite{han2025simplicity} can be proven for a random symmetric matrix here, and (2) the type of bounds in Theorem \ref{Theorem1.1}, \ref{Theorem1.2} can be proven for an i.i.d. matrix. This is left for future work.
    
\end{remark}

Using Theorem \ref{Theorem1.1}, \ref{Theorem1.2}, we can prove Conjecture \ref{conjecture1.6} for singular values of $A_n$ in a (close to one) fraction of the spectrum. Let $\sigma_1(A_n)\geq\sigma_2(A_n)\geq\cdots\sigma_n(A_n)$ be the singular values of $A_n$ and let $I\subset\mathbb{R}$ be an interval, then we define the minimal gap of singular values of $A_n$ in $I$ to be the quantity $\min_{i\in [n-1]}\sigma_i(A_n)-\sigma_{i+1}(A_n)\text{ for those }\sigma_i(A_n),\sigma_{i+1}(A_n)\in I.$

\begin{corollary}\label{theorem1.69}
    Let $A_n$ be a random symmetric matrix with entry distribution $\zeta$ having mean 0, variance 1 and a finite log-Sobolev constant. Then for any $\kappa>0$, with probability $1-e^{-cn}$, singular values of $A_n$ in $[\kappa n^{1/2},(2-\kappa)n^{1/2}]$ are distinct. Moreover, for any $\epsilon>0$,
    $$
\mathbb{P}(\text{Minimal gap of singular values of $A_n$ in $[\kappa n^{1/2},(2-\kappa)n^{1/2}]$}\leq \epsilon n^{-3/2})\leq C\epsilon+e^{-cn}
    $$ for some $C,c>0$ depending only on $\zeta$ and $\kappa$.

    More generally, fix some $\sigma\in(0,1]$ (and no longer assume that $\zeta$ satisfies LSI). Then we have the following estimate: for the random matrix $A_n$ and any $\epsilon>0$, \begin{equation}
    \mathbb{P}_{A_n}(\text{Minimal gap of singular values in $[\kappa n^{-1/2+\sigma},(2-\kappa)n^{1/2}]$}\leq \epsilon n^{-3/2})\leq C\epsilon+e^{-cn^{\sigma/2}}
    \end{equation} for some $C,c>0$ depending only on $\zeta$ and $\kappa$.
\end{corollary}

We conjecture that these estimates should also hold for singular values near the origin and at the spectral edge, so that Conjecture \ref{conjecture1.6} should be justified in full. The current restriction follows from the parametric restrictions imposed on Theorem \ref{Theorem1.1} and \ref{Theorem1.2}.

More general linear statistics can also be considered in this setting. For example, can $A_n$ have two eigenvalues where one is four times as large as the other one? Or, can one eigenvalue equal four times the other one plus six? The next corollary addresses these cases.

\begin{corollary}\label{theorem1.7} Let $A_n$ be a random symmetric matrix with entry distribution $\zeta$ of mean 0, variance 1 and having a subgaussian tail. Fix any $\kappa>0$ and $\Delta>0$. 

We say that two eigenvalues $x_1,x_2$ of $A$ are \textit{distant} if $|x_1-x_2|\geq\Delta n^{1/2}$. Let $I_\kappa:=[-(2-\kappa)\sqrt{n},(2-\kappa)\sqrt{n}]$. Suppose moreover that $\zeta$ has a finite log-Sobolev constant. Then for any ($n$-dependent) constants $a_1,a_2,D\in\mathbb{R}$ with $a_1=1,a_2=O(1)$, the following estimate holds for any $\epsilon\geq 0$:
 \begin{equation} \begin{aligned}   &\mathbb{P}(A_n \text{ has } \text{ two \textit{distant} eigenvalues } x_1,x_2\in I_\kappa \text{ satisfying } |\sum_{i=1}^2a_ix_i-D|\leq \epsilon n^{-3/2})\\&\leq C \epsilon+2e^{-cn}.\end{aligned}\end{equation}
 More generally, for any $\sigma\in(0,1)$, we say that two eigenvalues $x_1,x_2$ of $A$ are $\sigma$-\textit{distant} if $|x_1-x_2|\geq \Delta n^{\sigma-\frac{1}{2}}$. Then for $a_1=1,a_2=O(1)$ and $D\in\mathbb{R}$, we have for any $\epsilon>0$:
  \begin{equation} \begin{aligned}   &\mathbb{P}(A_n \text{ has two } 
  \sigma-\text{distant eigenvalues } x_1,x_2\in I_\kappa \text{ satisfying } |\sum_{i=1}^2a_ix_i-D|\leq \epsilon n^{-3/2})\\&\leq C\epsilon+e^{-cn^{\sigma/2}},\end{aligned}\end{equation}
  where the constants $C,c>0$ depend on $\zeta,\Delta,\kappa,\sigma$ and $|a_2|$.
\end{corollary}
The proof of both corollaries is elementary and is deferred to Section \ref{lastbooks}.

\subsection{The structure of this paper}This paper is an enhancement in two-point estimates of the one-point estimate in \cite{campos2024least}. The proof is very difficult, taking up over sixty pages and seems hard to simplify.
We feel the main novelty of this work lies in two respects:
\begin{enumerate}
    \item First, we refine the ``inversion of randomness'' technique in \cite{campos2025singularity} (where an earlier version of this idea came from \cite{tikhomirov2020singularity}), so that we can estimate the joint behavior of the arithmetic structure of random vectors associated with $A_n$ at two locations $s_1,s_2$. The actual quantity we are interested in is the arithmetic structure of $$(A_n-s_1I_n)^{-1}t_1w_0+(A_n-s_2I_n)^{-1}t_2w_0$$ for some $w_0\in\mathbb{R}^n$ and any $t_1,t_2\in\mathbb{R}$. This requirement of a two-location estimate significantly complicates the analysis, and we decompose the space of candidate unit vectors $(v_1,v_2)$ solving $(A_n-\lambda_i I_n)v_i=t_iw_0$ in a way depending on the inner product of (i.e., the angle between) $v_1,v_2$, exhausting all possibilities of inner product. To apply the ``inversion'' technique, we introduce threshold functions that measure the arithmetic structure of vector pairs $(v_1,v_2)$ with respect to a zeroed out matrix $M_n$, and the threshold function depends explicitly on the inner product of $v_1$ and $v_2$.
    \item Second, we find a new use of local semicircle law or the finite Log-Sobolev constant on the macroscopic structure of the spectrum, allowing us to decouple the eigenvalue events around two separated locations $|\lambda_1-\lambda_2|\gg n^{-1/2}$. This decoupling surprisingly works well, as compared to the case of i.i.d. matrix in \cite{han2025simplicity}.
\end{enumerate}

This paper thus naturally contains two main components: (I) a high probability verification that the eigenvectors and other random vectors associated to $A_n$ have good arithmetic properties (i.e., no rigid arithmetic structure); and (II) a decoupling step and the proof of two-location singular value estimates based on the verification component (I).

Component (I) is initiated in Section \ref{sectiontwos} and the main body of the work is performed in Section \ref{verification1} and Section \ref{verification2}. Component (II) begins in Section \ref{sectionfiveth} and proceeds to Section \ref{section9section9}. Actually, the proof of Theorem \ref{Theorem1.1} finishes in Section \ref{bootstrapping209} and the proof of Theorem \ref{Theorem1.2} finishes in Section \ref{section9section9}. The proof of a conditioned inverse Littlewood-Offord theorem, crucial for component (I), is deferred to Section \ref{whatchap23g>?}. The proof of the two corollaries are also deferred to Section \ref{whatchap23g>?}.

\subsection{Terminologies and conventions}
We outline here a set of notations that are used throughout the paper.  For a vector $v\in\mathbb{R}^n$, $\|v\|_2$ always denotes its Euclidean norm and $\|v\|_\infty:=\sup_{i\in[n]}|v_i|$ denotes its $\ell^\infty$ norm. For $v,w\in\mathbb{R}^n$, $\langle v,w\rangle$ denotes the standard Euclidean inner product $\langle v,w\rangle=\sum_{i=1}^nv_iw_i$.

Let $\zeta$ be a mean 0 variance 1 random variable. We define $$\|\zeta\|_{\psi_2}:=\sup_{p\geq 1}p^{-1/2}\mathbb{E}(|\zeta|^p)^\frac{1}{p}.$$ We say $\zeta$ has subgaussian tail if $\|\zeta\|_{\psi_2}<\infty$.

For a given $B>0$ we let $\Gamma_B$ be the collection of subgaussian random variables $\zeta$ such that $\|\zeta\|_{\psi_2}\leq B$.

For a given random variable $\zeta$ and $n\in\mathbb{N}$ we let $A_n\sim\operatorname{Sym}_n(\zeta)$ to denote a random symmetric matrix whose entries $(A_{ij})_{1\leq i\leq j\leq n}$ are i.i.d. copies of $\zeta$ and $A_{ij}=A_{ji}$.

Let $A=(A_{ij})$ be an $m\times n$ matrix for given $m,n\in\mathbb{N}$. Then $\|A\|=\|A\|_{op}:=\sup_{x\in\mathbb{S}^{n-1}}\|Ax\|_2$ denotes the operator norm of $A$ and $\|A\|_{HS}:=\sqrt{\sum_{i=1}^m\sum_{j=1}^n A_{ij}^2}$ denotes the Hilbert-Schmidt norm of $A$.
The minimal singular value of $A$ is denoted by $\sigma_{min}(A):=\inf_{x\in\mathbb{S}^{n-1}}\|Ax\|_2$.

We also need the following notations which are first introduced in \cite{rudelson2008littlewood}.
For a vector $v\in\mathbb{S}^{n-1}$ and two given constants $\delta,\rho\in(0,1)$, we say that $v$ is $(\delta,\rho)$-compressible if $v$ has Euclidean distance at most $\rho$ to a vector in $\mathbb{R}^n$ supported on at most $\delta n$ coordinates. We say $v$ is $(\delta,\rho)$-incompressible if it is not $(\delta,\rho)$-compressible. The set of $(\delta,\rho)$-incompressible vectors of $\mathbb{S}^{n-1}$ is denoted $\operatorname{Incomp}_n(\delta,\rho)$ and the set of $(\delta,\rho)$-compressible vectors of $\mathbb{S}^{n-1}$ is denoted $\operatorname{Comp}_n(\delta,\rho)$.

For fixed parameters $\alpha,\gamma\in(0,1)$ and a vector $v\in\mathbb{R}^n$, we define the essential least common denominator (LCD) $D_{\alpha,\gamma}(v)$ of $v$ as follows:
\begin{equation}
    D_{\alpha,\gamma}(v):=\inf\{\phi\geq 0:\|\phi v\|_\mathbb{T}\leq \min\{\gamma\|\phi v\|_2,\sqrt{\alpha n}\},
\end{equation} where $\|v\|_\mathbb{T}:=\min_{p\in\mathbb{Z}^n}\|v-p\|_2$.

We also define the Lévy concentration function: let $X$ be an $\mathbb{R}^n$-valued random variable, then we define 
$$
\mathcal{L}(X,t):=\sup_{w\in\mathbb{R}^n}\mathbb{P}(\|X-w\|_2\leq t).
$$

For a vector $v\in\mathbb{R}^n$ and $I\subset [n]$ a subset of coordinates, we use $v_I$ to denote the $\mathbb{R}^I$-valued vector which is the vector $v$ restricted to coordinates in $I$. For an integer $m\leq n$ we use $[m]$ to denote the integer interval $[1,m]\cap \mathbb{N}$ and for two integers $m_1\leq m_2$ we use $[m_1,m_2]$ to denote the integer interval $[m_1,m_2]\cap\mathbb{N}$. Therefore, $v_{[m_1,m_2]}$ denotes the restriction of $v$ to its coordinates indexed by $[m_1,m_2]$.

Throughout the paper, we use $a\lesssim b$ to mean $a\leq Cb$ for a constant $C>0$ depending only on the given parameters which will be clear from the context (these are usually the subgaussian moment $B$, the constants $\Delta$ and $\sigma$, and the Log-Sobolev constant of $\zeta$).

\section{A set of quasi-randomness conditions}\label{sectiontwos}

We outline a family of rather strong quasirandomness properties that should be satisfied by eigenvectors of the random matrix $A_n$ and other random vectors associated to $A_n$, with probability $1-\exp(-\Omega(n))$. The justification of these properties, to be carried out in Section \ref{verification1} and \ref{verification2}, is one of the most important components of this paper.

Fix two locations $\lambda_1,\lambda_2\in\mathbb{R}$. Let $A=A_n\sim\operatorname{Sym}_n(\zeta)$ and consider $X\sim\operatorname{Col}_n(\zeta)$ an independently distributed random vector. Define a family of four quasirandomness events on $A_n$
(which also depends implicitly on $\lambda_1,\lambda_2\in\mathbb{R}$) as follows. First, we define $\mathcal{E}_1$ via
\begin{equation}
    \mathcal{E}_1:=\{\|A_n\|_{op}\leq 4\sqrt{n}\}.
\end{equation}

Consider $X,X'\sim\operatorname{Col}_n(\zeta)$ two independent random vectors, and let $J\subset[n]$ be a $\mu$-independent subset in the sense that for each $j\in[n]$, we have that $j\in J$ holds independently with probability $\mu$. Then define a random vector $\tilde{X}:=X_J-X_J'$ in $\mathbb{R}^n$. We let $\mathbb{R}_*^2$ denote $\mathbb{R}^2\setminus\{0\}$ where $0$ is the zero element in $\mathbb{R}^2,$ and use the shorthand notation 
\begin{equation}\label{shorthandnotations}[\theta\cdot A^{-1}\widetilde{X}]:=\theta_1(A_n-\lambda_1 I_n)^{-1}\widetilde{X}+\theta_2(A_n-\lambda_2 I_n)^{-1}\widetilde{X}.
\end{equation}
Then we define $\mathcal{E}_2$ as the following event depending on $A_n$ and implicitly on $\lambda_1,\lambda_2$:
\begin{equation}
    \mathcal{E}_2:=\{\mathbb{P}_{\widetilde{X}}(\exists\theta\in\mathbb{R}_*^2,\|\theta\|_2\leq e^{c_*n}:\|[\theta\cdot A^{-1}\widetilde{X}]\|_2=1,[\theta\cdot A^{-1}\widetilde{X}]\in\operatorname{Comp}(\delta,\rho))\leq e^{-c_2n}\}
\end{equation}
for two constants $\delta,\rho\in(0,1)$ and for a $c_*>0$, all of which depend only on $B$.

The event $\mathcal{E}_3$ is defined as follows: for a certain choice of $\alpha,\gamma\in(0,1)$,
\begin{equation}
    \mathcal{E}_3:=\{D_{\alpha,\gamma}(v)\geq e^{c_3n}\text{ for every unit eigenvector } v\text{ of }A_n\}.
\end{equation}

Now we define the fourth event. For $\mu\in(0,1)$ we define the \textbf{subvector least common denominator} of two vectors $v,w\in\mathbb{R}^n$ as
$$
\hat{D}_{\alpha,\gamma,\mu}^{c_*}(v,w):=\min_{I\subset[n],|I|\geq (1-2\mu)n}\min_{\theta\in\mathbb{R}_*^2:\|\theta\|_2\leq e^{c_*n}}D_{\alpha,\gamma}(\frac{\theta_1v_I+\theta_2w_I}{\|\theta_1v_I+\theta_2w_I\|_2})\text{ if } \|\theta_1v_I+\theta_2w_I\|_2\geq 1.
$$  This definition differs from the standard definition of essential LCD of a vector pair \cite{rudelson2009smallest} in that we restrict the supremum of $\|\theta\|_2$ and we add a regularizing condition $\|\theta_1v_I+\theta_2w_I\|_2\geq 1$. By linearity, we only need to evaluate $D_{\alpha,\gamma}$ at those $\theta$ such that $\|\theta_1v_I+\theta_2w_I\|_2=1$ or those such that $\|\theta_1v_I+\theta_2w_I\|_2\in[1,2]$. 
If for every $\|\theta\|_2\leq e^{c_*n}$ we have $\|\theta_1v_I+\theta_2w_I\|_2\leq 1$, then we simply set $\hat{D}_{\alpha,\gamma,\nu}^{c_*}(v,w)=e^{c_*n}$.

For future use, we also define the subvector LCD for one vector $v\in\mathbb{R}^n$ via
$$
\hat{D}_{\alpha,\gamma,\mu}(v):=\min_{I\subset [n]:|I|\geq (1-2\mu)n}D_{\alpha,\gamma}(\frac{v_I}{\|v_I\|_2}).
$$

Now we can define $\mathcal{E}_4$ to be the event where $A_n$ satisfies
$$
\mathcal{E}_4:=\{\mathbb{P}_{\widetilde{X}}(\hat{D}_{\alpha,\gamma,\mu}^{c_*}((A_n-\lambda_1 I_n)^{-1}\widetilde{X},(A_n-\lambda_2 I_n)^{-1}\widetilde{X})\leq e^{c_4n})\leq e^{-c_4n}\}.
$$

The main quasi-randomness event we shall need is the following:
\begin{lemma}\label{mainquasirandomness1}
    For two distinct real numbers $\lambda_1,\lambda_2\in[-(2-\kappa)\sqrt{n},(2-\kappa)\sqrt{n}]$ denote by 
    $$
\mathcal{E}:=\mathcal{E}_1\cap\mathcal{E}_2\cap\mathcal{E}_3\cap\mathcal{E}_4.
    $$ For fixed $B>0,\zeta\in\Gamma_B$ and any sufficiently small $\alpha,\gamma,\mu\in(0,1)$,  we can find four constants $c_*,c_2,c_3,c_4\in(0,1)$ (with $c_4\leq c_*/2$) and two constants $(\delta,\rho)\in(0,1)$ such that 
    \begin{equation}
        \mathbb{P}_{A_n}(\mathcal{E}^c)\leq 2e^{-\Omega(n)}.
    \end{equation} The constants $c_*,c_2,c_3,c_4,\delta,\rho$ and the possible range of $\alpha,\gamma,\mu$ depend only on $B$ and $\kappa$.
\end{lemma}

An important ingredient in Lemma \ref{mainquasirandomness1} is that we do not assume any distance lower bound on $|\lambda_1-\lambda_2|$ as in the statement of Theorem \ref{Theorem1.1}, \ref{Theorem1.2}, so the conclusion remains valid even when $|\lambda_1-\lambda_2|\ll n^{-1/2}$. In addition, we introduce an auxiliary threshold $c_*>0$ to exclude troublesome behaviors at smaller scales.

Among the four events, the most challenging one is $\mathcal{E}_4$, whose verification will take up the next two sections. The event $\mathcal{E}_3$ is also very difficult to justify, but this is already proven in \cite{campos2024least} as we only consider an eigenvalue instead of an eigenvalue pair. The event $\mathcal{E}_2$ will be analyzed in the following, and the event $\mathcal{E}_1$ is immediate from the literature.

\subsection{On the compressible vectors}
We first note that the event $\mathcal{E}_1$, i.e. the operator norm, can be very well controlled:
\begin{lemma}\label{operatorbound}
    Given $B>0$, $\zeta\in\Gamma_B$ and $A_n\sim\operatorname{Sym}_n(\xi)$. 
Then for any $t>0$  we have
\begin{equation}
\mathbb{P}\left(\|A_n\|_{op}\geq (3+t)\sqrt{n}\right)\leq e^{-ct^{3/2}n}
\end{equation}
for some constant $c>0$ depending only on $B$.

\end{lemma}
This lemma can be deduced from the computations in \cite{feldheim2010universality}. We also need a basic bound on anticoncentration for a row matrix product: (see \cite{vershynin2014invertibility}, Proposition 4.1)
\begin{lemma}\label{c4.1}
    Fix $B>0,\zeta\in\Gamma_B$ and $A_n\sim\operatorname{Sym}_n(\zeta)$. Then for any $x\in\mathbb{S}^{n-1}$, one has
    $$
\mathcal{L}(A_nx,c_{\ref{c4.1}}n^{1/2})\leq 2e^{-c_{\ref{c4.1}}n}
    $$ where $c_{\ref{c4.1}}>0$ depends only on $B$.
\end{lemma}

We quote that the solution $v$ to $(A-s)v=tw$ is incompressible: ( \cite{campos2025singularity}, Lemma 3.2)
\begin{lemma}\label{incomptwoone}
    Let $B>0,\zeta\in\Gamma_B$ and $A_n\sim\operatorname{Sym}_n(\zeta)$. Then we can find three constants $\rho,\delta,c\in(0,1)$ depending only on $B$ so that
    \begin{equation}
        \sup_{w\in\mathbb{S}^{n-1}}\mathbb{P}(\exists x\in \operatorname{Comp}(\delta,\rho),\exists t,s\in\mathbb{R}:(A_n-s I_n)x=tw)\leq 2e^{-cn}.
    \end{equation}
\end{lemma}

To deal with the event $\mathcal{E}_2$, we need to prove the following, significantly strengthened version of Lemma \ref{incomptwoone} for vector pairs. 
\begin{lemma}\label{smalllemma2.5bound}
    Let $B>0,\zeta\in\Gamma_B$ and $A_n\sim\operatorname{Sym}_n(\zeta)$. Let $\lambda_1,\lambda_2\in[-4\sqrt{n},4\sqrt{n}]$. Fix any $w\in\mathbb{S}^{n-1}$. Then we can find $\delta,\rho,c,c_*\in(0,1)$ depending only on $B$ such that, with probability $1-e^{-cn}$, the following is true: 
    
   For
    any $\theta=(\theta_1,\theta_2)\in\mathbb{R}^2$ with $\|\theta\|_2\leq e^{c_*n}$  such that $$v:=\theta_1(A_n-\lambda_1 I_n)^{-1}w+\theta_2(A_n-\lambda_2I_n)^{-1}w$$ is a vector in $\mathbb{R}^n$ of unit norm, we must have that $v$ is $(\delta,\rho)$-incompressible.
\end{lemma} The assumption on the maximum of $\|\theta\|_2$ is technical but sufficient, as in our future applications of Esseen's Lemma, we will only integrate over those $\|\theta\|_2\leq e^{c_*n}$.
\begin{proof}
This vector $v$ solves the quadratic equation $$(A_n-\lambda_1 I_n)(A_n-\lambda_2 I_n)v=(\theta_1+\theta_2)A_nw-(\theta_1\lambda_2+\theta_2\lambda_1)w.$$    

We condition on the event $\|A_n\|\leq 4\sqrt{n}$ by Lemma \ref{operatorbound}. Let $\mathcal{N}$ be an $n^{-3/2}$-net of $[-e^{c_*n},e^{c_*n}]$ for some $c_*>0$ to be fixed at the end of proof. For given $\theta_1,\theta_2\in\mathbb{R}$, $\|\theta\|_2\leq e^{c_*n}$ we find $\theta_1',\theta_2'\in\mathcal{N}$ with $|\theta_i-\theta_i'|\leq n^{-3/2}$, $i=1,2$. Then we have
$$
\|(A_n-\lambda_1 I_n)(A_n-\lambda_2 I_n)v-(\theta_1'+\theta_2')A_nw+(\theta_1'\lambda_2+\theta_2'\lambda_1)w\|_2\leq 8n^{-1}.
$$ Now we fix the value of $\theta_1',\theta_2'$ throughout the rest of the proof. Assume that $v$ is $\rho$-close to a $\delta n$-sparse vector $v_0$ with $\|v_0\|\geq \frac{1}{2}$ (which surely holds when $\rho<\frac{1}{2}$), then using the bound on $\|A_n\|$ we get that 
\begin{equation}\label{twopointinverts}
\|(A_n-\lambda_1 I_n)(A_n-\lambda_2 I_n)v_0-(\theta_1'+\theta_2')A_nw+(\theta_1'\lambda_2+\theta_2'\lambda_1)w\|_2\leq 65\rho n.
\end{equation} Up to conjugating $A_n$ by a permutation matrix, we assume without loss of generality that the support of $v_0$ is precisely the last $\delta n$ coordinates. We decompose $A_n$ into a block form as $A_n=\begin{pmatrix}
    E&F\\G&H
\end{pmatrix}$ where $H$ has size $(\delta n)\times (\delta n)$.  Then $(A_n-\lambda_2 I_n)v_0$ is only measurable with respect to $F$ and $H$ because $v_0$ is only supported on the last $\delta n$ coordinates. We assume $\delta>0$ is small enough, then we have \begin{equation}\label{rowleastbounds}\|Fv_0-(\theta_1'+\theta_2')w_{[1,n-\delta n]}\|_2\geq \hat{c}n^{1/2}\end{equation} with probability $1-e^{c_5n}$ where $c_5,\hat{c}>0$ depend only on $B$. This is a standard claim and can be proven via anticoncentration on each row and the tensorization lemma, see \cite{rudelson2008littlewood}, Proposition 2.5.

Now we project the inequality \eqref{twopointinverts} onto the first $n-\delta n$ coordinates and get 
$$\|
(E-\lambda_1)Fv_o+F(H-\lambda_2)v_0-(\theta_1'+\theta_2')([E,F]\cdot w)+(\theta_1'\lambda_2+\theta_2'\lambda_1)w_{[n-\delta n]}\|_2\leq65\rho n,
$$ where $\begin{bmatrix}
    E,&F
\end{bmatrix}$ is the first $n-\delta n$ row of $A_n$.
Conditioning on a realization of $F,G,H$ and using only the randomness in $E$, the probability for this to happen is bounded by the following (where we use $\mathcal{L}_E$ to denote the Lévy concentration function using randomness only from $E$, conditioning on $F,G,H$)
\begin{equation}\label{fromthecomputation}
\mathcal{L}_E(E(Fv_0-(\theta_1'+\theta_2')w_{[1,n-\delta n]}),65\rho n).
\end{equation} Using the estimate \eqref{rowleastbounds} and Lemma
\ref{c4.1}, we see that when $\rho$ is sufficiently small relative to $B$, this probability is at most $2e^{-c_{\ref{c4.1}}n}.$ We let $c_6=\min(c_5,c_{\ref{c4.1}})$.

The cardinality of a $\rho$-net for such $\delta n$ sparse vectors $v_0$ is at most 
$
\binom{n}{\delta n}(\frac{1}{\rho})^{\delta n}.
$ Multiplying this by the individual probability $4e^{-c_6n}$ for each $v_0$ from the computation \eqref{fromthecomputation}, and computing the cardinality of $(\theta_1,\theta_2)\in \mathcal{N}$, the probability in question is bounded by 
$$
e^{-\Omega(n)}+4e^{-c_6n}\cdot \binom{n}{\delta n}(\frac{1}{\rho})^{\delta n}\cdot n^6e^{2c_*n},
$$
where the first term is from the conditioning of $\|A_n\|$. Finally, choosing $\delta,\rho,c_*\in(0,1)$ sufficiently small with respect to $c_6$, the resulting estimate is exponential small in $n$.
\end{proof}

\subsection{The main quasirandomness theorem}

We will reduce the proof of Lemma \ref{mainquasirandomness1}  to the following main quasirandomness theorem of this paper. Let $c_\Sigma>0$ be a constant to be fixed later, and we define a set of structured directions on the product sphere 
$$\begin{aligned}
\Sigma=\Sigma_{\alpha,\gamma,\mu}:=\{&(v_1,v_2)\in\mathbb{S}^{n-1}\times\mathbb{S}^{n-1}:\|\operatorname{Proj}_{v_1^\perp}v_2\|_2\geq e^{-c_*n/2},\\&\hat{D}_{\alpha,\gamma,\mu}(v_1,v_2)\leq e^{c_\Sigma n}\}
\end{aligned}$$ for a given constant $c_*>0$. The notation $\operatorname{Proj}_{v_1^
\perp}v_2$ is the projection of $v_2$ onto the orthogonal complement of $v_1$ in $\mathbb{R}^d$, and this condition is made for non-degeneracy so that $v_1,v_2$ are not colinear.
Now for $A\sim\operatorname{Sym}_n(\zeta)$ with $\zeta\in\Gamma$ and a vector $w_0\in\mathbb{R}^n$ we define 
\begin{equation}\label{masterquasi}\begin{aligned}&q_n(w_0)=q_n(w_0,\alpha,\gamma,\mu)\\&:=\mathbb{P}_A\left(\exists (v_1,v_2)\in\Sigma,\exists s_1,s_2\in[-4\sqrt{n},4\sqrt{n}],t_1,t_2\in[-8\sqrt{n},8\sqrt{n}]:\begin{pmatrix}
(A-s_1)v_1=t_1w_0\\(A-s_2)v_2=t_2w_0\end{pmatrix}\right)\end{aligned}
\end{equation}
and we define 
$$
q_n=q_n(\alpha,\gamma,\mu):=\sup_{w_0\in\mathbb{S}^{n-1}}q_n(w_0).
$$
The main quasirandomness result of this paper is stated as follows:
\begin{theorem}\label{mainquasiramdonmess2}
    Fix $B>0$ and $\zeta\in\Gamma_B$. We can find $\alpha,\gamma,\mu,c_\Sigma,c\in(0,1)$ and $c_*\in(0,1)$ depending only on $B$ so that we have
    $$
q_n(\alpha,\gamma,\mu)\leq 2e^{-cn}.
    $$
\end{theorem}
Before proving Lemma \ref{mainquasirandomness1} conditioned on Theorem \ref{mainquasiramdonmess2}, we introduce a few facts and notations. The first fact is about properties of an incompressible vector from \cite{rudelson2008littlewood}.
\begin{fact}\label{fact5.67}
    For any $\rho,\delta\in(0,1)$ there exists a constant $c_{\rho,\delta}\in(0,1)$ such that any $v\in \operatorname{Incomp}(\delta,\rho)$ satisfies that $c_{\rho,\delta}n^{-1/2}\leq |v_i|\leq c_{\rho,\delta}^{-1}n^{-1/2}$ for at least $c_{\rho,\delta}n$ indices of $i\in[n]$.
\end{fact}
We also need a notation for angles and orthogonal projection.

\begin{Notation}
    Let $v,w\in\mathbb{R}^n$ be nonzero vectors. Let $\operatorname{Proj}_{w^\perp}$ be the orthogonal projection of $\mathbb{R}^n$ onto the orthogonal complement of $w$. Then we denote by 
    $$
\sin(v,w):=\frac{\|\operatorname{Proj}_{w^
\perp}v\|_2}{\|v\|_2}.  
    $$ where one readily checks that $\sin(v,w)=\sin(w,v)$.
\end{Notation}

We also need the one-location estimate for least singular values:
\begin{Proposition}\label{proposition6.666} Fix $B>0$, $\zeta\in\Gamma_B$ and consider the random matrix $A_n\sim\operatorname{Sym}_n(\zeta)$. Fix any $\kappa>0$. Then, uniformly for all $\lambda\in[-(2-\kappa)\sqrt{n},(2-\kappa)\sqrt{n}]$, for any $\epsilon>0$,
    \begin{equation}
    \mathbb{P}(    \sigma_{min}(A_n-\lambda I_n)\leq\epsilon n^{-1/2})\leq c\epsilon+e^{-cn},
    \end{equation} where the constant $c>0$ depends only on $B$ and $\kappa$.
\end{Proposition}
The $\lambda=0$ case of this proposition is the main result of \cite{campos2024least}. Going through the proof, one can check that most parts of the proof generalize straightforwardly to any $\lambda\in\mathbb{R}$ except the part in \cite{campos2024least}, Section 8, where local semicircle law is used. As the estimates in the local semicircle law holds uniformly in the bulk of the spectrum (see Section \ref{chap3chap3chap3} where these local law estimates are stated), this part of the proof generalizes to all $\lambda\in[-(2-\kappa)\sqrt{n},(2-\kappa)\sqrt{n}]$. Therefore we omit the proof of this straightforward generalization here.

We now prove Lemma \ref{mainquasirandomness1} assuming the validity of Theorem \ref{mainquasiramdonmess2}.
\begin{proof}[\proofname\ of Lemma \ref{mainquasirandomness1}]
    For the event $\mathcal{E}_1$, we use Lemma \ref{operatorbound} so that $\mathbb{P}(\mathcal{E}_1^c)\leq 2e^{-cn}$.

    For the event $\mathcal{E}_2$, by Lemma \ref{smalllemma2.5bound}, there exists $c,c_*>0$ and $\rho,\delta\in(0,1)$ such that for any $u\in\mathbb{S}^{n-1}$,
    $$
\mathbb{P}_A(\exists\theta\in\mathbb{R}_*^2,\|\theta\|_2\leq e^{c_*n}:\| [\theta\cdot A^{-1}u]\|_2=1,[\theta\cdot A^{-1}u]\in\operatorname{Comp}(\delta,\rho))\leq e^{-cn}.
    $$
    By Talagrand's inequality, we can check that \begin{equation}\label{talagrands}\mathbb{P}(\|\widetilde{X}\|_2\geq 10\sqrt{\mu n} \text{ or }\|\widetilde{X}\|_2\leq \sqrt{\mu n}/10)=e^{-\Omega(\mu n)}.\end{equation}On the complement, i.e., when $\sqrt{\mu n}/10\leq\|\widetilde{X}\|_2\leq 10\sqrt{\mu n}$, we have by Markov's inequality that (for a suitably modified constant $c_*>0$ depending on $\mu $)
    $$\begin{aligned}
\mathbb{P}_A&(\mathbb{P}_{\widetilde{X}}(\exists\theta\in\mathbb{R}_*^2,\|\theta\|\leq e^{c_*n}:\|[\theta\cdot A^{-1}\widetilde{X}]\|_2=1,\\&\left.[\theta\cdot A^{-1}\widetilde{X}]\in\operatorname{Comp}(\delta,\rho)\geq e^{-cn/2}\right| \|\widetilde{X}\|_2\in[\sqrt{\mu n}/10,10\sqrt{\mu n}]))\leq e^{-cn/2}.
    \end{aligned}$$
This is easy to check since, if $\|[\theta\cdot A^{-1}\widetilde{X}]\|_2=1$, then $\|[(\theta\cdot \|\widetilde{X}\|_2)\cdot A^{-1}\frac{\widetilde{X}}{\|\widetilde{X}\|_2}]\|_2=1$ where $\frac{\widetilde{X}}{\|\widetilde{X}\|_2}\in\mathbb{S}^{n-1}$. The estimate in Lemma \ref{smalllemma2.5bound} holds for $\|\theta\|\widetilde{X}\|_2\|_2\leq e^{c_*n}$ for the $c_*$ in Lemma \ref{smalllemma2.5bound}, so it holds for $\|\theta\|_2\leq e^{c_*n}/(10\sqrt{\mu n})$ here. Finally it suffices to modify the value of $c_*$ to a value $c_*'$ so that $e^{c_*n}/(10\sqrt{\mu n})\geq e^{c_*'n}$ for $n$ large. We again denote $c_*'$ by $c_*$.
    
Combining this with estimate \eqref{talagrands}, we conclude that $\mathbb{P}(\mathcal{E}_2^c)\leq 2e^{-\Omega(n)}$.

    For the event $\mathcal{E}_3$, notice that this is exactly the same event $\mathcal{E}_3$ as in \cite{campos2024least}, Lemma 4.1. Then we use the cited result to get $\mathbb{P}(\mathcal{E}_3^c)\leq e^{-\Omega(n)}$.
    
    For the event $\mathcal{E}_4$, let the value of $c_*$ specified in Theorem \ref{mainquasiramdonmess2} be denoted by ${c_*}_{\ref{mainquasiramdonmess2}}$. Note that a difference between $\mathcal{E}_4$ and $q_n(w_0)$ is that in the latter we require that $\|\operatorname{Proj}_{v_1^\perp}v_2\|_2\geq e^{-{c_*}_{\ref{mainquasiramdonmess2}}n/2}$. If this is not the case, that is for $w=\widetilde{X}/\|\widetilde{X}\|_2\in\mathbb{S}^{n-1}$ we have \begin{equation}\label{projnotgood}\|\operatorname{Proj}_{((A_n-\lambda_1 I_n)^{-1}w)^\perp}(A_n-\lambda_2 I_n)^{-1}w\|_2\leq e^{-{c_*}_{\ref{mainquasiramdonmess2}}n/2}\|(A_n-\lambda_2 I_n)^{-1}w\|_2,\end{equation} we need to use a very different strategy by trying to use an estimate for one single location. We condition on an event $\Omega_{10}$ that $\|(A_n-\lambda_i I_n)^{-1}\|_{op}\leq e^{c_{10}n}$ for both $i=1,2$ and for some $c_{10}\in(0,{c_*}_{\ref{mainquasiramdonmess2}}/100)$, and that $\Omega_{10}$ holds with probability $1-e^{-\Omega(n)}$. This is guaranteed by the one-location (optimal) singular value estimate of $A_n$ as stated in Proposition \ref{proposition6.666}, which is an adaptation of the main result of \cite{campos2024least}. Conditioning on the event that $\|\widetilde{X}\|_2\in[\sqrt{\mu n}/10,10\sqrt{\mu n}]$, then for any $\theta\in\mathbb{R}^2$ with $\|\theta\|_2\leq e^{c_{10}n}$ we can find some $\theta_\tau\in\mathbb{R}$ such that 
    \begin{equation}\label{goodsoevers}
\|\theta_1(A_n-\lambda_1 I_n)^{-1}\widetilde{X}+\theta_2(A_n-\lambda_2 I_n)^{-1}\widetilde{X}-\theta_\tau (A_n-\lambda_1 I_n)^{-1}\widetilde{X}\|_2\leq e^{-{c_*}_{\ref{mainquasiramdonmess2}}n/4}. 
    \end{equation} This follows from combining three facts: the upper bound on $\|(A_n-\lambda_2 I_n)^{-1}\widetilde{X}\|_2$  as a consequence of bounds on $\|\widetilde{X}\|_2$ and $\|(A_n-\lambda_2 I_n)^{-1}\|_{op}$, the upper bound on $\|\theta\|_2$, and our assumption \eqref{projnotgood}. Then if  we further assume that $\|\theta_1(A_n-\lambda_1 I_n)^{-1}\widetilde{X}+\theta_2(A_n-\lambda_2 I_n)^{-1}\widetilde{X}\|_2\in[\frac{2}{3},\frac{4}{3}]$, then $\|\theta_\tau(A_n-\lambda_1 I_n)^{-1}\widetilde{X}\|_2\in[\frac{1}{2},\frac{3}{2}]$ and thus it is not hard to check that the proof of $\mathcal{E}_4$ for s small $c_4\in(0,c_{10})$ in this case \eqref{projnotgood} would follow if we can prove the following estimate:   $\hat{D}_{\alpha/2,\gamma/2,\mu}((A_n-\lambda_1 I_n)^{-1}\widetilde{X})\geq e^{c_4n}$, by stability of the essential LCD. Indeed, on the event $\mathcal{E}_2$, assume $\mu$ is small enough with respect to $\delta,\rho$, then by Fact \ref{fact5.67} we have that for some $c_{\rho,\delta}\in(0,1)$ and for any $|I|\geq (1-\mu )n$, $\|(\theta_\tau(A_n-\lambda_1 I_n)^{-1}\widetilde{X})_I\|_2\geq c_{\rho,\delta}^2/2$, and $\|(\theta_1(A_n-\lambda_1 I_n)^{-1}\widetilde{X}+\theta_2(A_n-\lambda_2 I_n)^{-1}\widetilde{X})_I\|_2\geq c_{\rho,\delta}^2/2$, and we have the following estimate that can be derived from \eqref{goodsoevers}:
    $$
\|\frac{(\sum_{i=1}^2\theta_i(A_n-\lambda_i I_n)^{-1}\widetilde{X})_I}{\|(\sum_{i=1}^2(\theta_i(A_n-\lambda_i I_n)^{-1}\widetilde{X})_I\|_2}-\frac{(\theta_\tau (A_n-\lambda_1 I_n)^{-1}\widetilde{X})_I}{\|(\theta_\tau (A_n-\lambda_1 I_n)^{-1}\widetilde{X})_I\|_2}\|_2=O(e^{-{c_*}_{\ref{mainquasiramdonmess2}}n/4}),
    $$ by the elementary fact that $|\|u\|_\mathbb{T}-\|v\|_\mathbb{T}|\leq \|u-v\|_2$ which follow from triangle inequality.
    The claimed estimate on $\hat{D}_{\alpha/2,\gamma/2,\mu}((A_n-\lambda_1 I_n)^{-1}\widetilde{X})\geq e^{c_4n}$, involving only one location, is established in \cite{campos2024least}, Lemma 4.1 on an event $\Omega_{11}$ that holds with probability $1-e^{-\Omega(n)}$ (which is also denoted $\mathcal{E}_4$ in \cite{campos2024least}, and we actually need a modified version of the cited result that considers a scalar shift $A_n\mapsto A_n-\lambda_1 I_n$, but this modification is straightforward).

Then we consider the more generic case, i.e. the event $\mathcal{E}_4$ on which \eqref{projnotgood} does not hold. This is where we can apply Theorem \ref{mainquasiramdonmess2}. Then Let $\Omega^\perp$ denote the event where \eqref{projnotgood} does not hold, we have
$$
\mathbb{P}_A(\hat{D}_{\alpha,\gamma,\mu}^{c_{10}}((A_n-\lambda_1 I_n)^{-1}\widetilde{X},(A_n-\lambda_2 I_n)^{-1}\widetilde{X}))\leq e^{c_4n},\Omega^\perp)\leq q_n(\widetilde{X}/\|\widetilde{X}\|_2)\leq e^{-\Omega(n)}, 
$$ whenever $c_4\leq c_\Sigma$, by Theorem \ref{mainquasiramdonmess2}.  Then we have, writing $v_1=(A_n-\lambda_1 I_n)^{-1}\widetilde{X}$, $v_2=(A_n-\lambda_2 I_n)^{-1}\widetilde{X}$ and applying Markov's inequality, the upper bound
$$\begin{aligned}
\mathbb{P}(\mathcal{E}_4^c)&\leq e^{-\Omega(n)}+\mathbb{P}_A(\mathbb{P}_X(D_{\alpha,\gamma,\mu}^{c_{10}}(v_1,v_2)\leq e^{c_4n},\Omega^\perp)\geq e^{-c_4n})\\&\leq e^{-\Omega(n)} + e^{c_4n}\mathbb{E}_{\widetilde{X}}\mathbb{P}_A(D_{\alpha,\gamma,\mu}^{c_{10}}(v_1,v_2)\leq e^{c_4n},\Omega^\perp)=e^{-\Omega(n)},\end{aligned}
$$ where in the last inequality we set $c_4>0$ to be sufficiently small since $q_n(\widetilde{X}/\|\widetilde{X}\|_2)\leq e^{-\Omega(n)}$ is independent of $c_4$.
Finally we take $c_*$ in the definition of $\mathcal{E}_2$, $\mathcal{E}_4$ to be the constant $c_{10}$ (which reduces the value of $c_*$ in $\mathcal{E}_2$). The first term in the second line of the inequality comes from the preceding estimates when \eqref{projnotgood} holds, with one contribution from the event $\Omega_{10}^c$ and the other from the event $\Omega_{11}^c$. We shall always take $c_4\leq c_*/2$. 
    
    Therefore we have $\mathbb{P}(\mathcal{E}_1^c\cup\mathcal{E}_2^c\cup\mathcal{E}_3^c\cup\mathcal{E}_4^c)\leq 2e^{-\Omega(n)}$.
\end{proof}

\subsection{On the shape of vectors}

We need a better description for the size of coordinates of a pair of vectors $v_1,v_2$ satisfying \eqref{masterquasi}.

First, we take an orthogonal projection of $v_2$ onto $v_1$ and its orthogonal complement, and write $v_2=cv_1+\epsilon_1v_3$ for some $c,\epsilon_1\in\mathbb{R}$, some $v_3\in\mathbb{S}^{n-1}$ with $\langle v_1,v_3\rangle=0$. Then $v_3$ satisfies $v_3=\epsilon_1^{-1}v_2-c\epsilon_1^{-1}v_1=\epsilon_1^{-1}(A_n-s_2I_n)^{-1}t_2w-c\epsilon_1^{-1}(A_n-s_1I_n)^{-1}t_1w$. By our assumption on the event $\mathcal{E}_4$, we have $|\epsilon_1|\geq e^{-c_*n/2}$. On the event $\mathcal{E}_2$, we have that both $v_1$ and $v_3$ are $(\delta,\rho)$-incompressible. Therefore, we now have two orthogonal unit vectors $v_1$ and $v_3$ that are both $(\delta,\rho)$-incompressible. For future applications we need to restrict the vectors to a small fraction of coordinates: we need to find, for any sufficiently small $c_0>0$, a set of coordinates of cardinality $c_0^2n/4$ on which the restriction of $v_1,v_3$ are not close to colinear and the coordinates of both $v_1,v_3$ are within a good range.

We will mainly use incompressibility to achieve this goal. 
In addition to Fact \ref{fact5.67}, the following fact is also obvious: \begin{fact}\label{fact408}for any $c_0>0$ and any unit vector $v\in\mathbb{R}^n$, there are at most $c_0^2n/96$ coordinates $i\in[n]$ such that $|v_i|\geq \sqrt{\frac{96}{c_0^2n}}$.\end{fact} 

Let $v_1,v_3$ be orthogonal unit vectors that are both $(\delta,\rho)$-incompressible. For any $c_0\leq c_{\rho,\delta}/4$, by Fact \ref{fact5.67} and \ref{fact408} we can find subsets $D_1,D_2\subset[n]$ with $|D_1|=|D_2|=c_0^2n/12$ such that for any $i\in D_1,j\in D_2$ we have $c_{\rho,\delta}^{-1}n^{-1/2}\geq |(v_1)_i|,|(v_3)_j|\geq c_{\rho,\delta}n^{-1/2}$ and that for $i\in D_1\cup D_2$ we have $|(v_1)_i|,|(v_3)_i|\leq \sqrt{\frac{96}{c_0^2n}}$. However, there is a problem that $(v_1)_{D_1\cup D_2}$ and $(v_3)_{D_1\cup D_2}$ may have a very large overlap or even become colinear. We claim that this can be suppressed by adding some new coordinates to $D_1,D_2$.

Suppose that, for some $x\in\mathbb{R}$, $y\in(0,c_{\rho,\delta}^2c_0^2/64)$ we have $\|(v_1-xv_3)_{D_1\cup D_2}\|_2\leq y$. It is not hard to check that such $x$ must satisfy $|x|=O(1/c_{\rho,\delta}\sqrt{c_0^2/12})$ since, we must have $c_{\rho,\delta}\sqrt{c_0^2/12}\leq \|(v_i)_{D_1\cup D_2}\|_2\leq 1$ 
for both $i=1,3$, and then we use triangle inequality to get $c_{\rho,\delta}\sqrt{c_0^2/12}\cdot x\leq \|(xv_3)_{D_1\cup D_2}\|_2\leq c_{\rho,\delta}^2c_0^2/64+1$, which implies an upper bound for $|x|$. Then for this fixed $c_0$ and $c_{\rho,\delta}$, we can check that on the event $\mathcal{E}_2$, the vector $(v_1-xv_3)/\|v_1-xv_3\|_2$ is also $(\delta,\rho)$-incompressible, since it is a linear combination of $v_1,v_2$ of coefficients bounded by $O(e^{c_*n})$. Then we must be able to find some $D_3\subset[n]\setminus(D_1\cup D_2)$ with $\|(v_1-xv_3)_{D_3}\|_2\geq \sqrt{1+|x|^2}c_{\rho,\delta}\sqrt{|D_3|/n}$ where $|D_3|=c_0^2n/12$ and such that $|(v_1)_j|,|(v_3)_j|\leq\sqrt{\frac{96}{c_0^2n}}$ for all $j\in D_3$. The first condition is guaranteed by Fact \ref{fact5.67} and the second condition is guaranteed by Fact \ref{fact408} via removing very few vertices.

We now claim that $(v_1)_{D_1\cup D_2\cup D_3}$ and $(v_3)_{D_1\cup D_2\cup D_3}$ are not colinear.

\begin{fact}\label{poisson424}
    For any $z\in\mathbb{R}$ we have that \begin{equation}\label{eqsgageag}\|(v_1-zv_3)_{D_1\cup D_2\cup D_3}\|_2\geq c_{\rho,\delta}^2c_0^2/64.\end{equation} Thus \begin{equation}\label{sinefunction}
|\sin( (v_1)_{D_1\cup D_2\cup D_3},(v_3)_{D_1\cup D_2\cup D_3})
|\geq c_{\rho,\delta}^2c_0^2/64.    \end{equation}
\end{fact}
\begin{proof}
In the first case (A), when no such $x$ exists as stated in the previous paragraph, then for any $z\in\mathbb{R}$ we have $\|(v_1-zv_3)_{D_1\cup D_2\cup D_3}\|_2\geq \|(v_1-zv_3)_{D_1\cup D_2}\|_2\geq c_{\rho,\delta}^2c_0^2/64$. Then the first claim directly follows. In the second case $(B)$, we can find some $x\in\mathbb{R}$ such that $\|(v_1-xv_3)_{D_1\cup D_2\cup D_3}\|_2\leq c_{\rho,\delta}^2c_0^2/64$. Then we consider two cases (B1) when $|z-x|\leq c_{\rho,\delta}\sqrt{c_0^2/32}/2$ and (B2) when $|z-x|\geq c_{\rho,\delta}\sqrt{c_0^2/32}/2.$  In subclass (B1), we deduce via triangle inequality that $\|(v_1-zv_3)_{D_3}\|_2\geq c_{\rho,\delta}\sqrt{|D_3|/n}-|z-x|\|(v_3)_{D_3}\|_2\geq c_{\rho,\delta}\sqrt{|D_3|/n}-c_{\rho,\delta}c_0/8\geq c_{\rho,\delta}^2c_0^2/64$, where we use the trivial upper bound $\|(v_3)_{D_3}\|_2\leq 1.$ In subclass $(B2)$, we deduce via triangle inequality that $\|(v_1-zv_3)_{D_1\cup D_2}\|_2\geq |z-x|\|(v_3)_{D_1\cup D_2}\|_2-\|(v_1-xv_3)_{D_1\cup D_2}\|_2\geq c_{\rho,\delta}^2c_0^2/32-c_{\rho,\delta}^2c_0^2/64\geq c_{\rho,\delta}^2c_0^2/64$. This justifies \eqref{eqsgageag} in all possible cases.

The estimate \eqref{eqsgageag} implies that the projection of $(v_1)_{D_1\cup D_2\cup D_3}$ onto the subspace orthogonal to $(v_3)_{D_1\cup D_2\cup D_3}$ satisfies that  $$P_{((v_3) _{D_1\cup D_2\cup D_3})^\perp}(v_1)_{D_1\cup D_2\cup D_3}\geq c_{\rho,\delta}^2c_0^2/64\geq c_{\rho,\delta}^2c_0^2 \|(v_1)_{D_1\cup D_2\cup D_3}\|_2/64.$$ This completes the proof by our definition of the $\sin$ function.
\end{proof}

The estimate in \eqref{sinefunction} deteriorates significantly when $c_0$ is very small, which leads to an exponentially growing factor when applying inverse Littlewood-Offord type inequalities. Fortunately, this will not ruin our estimate eventually.

Therefore, when proving Theorem \ref{mainquasiramdonmess2}, we only need to consider vector pairs $(v_1,v_2)$ satisfying the aforementioned properties:

\begin{corollary}\label{corollary2.678} Fix $B>0,\zeta\in\Gamma_B$ and $A_n\sim\operatorname{sym}_n(\zeta)$. Let $c_*>0$ be the constant given in Lemma \ref{smalllemma2.5bound}. Then we can find constants $c_{\rho,\delta},\kappa_0\in(0,1)$ and $\kappa_1>\kappa_0>0$ depending only on $B$, such that the following is true for any $c_0<c_{\rho,\delta}/4$:

To prove Theorem \ref{mainquasiramdonmess2} we only need to prove the theorem for unit vector pairs $(v_1,v_2)$ where $v_2=cv_1+\epsilon _1v_3$ with $v_3$ a unit vector orthogonal to $v_1$, $|\epsilon_1|\geq e^{-c_*n/2}$, and where we can find disjoint subsets $D_1,D_2,D_3\subset[h]$ of cardinality $c_0^2n/12$ each, with $D:=D_1\cup D_2\cup D_3$, such that 

    \begin{enumerate}
        \item $(\kappa_0+\kappa_0/2)n^{-1/2}\leq |(v_1)_i|, |(v_3)_j|\leq (\kappa_1-\kappa_0/2)n^{-1/2}$ for all $i\in D_1,j\in D_2$, 
        \item $\max(|(v_1)_i|,|(v_3)_i|)\leq \sqrt{\frac{96}{c_0^2n}}$ for all $i\in D_1\cup D_2\cup D_3$.
        \item The inner product of $v_1,v_3$ satisfies
        $$
|\sin((v_1)_D,(v_3)_D)|\geq c_{\rho,\delta}^2c_0^2/64.
        $$
    \end{enumerate}
\end{corollary}
\begin{proof}
   For any fixed $w_0\in\mathbb{S}^{n-1}$, with probability $1-e^{-\Omega(n)}$ the claim of Lemma \ref{smalllemma2.5bound} holds for this vector $w_0$. Then the claim follows from all the previous discussions after fact \ref{fact408} and from the conclusion of Fact \ref{poisson424}. 
\end{proof}
We will fix the values of $c_{\rho,\delta},\kappa_0\in(0,1)$ and $\kappa_1>0$ throughout the paper. The value of $c_0$ will be chosen sufficiently small in Section \ref{verification2}. The subsets $D_1,D_2,D_3$ depend on the vectors $v_1,v_2$, and we will take a union bound over all possible subsets $D_i$'s.

Finally, we will need a simple fact that an approximation of a pair of vectors does not change the sin function $\sin(v_1,v_2)$ too much. 
\begin{fact}\label{fact703}
    Let $v_1,v_3$ be the two unit vectors in $\mathbb{R}^n$ specified in Corollary \ref{corollary2.678}.  Let $v_1',v_3'\in\mathbb{R}^n$ be unit vectors such that $\|v_i-v_i'\|_\infty\leq c_{\rho,\delta}^4c_0^4/2^{20}\cdot n^{-1/2}$ for both $i=1,3$ Then we have 
    $$
|\sin((v_1')_D,(v_3')_D)|\geq c_{\rho,\delta}^2c_0^2/128.
    $$
\end{fact}
\begin{proof}
First, we can check that $|\langle (v_1)_D,(v_3)_D\rangle-\langle (v_1')_D,(v_3')_D\rangle|\leq c_{\rho,\delta}^4c_0^4/2^{18}$ by triangle inequality. The conditions imposed in Corollary \ref{corollary2.678} imply that (via a manipulation of the definition of $\sin$ function) we have $|\langle (v_1)_D,(v_3)_D\rangle|\leq \|(v_1)_D\|_2\|(v_3)_D\|_2\sqrt{1-\tau^2}$ where $\tau=c_{\rho,\delta}^2c_0^2/64$. By construction of $v_1,v_3$, we have $\|(v_i)_D\|_2\geq c_{\rho,\delta}c_0/4$ for $i=1,3$; and $|\|(v_i')_D\|_2-\|(v_i)_D\|_2|\leq c_{\rho,\delta}^4c_0^4/2^{20}$ for both $i=1,3$. Combining the two inequalities, we can check after an elementary computation that 
$$
|\langle (v_1')_D,(v_3')_D\rangle |\leq \|(v_1')_D\|_2\|(v_3')_D\|_2\sqrt{1-(\tau/2)^2}
$$ and thus prove the fact by manipulating the definition of the $\sin$ function.
\end{proof}

\section{Verifying quasi-randomness conditions I}\label{verification1}

In this section and the next section we prove Theorem \ref{mainquasiramdonmess2}, one of the most essential technical ingredients of the paper.
There are two main challenges in this procedure:
\begin{enumerate}
    \item The most essential difficulty is that the entries of $A_n$ are not independent. An inversion of randomness technique is introduced in \cite{campos2025singularity} to address such problems. Here we need to upgrade this inversion method to a vector pair. The main idea of this inversion technique is that we use a random generation method (Lemma \ref{randmgenerationoflcd}) to sample vectors from a grid (Definition \ref{definitionboxnew}) and define certain threshold functions \eqref{thresholdsillus} measuring anti-concentration with respect a zeroed-out matrix $M_n$\eqref{recaplosts}. We use Littlewood-Offord considerations for the submatrix $H_1$, conditioning on a robust lower bound on certain intermediate singular values of $H_1$ (Theorem \ref{theorem12.234567twovectors}). Then we recover the loss of randomness by using instead the randomness from the random selection of coordinates associated to the column of $H_1^T$, which uses the known lower bound on intermediate singular values (Lemma \ref{singularmutualgoods}). This double counting approach effectively controls the number of grid points in certain nets defined by the threshold function (Theorem \ref{mainresultchapter3}). Finally in Section \ref{verification2}, we pass the estimate from the threshold function back to the essential LCDs.
    \item For Theorem \ref{mainquasiramdonmess2}, we need to consider a pair of vectors $(v_1,v_2)$. In this vector pair case, a new challenge is that the vector pair $(v_1,v_2)$ may not be orthogonal. Rather, we need a careful partitioning based on the exact inner product of $v_1,v_2$ and define the threshold functions accordingly. When $v_1,v_2$ are close to colinear, the anti-concentration is much weaker but we also have much fewer vectors of this kind. A family of notations for this purpose are introduced in Definition \ref{netsandlevels}.  
\end{enumerate}

We first introduce a zeroed out matrix $M_n$ relative to $A_n$: for a given $d\in\mathbb{N}$,
\begin{equation}\label{recaplosts}
    M_n=\begin{bmatrix}
        \mathbf{0}_{[d]\times[d]}&H_1^T\\H_1&\mathbf{0}_{[d+1,n]\times[d+1,n]}
    \end{bmatrix}
\end{equation} where $H_1$ is an $(n-d)\times d$ matrix with i.i.d. coordinates of distribution $(\zeta-\zeta')Z_\nu$, where $\zeta'$ is an independent copy of $\zeta$, $\nu\in(0,1)$ and $Z_\nu$ is a Bernoulli random variable with mean $\nu$, independent of $\zeta,\zeta'$. We also let $\Phi_\nu(2d,\zeta)$ denote an $2d$-dimensional random vector with i.i.d. entries of distribution $(\zeta-\zeta')Z_\nu$: this will be used in Theorem \ref{theorem12.234567twovectors}.

In our proof we shall always set $d=c_0^2n/4$ for some $c_0>0$ to be fixed later. When $c_0^2n/4$ is not an integer, we shall take $d=\lceil c_0^2n/4\rceil$ instead, but we again use the notation $d=c_0^2n/4$ just for notational simplicity (and also for $c_0^2/12$ in the following definition).

We now introduce a family notations for discrete approximations:

\begin{Definition}\label{netsandlevels} In this definition we always refer to $D=(D_1,D_2,D_3)$ as the collection of three disjoint intervals $D_1,D_2,D_3\subset[n]$ of cardinality $c_0^2n/12$ each, and we also use $D$ to denote the union $D=D_1\cup D_2\cup D_3$.
Fix two constants $\epsilon,\epsilon_1>0$. We define the following sets, threshold functions and discrete approximations:
\begin{enumerate}
    \item 
Subsets indicating the value of $v$ and $r$ on $D$:
$$\begin{aligned}\mathcal{I}(D):=&\{(v,r)\in \mathbb{S}^{n-1}\times\mathbb{S}^{n-1}:\langle v,r\rangle=0,\\&(\kappa_0+\kappa_0/2)n^{-1/2}\leq |v_i|,|r_j|\leq(\kappa_1-\kappa_0/2)n^{-1/2},\forall i\in D_1,\forall j\in D_2,\\&|v_i|,|r_i|\leq \sqrt{\frac{96}{c_0^2n}},\forall i\in D_1\cup D_2\cup D_3,\quad |\sin(v_D,r_D)|\geq c_{\rho,\delta}^2c_0^2/64 \}.\end{aligned}$$

$$\begin{aligned}\mathcal{I}'(D):=&\{(v,r)\in \mathbb{R}^{n+n}:\quad |\langle v,r\rangle|\leq \frac{1}{10},\\&\kappa_0n^{-1/2}\leq |v_i|,|r_j|\leq \kappa_1n^{-1/2},\forall i\in D_1,\forall j\in D_2\\&|v_i|,|r_i|\leq \sqrt{\frac{100}{c_0^2n}}\forall i\in D_1\cup D_2\cup D_3,\quad |\sin(v_D,r_D)|\geq c_{\rho,\delta}^2c_0^2/128 \}.\end{aligned}$$

\item For fixed $D$ and constants $0<\epsilon\leq\epsilon_1$, and any $c\in\mathbb{R}$ define the following sets
$$\begin{aligned}
\mathcal{P}_{\epsilon_1}(c)&:=\{(v,w)\in\mathbb{R}^n\times\mathbb{R}^n:\|v\|_2=1,\frac{9}{10}\leq\|w\|_2\leq\frac{11}{10},\exists r\in\mathbb{S}^{n-1}\text{ such that }\\&w=cv+\epsilon_1r,\quad (v,r)\in\mathcal{I}(D)
\},\end{aligned}$$  
$$\begin{aligned}\mathcal{P}^0_\epsilon(c)&:=\{(v,w)\in\mathbb{R}^{n}\times\mathbb{R}^{n}:\|v\|_2=1,\frac{9}{10}\leq\|w\|_2\leq\frac{11}{10},\exists r\in\mathbb{S}^{n-1},\exists t\in\mathbb{R},|t|\leq\epsilon\\&\text{such that }w=cv+tr,\quad (v,r)\in\mathcal{I}(D)\}.\end{aligned}$$

\item Discrete approximation of $\mathcal{P}_{\epsilon_1}$ and for $\mathcal{P}_\epsilon^0$ at scale $\epsilon:$ when $\epsilon_1\geq \epsilon$,
$$\begin{aligned}
\Lambda_{\epsilon,\epsilon_1}(c):=&\{
(v,w)\in B_n(0,2)\times B_n(0,2): v\in 4\epsilon n^{-1/2}\cdot\mathbb{Z}^n,\\&\text{ there exists } r\in B_n(0,2)\cap \frac{4\epsilon}{ \epsilon_1}n^{-1/2}\cdot\mathbb{Z}^n\text { such that }
w=cv+\epsilon_1r,\\&(v,r)\in \mathcal{I}'(D).\end{aligned}
$$

$$
\Lambda_\epsilon^0:=\{(v,v)\in B_n(0,2)\times B_n(0,2):v\in 4\epsilon n^{-1/2}\cdot\mathbb{Z}^n,v\in\mathcal{I}'(D)\}
$$(For a single vector $v$, we use $v\in \mathcal{I}'(D)$ to mean that $v$ satisfies all the conditions in $\mathcal{I}'(D)$ imposed on $v$, ignoring all the restrictions imposed on $r$.)

In the case when $D=[d]=[c_0^2n/4]$, we further make the following definitions:
\item Threshold: given $L>0$, $\epsilon_1>0$, for $(v,w)\in\mathcal{P}_{\epsilon_1}$(c) or $(v,w)\in{\Lambda_{\epsilon,\epsilon_1}}(c)$ define
\begin{equation}\label{thresholdsillus}
\tau_{L,\epsilon_1}(v,w):=\sup\{t\in[0,1]:\mathbb{P}(\|(M_nv,M_nw)\|_2\leq t\sqrt{n})\geq(4L^2t^2/\epsilon_1)^{n}\},
\end{equation}
and for $(v,w)\in\mathcal{P}_{\epsilon}^0(c)$ or $(v,w)\in \Lambda_\epsilon^0$, define
$$
\tau_{L,0}(v,w):=\sup\{t\in[0,1]:\mathbb{P}(\|(M_nv,M_nw)\|_2\leq t\sqrt{n})\geq(4L_0^2t)^{n}\}.
$$ where we take \begin{equation}\label{whatisell0}L_0={L}\cdot {\min(\kappa_0/8,c_{\rho,\delta}^4c_0^42^{-22})}.\end{equation}
\item $\mathcal{P}_{\epsilon_1}(c)$ and $\mathcal{P}_\epsilon^0(c)$ stratified by threshold function: we define, for any $\epsilon\leq\epsilon_1$, 
$$
\Sigma_{\epsilon,\epsilon_1}(c):=\{(v,w)\in\mathcal{P}_{\epsilon_1}(c):\tau_{L,\epsilon_1}(v,w)\in [\epsilon,2\epsilon]\},
$$
$$
\Sigma_{\epsilon,0}(c):=\{(v,w)\in\mathcal{P}_\epsilon^0(c)
:\tau_{L,0}(v,w)\in [\epsilon,2\epsilon]\}.$$
\item  Nets: for fixed $\epsilon_1\geq\epsilon$ we define the following nets: 
\begin{equation}\label{secondmomemtcomp}\begin{aligned}
\mathcal{N}_{\epsilon,\epsilon_1}(c)&:=\{(v,w)\in\Lambda_{\epsilon,\epsilon_1}(c):(\frac{L^2\epsilon^2}{\epsilon_1})^{
n}\leq \mathbb{P}(\|(M_nv,M_nw)\|_2\leq 4\epsilon\sqrt{n}),\\&\mathcal{L}_{A,op}(v,w,\epsilon\sqrt{n})\leq (2^{18}\frac{L^2\epsilon^2}{\epsilon_1})^{
n}\}, \end{aligned}
\end{equation} 
\begin{equation}\begin{aligned}\label{Nepsilon0}
    \mathcal{N}_{\epsilon,0}:=\{(v,w)\in\Lambda_\epsilon^0:&\quad (L_0^2\epsilon)^{n}\leq \mathbb{P}(\|(M_nv,M_nw)\|_2\leq 4\epsilon\sqrt{n}),\\&
    \mathcal{L}_{A,op}(v,w,\epsilon\sqrt{n})\leq (2^{18}L^2\epsilon)^{n}\},
    \end{aligned}
\end{equation}
where for $v,w\in\mathbb{R}^n$ and $t>0$ we define the concentration function via
$$
\mathcal{L}_{A,op}(v,w,t):=\sup_{z\in\mathbb{R}^{2n}}\mathbb{P}_{A_n}\left(\left\|\begin{bmatrix}A_nv-z_{[1,n]}\\ A_nw-z_{[n+1,2n]}\end{bmatrix}\right\|_2\leq t\text{ and } \|A_n\|_{op}\leq 4\sqrt{n}\right),
$$ where $(M_nv,M_nw)\in\mathbb{R}^{2n}$ has its first (resp. last) $n$ coordinates $M_nv$
(resp. $M_nw$).

\end{enumerate}
\end{Definition}

We also define $\mathcal{L}_{A}(v,w,t)$ which is the same as $\mathcal{L}_{A,op}(v,w,t)$ without taking intersection with the event $\{\|A_n\|_{op}\leq 4\sqrt{n}\}$. 

We note that in the definition of $\tau_{L,0}(v,w)$, $\Sigma_{\epsilon,0}(v,w)$ and $\mathcal{N}_{\epsilon,0}$ we have used a new (smaller) constant $L_0$ in place of $L$: this change will be useful for Proposition \ref{verysmallepsilon1}.

We end the definition with an important remark on how we assign a vector pair $(v,w)\in\mathbb{R}^{n+n}$ to $\Sigma_{\epsilon,\epsilon_1}(c)$ or $\Sigma_{\epsilon,0}(c)$. Given $(v,w)$ that satisfies the conditions in Definition \ref{netsandlevels} (2), then we can always find some $c,\epsilon_1$ such that $(v,w)\in\mathcal{P}_{\epsilon_1}(c)$. Then whenever $\epsilon_1\neq 0$, we evaluate the threshold function $\tau_{L,\epsilon_1}(v,w)$. There are two cases: \begin{enumerate}\item  Either $\tau_{L,\epsilon_1}(v,w)\leq 2 \epsilon_1$, then we assign  $(v,w)\in \Sigma_{\epsilon,
\epsilon_1}(c)$ for some $\epsilon\leq\epsilon_1$. \item Otherwise, $\tau_{L,\epsilon_1}(v,w)\geq 2\epsilon_1$, then this implies that $\tau_{L,0}(v,w)\geq2\epsilon_1$. In this scenario let $\epsilon:=\tau_{L,0}(v,w)$, then we assign  $(v,w)\in\Sigma_{\epsilon,0}(c)\subset \mathcal{P}_\epsilon^0(c)$.\end{enumerate}
This additional assignment explains why in (5), for the definition of $\Sigma_{\epsilon,\epsilon_1}(c)$, the threshold value of $(v,w)$ is always smaller than $2\epsilon_1$. This also explains why for $\Sigma_{\epsilon,0}(c)$ we only work with the threshold value in one single interval $[\epsilon,2\epsilon]$. 

In the general case for any $D\subset[n]$, one only needs to redefine the zeroed out matrix $M_n$ by letting the submatrix $H_1$ be supported on columns labeled by $D$ and rows labeled by $[n]\setminus D$. Then we define analogously the threshold functions and nets with respect to this zeroed out matrix determined by $D$. We omit the statement of these analogous definitions.

\subsection{LCD of a randomly generated vector pair}

In this section we prove that a pair of random vectors uniformly sampled from a grid has no rigid arithmetic structure with high probability. This procedure has played an important role in \cite{tikhomirov2020singularity} and \cite{campos2025singularity}, and also in \cite{han2025simplicity} for the case of a vector pair. Firstly, we define the grid we sample from.

\begin{Definition}\label{definitionboxnew} Fix three subsets $D_1,D_2,D_3\subset[n]$ with $|D_1|=|D_2|=|D_3|=c_0^2n/12$  and $D=D_1\cup D_2\cup D_3=[ c_0^2n/4]$. Fix two constants $\kappa\geq 2,\kappa'\geq \kappa$ and two integers $N\geq N_1\geq 2$. An $(N,N_1,\kappa,\kappa',D_1,D_2,D_3)$ box pair is defined as a pair of product sets $\mathcal{B}_1,\mathcal{B}_2$ of the form $$\mathcal{B}_1=B_1^1\times B_2^1\times\cdots\times B_{n}^1\subset\mathbb{Z}^{n},\quad \mathcal{B}_2=B_1^2\times B_2^2\times\cdots\times B_{n}^2\subset\mathbb{Z}^{n}$$ where we have
    \begin{enumerate}
    \item $|B_i^1|\geq N$ for each $i\in[n]$ and $|B_i^2|\geq N_1$ for each $i\in[n]$,\item $B_i^1=[-\kappa N,-N]\cup[N,\kappa N]\forall i\in D_1$,   $B_i^2=[-\kappa N_1,-N_1]\cup[N_1,\kappa N_1]\forall i\in D_2$.
    \item For each $i\in D$, $\max\{|x|:x\in B_i^1\}\leq \kappa'N$, $\max\{|x|:x\in B_i^2\}\leq \kappa'N_1$.
    \item $|\mathcal{B}_1|\leq(\kappa N)^{n}$ and $|\mathcal{B}_2|\leq(\kappa N_1)^{n}$.\end{enumerate} 
\end{Definition}

These boxes form a good covering for our candidate set of vectors:
\begin{lemma} \label{lemma3.23}
    For any given disjoint subsets $D_1,D_2,D_3\subset[n]$ of cardinality $c_0^2n/12$ each and $D_1\cup D_2\cup D_3=[c_0^2n/4]$,  the subset $\Lambda_{\epsilon,\epsilon_1}(c)$ can be covered by a family $\mathcal{F}_{\epsilon,\epsilon_1}(c)$ of $(N,N_1,\kappa,\kappa',D_1,D_2,D_3)$ box pairs $(\mathcal{B}_1,\mathcal{B}_2)_i,i\in\mathcal{F}_{\epsilon,\epsilon_1}(c)$ , where we take $N={\kappa_0}/4\epsilon$, $N_1=\kappa_0\epsilon_1/4\epsilon$, $\kappa=\max({\kappa_1}/{\kappa_0},2^8{\kappa_0}^{-4})$ and $\kappa'=\sqrt{\frac{100}{c_0^2}}/\kappa_0$: 
    \begin{equation}\label{productboxdef}
\Lambda_{\epsilon,\epsilon_1}(c)\subset\cup_{i\in\mathcal{F}_{\epsilon,\epsilon_1}(c)}4\epsilon n^{-1/2}\cdot (\mathcal{B}_1\times(c\mathcal{B}_1+\mathcal{B}_2))_i,
    \end{equation} 
    in the sense that for any $(v,w)\in \Lambda_{\epsilon,\epsilon_1}(c)$ with $w=cv+\epsilon_1 r$ then we have $(v,r)\in (4\epsilon n^{-1/2})\mathcal{B}_1\times (4\epsilon n^{-1/2}/\epsilon_1)\mathcal{B}_2$ for one such box $(\mathcal{B}_1\times\mathcal{B}_2)_i$ in the family $i\in \mathcal{F}_{\epsilon,
    \epsilon_1}(c)$.
    
  The index set has cardinality $|\mathcal{F}_{\epsilon,\epsilon_1}(c)|\leq \kappa^{2n}$.
\end{lemma}
The upper bound on $|\mathcal{F}_{\epsilon,\epsilon_1}(c)|$ only depends on $\kappa$ (which only depends on $B$) and not on $\kappa'$ which could be very large for small $c_0$. The proof of this lemma is similar to \cite{campos2025singularity}, Lemma 7.8 or \cite{han2025simplicity}, Lemma 3.27 and hence omitted. Note that the only added assumption here is assumption (3), but this assumption only reduces the cardinality of boxes.

 Random vectors generated from a box satisfy the following property: 
 \begin{lemma}\label{randmgenerationoflcd} 
Fix integers $N\geq N_1\geq 2$ and two constants $\kappa>2$ and $\kappa'>\kappa$. Consider an $(N,N_1,\kappa,\kappa',D_1,D_2,D_3)$ box pair $(\mathcal{B}_1,\mathcal{B}_2)$. 
    Let $\mathbb{P}_{X,Y}$ denote the probability measure where $X$ is chosen uniformly at random from $\mathcal{B}_1$ and $Y$ is chosen uniformly at random from $\mathcal{B}_2$ and independent of $X$.  Then we have, for any constant $\alpha>0$ such that $\alpha d\geq 4$ and $2^{14}B^2\kappa'Nd^{-1/2}/\alpha\leq 3^d$, the following estimate holds:
   \begin{equation}\label{assumptionmatrix}\begin{aligned}
        \mathbb{P}_{X,Y}&(\text{There exists $\phi=(\phi_1,\phi_2)\in\mathbb{R}^2: |\phi_1|\geq (4\kappa'N)^{-1}$ or $|\phi_2|\geq (4\kappa'N_1)^{-1}$}  \text{ and}\\& \|\phi\|_2\leq 512B^2d^{-1/2} \text{such that }\|\phi_1 X_D+\phi_2Y_D\|_\mathbb{T}\leq \sqrt{\alpha d/2})\leq (10^{12}(2+S)^2\alpha)^{d/4}, 
    \end{aligned}\end{equation} where $S=\frac{4\kappa'}{\kappa}$, and we recall that $\kappa'=\frac{10}{c_0\kappa_0}$.
\end{lemma}

The proof of this lemma is essentially the same as \cite{han2025simplicity}, Lemma 3.16, which was modified from \cite{campos2025singularity}, Lemma 7.4. We omit the straightforward modifications.

Now we define a notion of essential LCD for the vector pairs X and $Y$ from the box. 
\begin{Definition}\label{definitiontwocomponentlcds}
Fix some $L\geq 1$ and $\alpha_0\in(0,1)$.
Consider two vectors $X,Y\in\mathbb{R}^n$ and two fixed constant $t_1,t_2>0$. We define the essential least common denominator (LCD) of the vector pair $(X,Y)$ associated with these parameters as follows:
$$
\operatorname{LCD}_{L,\alpha_0}^\mathbf{t}(X,Y):=\inf\{\|\theta\|_2:\theta\in \mathbb{R}^2,\| \theta_1X+\theta_2Y\|_\mathbb{T}\leq L\sqrt{\log_+\frac{\alpha_0\|\sqrt{\sum_{i=1}^2\theta_i^2/t_i^2}}{L}}\},
$$
where $\log_+(x)=\max(\log x,0)$ and $\mathbf{t}$ denotes the tuple $(t_1,t_2)$.\end{Definition} 

We have made a repeated use of the symbol $L$ in both the threshold function and the essential LCD. We hope this will not lead to any confusion as we will later only use the explicit value for $L$ when used for the essential LCD.

This definition differs slightly from the standard definition of essential LCD of a matrix in \cite{rudelson2016no}, but is nevertheless essentially equivalent to this standard definition.

Throughout this chapter, $p$ is a constant satisfying $p\in(\frac{1}{2^7B^4},1)$ and $\nu=2^{-15}$.

\begin{corollary}\label{corollaryofthelcd}
Let $X,Y$ be uniformly sampled from an $(N,N_1,\kappa,\kappa',D_1,D_2,D_3)$ box pair with $N\geq N_1\geq 2$, and let $\alpha>0$ satisfy $\alpha d\geq 4,2^{16}B^2\kappa'Nd^{-1/2}/\alpha \leq 3^d$. Let $t_1=N^{-1}$ and $t_2=N_1^{-1}$. Then whenever $\log N$ is small enough relative to $\alpha$ in the sense that
$
\sqrt{\alpha c_0^2n/8}\geq (\frac{8\sqrt{2}}{\sqrt{\nu p}}+\frac{256\sqrt{2}B^2}{\sqrt{c_0}})\sqrt{\log_+\frac{c_{\rho,\delta}^2c_0^42^{-11}N}{\frac{8\sqrt{2}}{\sqrt{\nu p}}+\frac{256\sqrt{2}B^2}{\sqrt{c_0}}}}
$, the following probability estimate holds true: 
$$\begin{aligned}
\mathbb{P}_{X,Y}(\operatorname{LCD}&^\mathbf{t}_{\frac{8\sqrt{2}}{\sqrt{\nu p}}+\frac{256\sqrt{2}B^2}{\sqrt{c_0}},c_{\rho,\delta}^2c_0^42^{-16}}(r_nX_D,r_nY_D)\leq 2^{11}B^2, \\&|\sin(X_D,Y_D)|\geq c_{\rho,\delta}^2c_0^2/128)\leq (10^{12}\alpha(2+\frac{4\kappa'}{\kappa})^2)^{d/4},
\end{aligned}$$
where we take $r_n=\frac{c_0}{32\sqrt{n}}$ and in the probability we take an intersection with the event where the lower bound on $|\sin(X_D,Y_D)|$ holds.
\end{corollary}
The proof of this corollary is essentially the only place where we we use the lower bound on $|\sin(X_D,Y_D)|$. However, it is also not hard to check that this lower bound on $|\sin(X_D,Y_D)|$ is fundamental and can not be removed.
\begin{proof} Denote by $\psi=r_n\cdot\phi$.
    When $|\psi_1|\leq (4\kappa'N)^{-1}$ and $|\psi_2|\leq (4\kappa'N_1)^{-1}$, then for each $i\in D$ we have that $|\psi_1X_i+\psi_2Y_i|\leq \frac{1}{2}$ by definition of the box. Thus $$\begin{aligned}&\|\psi_1X_D+\psi_2Y_D\|_\mathbb{T}=\|\psi_1X_D+\psi_2Y_D
    \|_2\geq \sqrt{\psi_1^2N_1^2d+\psi_2^2N_2^2d}\cdot c_{\rho,\delta}^2c_0^22^{-10}\\&\geq c_{\rho,\delta}^2c_0^42^{-16}\sqrt{\phi_1^2/t_1^2+\phi_2^2/t_2^2}\geq (\frac{8\sqrt{2}}{\sqrt{\nu p}}+\frac{256\sqrt{2}B^2}{\sqrt{c_0}})\sqrt{\log_+\frac{c_{\rho,\delta}^2c_0^42^{-16}\sqrt{\sum_{i=1}^2\phi_i^2/t_i^2}}{\frac{8\sqrt{2}}{\sqrt{\nu p}}+\frac{256\sqrt{2}B^2}{\sqrt{c_0}}}}\end{aligned},$$
where we use the assumption on $\sin(X_D,Y_D)$ in the first inequality.

     When $|\psi_1|\geq (4\kappa'N)^{-1}/2$ or $|\psi_2|\geq (4\kappa'N)^{-1}/2$ and $\|\psi\|_2\leq 16d^{-1/2}$ (whenever $\|\phi\|_2\leq 16$ we have that $\|\psi\|_2\leq 16d^{-1/2}$), then we can use Lemma \ref{randmgenerationoflcd} to deduce that with probability $(10^{12}\alpha(2+\frac{4\kappa'}{\kappa})^2)^{d/4}$, we have that $$\begin{aligned}\|\psi_1r_nX_D+\psi_2r_nY_D\|_\mathbb{T}\geq \sqrt{\alpha d/2}\geq (\frac{8\sqrt{2}}{\sqrt{\nu p}}+\frac{256\sqrt{2}B^2}{\sqrt{c_0}})\sqrt{\log_+\frac{c_{\rho,\delta}^2c_0^42^{-16}\sqrt{\sum_{i=1}^2\theta_i^2/t_i^2}}{\frac{8\sqrt{2}}{\sqrt{\nu p}}+\frac{256\sqrt{2}B^2}{\sqrt{c_0}}}}\end{aligned}$$ where we used our assumption on $\theta_1,\theta_2$ in the second inequality.
\end{proof}

The key point of Corollary \ref{corollaryofthelcd} is that for a fixed $c_0$, we can take $\alpha>0$ sufficiently small so that the probability on the right hand side is smaller than any exponential in $n$, hence is super-exponentially small. This is crucial for Section \ref{crucialsect3}.

\subsection{A conditioned inverse Littlewood-Offord theorem for vector pairs}

We shall make a crucial use of the following inverse Littlewood-Offord theorem, which generalizes Theorem 6.1 in \cite{campos2025singularity} to multiple vectors. A more general version of the theorem is presented in Theorem \ref{theorem12.234567} and proven there. Here we only consider two vectors for simplicity.

\begin{theorem}\label{theorem12.234567twovectors} In this theorem we take $\nu=2^{-15}$ and $p\in(\frac{1}{2^7B^4},1)$ is a constant depending only on $\zeta$.
    For given $n\in\mathbb{N}$, $0<c_0\leq 2^{-50}B^{-4}$, let $d\leq c_0^2n$ and fix $\alpha_0\in(0,1)$. Let $\mathbf{t}=(t_1,t_2)$ be two positive real numbers. Consider two vectors $X_1,X_2\in\mathbb{R}^d$ that satisfy $\operatorname{LCD}_{\frac{8\sqrt{2}}{\sqrt{\nu p}}+\frac{256\sqrt{2}B^2}{\sqrt{c_0}},\alpha_0}^\mathbf{t}(\frac{c_0}{32\sqrt{n}}\mathbf{X})\geq 512B^2$ (where we denote $\frac{c_0}{32\sqrt{n}}\mathbf{X}=(\frac{c_0}{32\sqrt{n}}X_1,\frac{c_0}{32\sqrt{n}}X_2)$). Let $H$ be an $(n-d)\times 2d$ random matrix with i.i.d. rows of distribution $\Phi_\nu(2d;\zeta)$.  Then whenever $k\leq 2^{-21}B^{-4}d$ and $R^2t_1t_2\geq\exp(-2^{-23}B^{-4}d)$, we have the following 
    \begin{equation}\begin{aligned}
        \mathbb{P}_H&(\sigma_{2d-k+1}(H)\leq c_02^{-4}\sqrt{n}\text{ and }\|H_1X_i\|_2,\|H_2X_i\|_2\leq n\text{ for both }i=1,2)\\&\leq e^{-c_0nk/12}(\prod_{i=1}^2\frac{R t_i}{\alpha_0})^{2n-2d},
   \end{aligned}\end{equation}
    where $H_1=H_{[n-d]\times[d]}$, $H_2=H_{[n-d]\times[d+1,2d]}$ and $R=2^{44}B^2c_0^{-3}(\frac{8}{\sqrt{\nu p}}+\frac{256B^2}{\sqrt{c_0}})$, and $\sigma_1(H)\geq\sigma_2(H)\geq\cdots\sigma_{2d}(H)$ are the singular values of $H$ arranged in a decreasing order.
\end{theorem}

\subsection{Determining cardinality of the net via double counting}\label{crucialsect3}

The main result of this section is the following upper bound on the cardinality of $\mathcal{N}_{\epsilon,\epsilon_1}(c)$:

\begin{theorem}\label{mainresultchapter3}
    Given $L\geq 2$, $0<c_0\leq 2^{-24}$, $d=\lceil c_0^2n/4\rceil$ and two constants $0<\epsilon\leq\epsilon_1$ with $\epsilon/\epsilon_1\leq\kappa_0/8$ and $\log\epsilon^{-1}\leq V_{\ref{mainresultchapter3}}(B,c_0)L^{-64/c_0^2}n,$ where $V_{\ref{mainresultchapter3}}(B,c_0)>0$ is a function depending only on $B>1$ and $c_0\in(0,1)$. Then for any $|c|\leq 2$ we have, whenever $n\in\mathbb{N}_+$ is sufficiently large relative to $L$ and $c_0$,
    $$
|\mathcal{N}_{\epsilon,\epsilon_1}(c)|\leq(\frac{C_B}{c_0^{16}L^4\epsilon^2/\epsilon_1})^n,
    $$ where $C_B>0$ is a constant depending only on $B>0$.
\end{theorem}

 In the range $\epsilon/\epsilon_1\in(\kappa_0/8,1)$, we will not use the net $\mathcal{N}_{\epsilon,\epsilon_1}(c)$ but rather approximate these vectors by the other net $\mathcal{N}_{\epsilon_1,0}$ in proposition \ref{verysmallepsilon1}.

A similar upper bound can be obtained for $|\mathcal{N}_{\epsilon,0}|$, as follows. 

\begin{Proposition}\label{proposition717717}
    Given $L\geq 2$ such that $L_0\geq 2$ (see \eqref{whatisell0}), $0<c_0\leq \min(2^{-24},\kappa_0)$, $d=\lceil c_0^2n/4\rceil$ and a constant $\epsilon>0$ with $\log\epsilon^{-1}\leq V_{\ref{mainresultchapter3}}(B,c_0)L^{-64/c_0^2} n$.  Then whenever $n\in\mathbb{N}_+$ is sufficiently large relative to $L$ and $c_0$, we have 
    $$
|\mathcal{N}_{\epsilon,0}|\leq(\frac{C_B}{c_0^{22}L^4\epsilon})^n,
    $$ where $C_B>0$ is a constant depending only on $B$.
\end{Proposition}

Since $|\mathcal{N}_{\epsilon,0}|$ essentially involves only one vector, Proposition \ref{proposition717717} is implied by the proof of \cite{campos2025singularity}, Theorem 7.1 save for a few modifications, and we will only give a sketch. Rather, the proof of Theorem \ref{mainresultchapter3} is highly nontrivial and takes up the rest of this section.

In both Theorem \ref{mainresultchapter3} and Proposition \ref{proposition717717}, the statement “whenever $n$ is sufficiently large relative to $L$ and $c_0$” can be quantified as requiring that $n$ is larger than an explicit function of $L$ and $c_0$, see the paragraph before Lemma \ref{lemmafinaltwomoments} and equation \eqref{whatisalphaell}. As the explicit function is very complicated and not informative, we do not write it out. 

To implement the double counting procedure we will need a technical lemma on anticoncentration of matrix-vector product from \cite{campos2025singularity}, Lemma 7.5:
\begin{lemma}\label{singularmutualgoods}
    Fix $N\in\mathbb{N},n,d,k\in\mathbb{N}$ with $n-d\geq 2d\geq 2k$. Consider $H$ an $(n-d)\times 2d$ matrix with $\sigma_{2d-k}(H)\geq c_0\sqrt{n}/16$. Fix intervals $B_1,\cdots,B_{n-d}\subset\mathbb{Z}$ with $|B_i|\geq N$. Let $X$ be sampled uniformly from the product box $\mathcal{B}=B_1\times\cdots\times B_{n-d}$. Then 
$$
\mathbb{P}_X(\|HX\|_2\leq n)\leq (\frac{Cn}{dc_0N})^{2d-k}
$$
    for a universal constant $C>0$.
\end{lemma}

Recall our definition of the zeroed out matrix 
$$
M_n=\begin{bmatrix}
    \mathbf{0}_{[d]\times[d]}&H_1^T\\H_1&\mathbf{0}_{[n-d]\times[n-d]}
\end{bmatrix}
$$ where $H_1$ is a $(n-d)\times d$ matrix with i.i.d. entries of distribution $(\zeta-\zeta')Z_\nu$. Consider $H_2$ an independent copy of $H_1$ and define $H$ as the augmented $(n-d)\times 2d$ matrix
$$
H:=\begin{bmatrix}
    H_1&H_2
\end{bmatrix}.
$$
To initiate the double counting, for two vectors $X,Y\in\mathbb{R}^n$ we define two events $\mathcal{A}_1=\mathcal{A}_1(X,Y)$ and $\mathcal{A}_2=\mathcal{A}_2(X,Y)$ via 
$$
\mathcal{A}_1:=\{H:\|H_1X_{[d]}\|_2\leq 3n,\|H_1Y_{[d]}\|_2\leq 3n\},\|H_2X_{[d]}\|_2\leq 3n,\|H_2Y_{[d]}\|_2\leq 3n\}
$$

$$
\mathcal{A}_2:=\{H:\|H^TX_{[d+1,n]}\|_2\leq 6n,\|H^TY_{[d+1,n]}\|_2\leq 6n\}.
$$ We now correlate $\mathcal{A}_1,\mathcal{A}_2$ with the smallness of $\|MX\|_2,\|MY\|_2$:

\begin{fact}\label{fact737800} For any $c\in[-2,2]$ and $X,Y\in\mathbb{R}^n$ with $\mathcal{A}_1,\mathcal{A}_2$ defined as above. Then we have
$$
(\mathbb{P}_M(\|(MX,cMX+MY)\|_2\leq n))^2\leq\mathbb{P}_H(\mathcal{A}_1\cap\mathcal{A}_2
).$$    
\end{fact}

\begin{proof}
    Consider $M'$ an independent copy of $M$ and take a shorthand notation $Z=cX+Y$. We expand the event $\mathbf{1}(\|(MX,MZ)\|_2\leq n)$ as a sum of indicators to see, after taking the expectation,
    $$
(\mathbb{P}_M(\|(MX,MZ)\|_2\leq n))^2=\sum_{M,M'}\mathbb{P}(M)\mathbb{P}(M')\mathbf{1}(\|(MX,MZ)\|_2,\|(M'X,M'Z)\|_2\leq n)
    ,$$
    which is at most 
    $$\begin{aligned}
\sum_{H_1,H_2}\mathbb{P}(H_1)\mathbb{P}(H_2)&\mathbf{1}(\|H_1X_{[d]}\|_2\leq n,\|H_1Z_{[d]}\|_2\leq n,\|H_2X_{[d]}\|_2\leq n,\\&\|H_2Z_{[d]}\|_2\leq n,\|H^TX_{[d+1,n]}\|_2\leq 2n,\|H^TZ_{[d+1,n]}\|_2\leq 2n\}.
  \end{aligned}  $$

By our assumption, $Z=cX+Y$ with $|c|\leq2$, so the above expression is at most
 $$\begin{aligned}
\sum_{H_1,H_2}\mathbb{P}(H_1)\mathbb{P}(H_2)&\mathbf{1}(\|H_1X_{[d]}\|_2\leq n,\|H_1Y_{[d]}\|_2\leq 3n,\|H_2X_{[d]}\|_2\leq n,\\&\|H_2Y_{[d]}\|_2\leq 3n,\|H^TX_{[d+1,n]}\|_2\leq 2n,\|H^TY_{[d+1,n]}\|_2\leq 6n\},
  \end{aligned}  $$
  which is upper bounded by $\mathbb{P}_H(\mathcal{A}_1\cap\mathcal{A}_2)$ and this completes the proof.
\end{proof}

Then we introduce the notation of robust rank for rectangular random matrices. Define a family of events
$$
\mathcal{E}_k:=\{H:\sigma_{2d-k}(H)\geq c_0\sqrt{n}/16\text{ and }\sigma_{2d-k+1}(H)\leq c_0\sqrt{n}/16\},
$$  so that exactly one event from $\mathcal{E}_k,k\in[0,2d]$ holds for any given $H$.

For any $(N,N_1,\kappa,\kappa',D_1,D_2,D_3)$ box pair $(\mathcal{B}_1,\mathcal{B}_2)$, we specify a subset of atypical vectors $\operatorname{AT}(\mathcal{B}_1,\mathcal{B}_2)$ as the following subset of $\mathcal{B}_1\times\mathcal{B}_2$:
$$\begin{aligned}
\operatorname{AT}=&\operatorname{AT}(\mathcal{B}_1,\mathcal{B}_2)=\{(X,Y)\in \mathcal{B}_1\times\mathcal{B}_2:\operatorname{LCD}^{N^{-1},N_1^{-1}}_{\frac{8\sqrt{2}}{\sqrt{\nu p}}+\frac{256\sqrt{2}B^2}{\sqrt{c_0}},c_{\rho,\delta}^2c_0^42^{-16}}\\&(c_02^{-5}n^{-1/2}X_D,c_02^{-5}n^{-1/2}Y_D)\leq 2^{11}B^2,|\sin(X_D,Y_D)|\geq c_{\rho,\delta}^2c_0^2/128\}.\end{aligned}
$$
Those $(X,Y)$ not satisfying the lower bound on $|\sin(X_D,Y_D)|$ are not placed in $\operatorname{AT}$.

For any given fixed $L>0$ we take $\alpha=\alpha(L)$ to be the constant satisfying 
\begin{equation}\label{whatisalphaell}
(10^{12}\alpha(L)\cdot (2+\frac{40}{c_0\kappa_0\kappa})^2)^{c_0^2n/16}= (\frac{1}{2L^4})^n,\end{equation}
 so that $$\alpha(L)=L^{-64/c_0^2}\cdot \text{(a function of $c_0$ and $B$)}.$$  As the exact dependence of this function on $c_0$ and $B$ is involved and not informative enough, we will not work out the explicit expression of this function.

By Corollary \ref{corollaryofthelcd} (recall that $\nu=2^{-15}$ and $p\geq 1/(2^7B^4)$), there exists a function $V_{\ref{corollaryofthelcd}}(B,c_0)\in\mathbb{R}_+$ depending only on $B$ and $c_0$ such that whenever \begin{equation}\label{restrictiononN}N\leq \exp(L^{-64/c_0^2}\cdot V_{\ref{corollaryofthelcd}}(B,c_0)\cdot n),\end{equation} we have 
$$
\mathbb{P}_{X,Y}((X,Y)\in \operatorname{AT})\leq (\frac{1}{2L^4})^n.
$$

Throughout the following, we assume that $n$ is large enough relative to the constants $L,c_0$ so that $\alpha(L)\cdot d\geq \alpha(L)c_0^2n/4\geq 4$, justifying a condition made in Corollary \ref{corollaryofthelcd}.

We complete the double counting procedure in the following lemma:

\begin{lemma}\label{lemmafinaltwomoments} There exists a function $V_{\ref{lemmafinaltwomoments}}(B,c_0)>0$ depending only on $B>1$ and $c_0\in(0,1)$ such that the following holds. Fix some $c\in[-2,2]$ and two integers $N\geq N_1\geq 2$ and $n\in\mathbb{N}$ satisfying 
    $N\leq \exp(L^{-64/c_0^2}V_{\ref{lemmafinaltwomoments}}(B,c_0)n)$. Let $X,Y$ be uniformly sampled from a $(N,N_1,\kappa,\kappa',D_1,D_2,D_3)$ box. Then whenever $n$ is sufficiently large relative to $L,c_0$,
    $$\begin{aligned}
\mathbb{P}_{X,Y}& \left(\mathbb{P}_{M_n}(\|(M_nX,M_n(cX+Y))\|_2\leq n)\geq(\frac{L^2}{NN_1})^n,|\sin(X_D,Y_D)|\geq \frac{c_{\rho,\delta}^2c_0^2}{128}\right)\\&\leq (\frac{R_2}{L^2})^{2n}
 \end{aligned}   $$ where $R_2=C_Bc_0^{-16}$ and $C_B>0$ is a constant depending only on $B$.
\end{lemma}

\begin{proof} Let the events $\mathcal{A}_1,\mathcal{A}_2,\mathcal{E}_k$ be defined as above.
We denote 
$$\begin{aligned}
\mathcal{E}:=&\{(X,Y)\in\mathcal{B}_1\times\mathcal{B}_2:\quad |\sin(X_D,Y_D)|\geq c_{\rho,\delta}^2c_0^2/128,
\\&\mathbb{P}_{M_n}(\|(M_nX,M_n(cX+Y))\|_2\leq n)\geq (L^2/NN_1)^n\}.
\end{aligned}$$ Then we can write 
$$
\mathbb{P}_{X,Y}(\mathcal{E})\leq\mathbb{P}_{X,Y}((X,Y)\in \operatorname{AT})+\mathbb{P}_{X,Y}(\mathcal{E}\cap\{(X,Y)\notin \operatorname{AT}\}).
$$ We also define the set of typical directions for $(X,Y) \in\mathcal{B}_1\times\mathcal{B}_2$:
$$
T=\{(X,Y):\operatorname{LCD}^{N^{-1},N_1^{-1}}_{\frac{8\sqrt{2}}{\sqrt{\nu p}}+\frac{256\sqrt{2}B^2}{\sqrt{c_0}},c_{\rho,\delta}^2c_0^42^{-16}}(c_02^{-5}n^{-1/2}X_D,c_02^{-5}n^{-1/2}Y_D)\geq 2^{11}B^2\},
$$
and define a two variable function 
$$\begin{aligned}
f(X,Y)&:=\mathbb{P}_M(\|(M_nX,M_n(cX+Y))\|_2\leq n)\cdot\mathbf{1}((X,Y)\in T).
\end{aligned}$$

Then by the above reasoning and Markov's inequality we have 
\begin{equation}\label{whatdoeswehaveattneend?}
    \mathbb{P}_{X,Y}(\mathcal{E})\leq \mathbb{P}_{X,Y}(f(X,Y)\geq (\frac{L^2}{NN_1})^n)+(\frac{1}{2L^4})^n\leq (\frac{N^2N_1^2}{L^4})^n\mathbb{E}_{X,Y}f(X,Y)^2+(\frac{1}{2L^4})^{n}.
\end{equation}

By Fact \ref{fact737800} and a further partitioning, we have the decomposition
\begin{equation}
    \mathbb{P}_{M_n}(\|(M_nX,M_n(cX+Y)\|_2\leq n)^2\leq\sum_{k=0}^{2d}\mathbb{P}_H(\mathcal{A}_2\mid\mathcal{A}_1\cap\mathcal{E}_k)\mathbb{P}_H(\mathcal{A}_1\cap\mathcal{E}
_k),\end{equation}
so that 
\begin{equation}
f(X,Y)^2\leq\sum_{k=1}^{2d}\mathbb{P}_H(\mathcal{A}_2\mid\mathcal{A}_1\cap\mathcal{E}_k)\mathbb{P}_H(\mathcal{A}_1\cap\mathcal{E}_k)\mathbf{1}((X,Y)\in T).
\end{equation}

We first take $k\leq 2^{-21}B^{-4}d$. Then by Theorem \ref{theorem12.234567twovectors}, we get that whenever $N^2\leq R^2\exp(-2^{-23}B^{-4}d)$ where the constant $R$ is specified in Theorem \ref{theorem12.234567twovectors}, we have 
$$
\mathbb{P}(\mathcal{A}_1\cap\mathcal{E}_k)\leq \exp(-c_0nk/12)(\frac{R_0^2}{NN_1})^{2n-2d},
$$ where $R_0=2^{70}B^4c_0^{-8}c_{\rho,\delta}^{-2}$ (we used $p\geq 1/(2^7B^4)$ and $\nu=2^{-15}$). (More precisely, $\mathcal{A}_1$ is the event where $\|H_1X_D/3\|_2,\|H_1Y_D/3\|_2,\|H_2X_D/3\|_2,\|H_2 Y_D/3\|_2\leq n$, and by assumption the essential LCD (for abbreviation, let $L_*,\alpha_*$ denote the parameters we have used for the LCD of $X_D,Y_D$) satisfies $\operatorname{LCD}_{L_*,\alpha_*}^{N^{-1},N_1^{-1}}(X_D,Y_D)\geq2^{11}B^2$, so that $\operatorname{LCD}_{L_*,\alpha_*}^{3/N,3/N_1}(X_D/3,Y_D/3)\geq 512B^2$, and we can apply Theorem \ref{theorem12.234567twovectors} with the new parameter $\mathbf{t}=(3/N,3/N_1)$).

Now we consider $\mathbb{P}_{(X,Y),H}(\mathcal{A}_2\mid\mathcal{A}_1\cap\mathcal{E}_k)$. Define $g_k(X,Y):=\mathbb{P}_H(\mathcal{A}_2\mid\mathcal{A}_1\cap\mathcal{E}_k)$. Then we have
$$
\mathbb{E}_{X,Y}[g_k(X,Y)]=\mathbb{E}_{X_{[d]},Y_{[d]}}\mathbb{E}_H[\mathbb{E}_{X_{[d+1,n]},Y_{[d+1,n]}}\mathbf{1}[\mathcal{A}_2]]\mid\mathcal{A}_1\cap\mathcal{E}_k].
$$

For any given $H\in\mathcal{A}_1\cap\mathcal{E}_k$ where $k\leq 2^{-21}B^{-4}d$, we have $\sigma_{2d-k}(H)\geq c_0\sqrt{n}/16$, so that by Lemma \ref{singularmutualgoods} and by the independence of $X$ and $Y$, we have that 
$$\begin{aligned}\mathbb{E}_{X_{[d+1,n]},Y_{[d+1,n]}}\mathbf{1}(\mathcal{A}_2)&=\mathbb{P}_{X_{[d+1,n]}}(\|H^TX_{[d+1,n]}\|_2\leq 6n)
\mathbb{P}_{Y_{[d+1,n]}}(\|H^TY_{[d+1,n]}\|_2\leq 6n)\\&
\leq (\frac{C'n}{c_0dN})^{2d-k}(\frac{C'n}{c_0dN_1})^{2d-k}
\leq (\frac{16C'^2}{c_0^6 NN_1})^{2d-k},\end{aligned}$$
where $C'>0$ is an absolute constant. We now take $R_1=\max(16C'^2c_0^{-6},R_0^2)$.

Finally, we consider  $k\geq 2^{-21}B^{-4}d:=\alpha'd$. In this case, we have (where $\alpha'=2^{-21}B^{-4})$:
$$
\sum_{k=2^{-21}B^{-4}d}^{2d}\mathbb{P}_H(\mathcal{A}_2\mid\mathcal{A}_1\cap\mathcal{E}_k)\mathbb{P}_H(\mathcal{A}_1\cap\mathcal{E}_k)\leq \mathbb{P}(\sigma_{2d-\alpha'd}(H)\leq c_0\sqrt{n}/16)\leq \exp(-c_0\alpha'dn/12).
$$
Combining all the above cases, we conclude that 
$$
\mathbb{E}_{X,Y}f(X,Y)^2\leq \sum_{k=1}^{\alpha'd}\exp(-c_0kn/12)(\frac{R_1}{NN_1})^{2n-k}+\exp(-c_0\alpha'dn/12).
$$ When $N$ is not too large in the sense that $(R_1/N)^2\geq\max(\exp(-c_0n/12),\exp(-c_0\alpha'd/24))$, then we have that 
$$
\mathbb{E}_{X,Y}f(X,Y)^2\leq (\frac{2R_1}{NN_1})^{2n}.
$$
Taking this estimate into \eqref{whatdoeswehaveattneend?} completes the proof.

Throughout the proof there are three places where we placed an upper bound on the size of $N$. The first is in \eqref{restrictiononN}, the second is when we take $N^2\leq R^2\exp(-2^{-23}B^{-4}d)$ and the third is in the last paragraph. We take the intersection of all these constraints into the following condition: for any $L\geq 1$ we require that $N\leq \exp(L^{-64/c_0^2}V_{\ref{lemmafinaltwomoments}}(B,c_0)n)$ where $V_{\ref{lemmafinaltwomoments}}(B,c_0)>0$ is a function that depends only on $B$ and $c_0$. 

\end{proof}

Now we are ready to complete the proof of Theorem \ref{mainresultchapter3}.

\begin{proof}[\proofname\ of Theorem \ref{mainresultchapter3}] We apply Lemma \ref{lemma3.23} with $N=\kappa_0/4\epsilon,N_1=\kappa_0\epsilon_1/4\epsilon$ and write 
$$
\mathcal{N}_{\epsilon,\epsilon_1}(c)\subseteq\cup_{(\mathcal{B}_1,\mathcal{B}_2)\in\mathcal{F}_{\epsilon,\epsilon_1}(c)}((4\epsilon n^{-1/2})\cdot(\mathcal{B}_1,c\mathcal{B}_1+\mathcal{B}_2))\cap\mathcal{N}_{\epsilon,\epsilon_1}(c)
$$
    where $({B}_1,c\mathcal{B}_1+\mathcal{B}_2)$ denotes $\{(X,cX+Y):X\in \mathcal{B}_1,Y\in\mathcal{B}_2$\}. Thus
    $$
|\mathcal{N}_{\epsilon,\epsilon_1}(c)|\leq|\mathcal{F}_{\epsilon,\epsilon_1}(c)|\max_{(\mathcal{B}_1,\mathcal{B}_2)\in\mathcal{F}_{\epsilon,\epsilon_1}(c)}|(4\epsilon n^{-1/2}\cdot (\mathcal{B}_1,c\mathcal{B}_1+\mathcal{B}_2))\cap\mathcal{N}_{\epsilon,\epsilon_1}(c)|. 
    $$
    Multiplying by $\sqrt{n}/4\epsilon$, we get from Lemma \ref{lemmafinaltwomoments} that
    $$\begin{aligned}
&|(4\epsilon n^{-1/2}\cdot(\mathcal{B}_1,c\mathcal{B}_1+\mathcal{B}_2))\cap\mathcal{N}_{\epsilon,\epsilon_1}(c)|\\&\leq\{(X,Y)\in \mathcal{B}_1\times\mathcal{B}_2:\mathbb{P}_{M_n}(\|(M_nX,M_n(cX+Y))\|_2\leq n\}\geq (L^2\epsilon^2/\epsilon_1)^n\}|\leq(\frac{R_2'}{L^2})^{2n}|\mathcal{B}|,
    \end{aligned}$$ where $R_2'$ is $R_2$ multiplied by a constant depending on $\kappa_0$ and our upper bound on $\log \epsilon^{-1}$ ($\log\epsilon^{-1}\leq V_{\ref{mainresultchapter3}}(B,c_0)L^{-64/c_0^2}n$ for a suitable given function $V_{\ref{mainresultchapter3}}(B,c_0)>0$) yields the necessary upper bound on $\log N$ required in the application of Lemma \ref{lemmafinaltwomoments}.

    By Lemma \ref{lemma3.23} we have $|\mathcal{F}_{\epsilon,\epsilon_1}(c)|\leq\kappa^n$, and that $|\mathcal{B}_1|\leq (\kappa N)^n,|\mathcal{B}_2|\leq(\kappa N_1)^n$. Thus
    $$
|\mathcal{N}_{\epsilon,\epsilon_1}(c)|\leq\kappa^n(\frac{R_2'}{L^2})^{2n}\kappa^{2n}N^nN_1^n\leq (\frac{C_B}{c_0^{16}L^4\epsilon^2/\epsilon_1})^n
    $$ for a constant $C_B>0$ depending only on $B$. This completes the proof.
\end{proof}

We also sketch the proof for Proposition \ref{proposition717717} for completeness:
\begin{proof}[\proofname\ of Proposition \ref{proposition717717}(sketch)] By definition of $\Lambda_\epsilon^0$, the two vectors in $\mathcal{N}_{\epsilon,0}$ are the same. It suffices to prove the claimed upper bound holds for the cardinality of the following set
\begin{equation}\label{con917}
\{v\in\Lambda_\epsilon^0:\mathbb{P}(\|Mv\|_2\leq 4\epsilon\sqrt{n})\geq (L_0^2\epsilon)^n\}.
\end{equation} This is proven in \cite{campos2024least}, Theorem III.2 (more precisely, although the definition of $\mathcal{N}_\epsilon$ involves one extra condition, the proof for the upper bound for $|\mathcal{N}_\epsilon|$ only uses the condition in \eqref{con917}), with the constant $L$ there in place of $L_0^2$ here. It suffices to plug in the expression of $L_0$.
    
\end{proof}

\subsection{Verifying properties of a net}
We need a simple Fourier replacement lemma that compares the anticoncentration with respect to $M_n$ to the anticoncentration with respect to $A_n$. The following lemma is very similar to \cite{campos2025singularity}, Lemma 8.1 and \cite{han2025simplicity}, Lemma 3.24 and thus we omit the straightforward proof.
\begin{lemma}\label{lemmafs}
    For any $L>0$, $\epsilon_1>0$, $t_0>0$ and any $(v,w)\in\mathbb{R}^{2n}$ satisfying that 
    $$
\mathbb{P}(\|(M_nv,M_nw)\|_2\leq t\sqrt{n})\leq (4L^2t^2/\epsilon_1)^{n}\quad \forall t\geq t_0,\quad \text{(resp. $(4L^2t)^{n}\quad\forall t\geq t_0$) },
    $$
    we must have, 
    $$
\mathcal{L}\left((A_nv,A_nw),t_0\sqrt{n}\right)\leq (50L^2t_0^2/\epsilon_1)^{n},\quad\text{(resp. $(50L^2t_0)^{n}$)},
    $$ where $(A_nv,A_nw)\in \mathbb{R}^{2n}$ has its first (resp. last) n coordinates $A_nv$ (resp., $A_nw$).
\end{lemma}
The following proposition shows that $\Sigma_{\epsilon,\epsilon_1}(c)$ serves as a net:
\begin{Proposition}\label{proposition102} Assume that $2^{16}B^4d\leq n$.
    Let $\epsilon,\epsilon_1>0$ satisfy that \begin{equation}\label{satisfy3.14wes}\epsilon,\frac{\epsilon}{\epsilon_1}\in(0,\min(\kappa_0/8,c_{\rho,\delta}^4c_0^4 2^{-22})).\end{equation} Then for any $(v,w)\in\Sigma_{\epsilon,\epsilon_1}(c)$ with $w=cv+\epsilon_1 r$ there exists $(v',w')=(v',cv'+\epsilon_1r')\in \mathcal{N}_{\epsilon,\epsilon_1}(c)$ with $\|v-v'\|_\infty\leq4\epsilon n^{-1/2}$, $\|r-r'\|_\infty\leq 4\epsilon\epsilon_1^{-1}n^{-1/2}$ and $\|w-w'\|_\infty\leq 12\epsilon n^{-1/2}$.

\end{Proposition}

\begin{proof} First, it is not hard to check that a necessary condition for the set $\mathcal{P}_{\epsilon_1}(c)$ and $\mathcal{P}_{\epsilon,0}(c)$ in Definition \ref{netsandlevels} to be nonempty is that $c\in[-\sqrt{2},\sqrt{2}]$.

For a given $(v,w)=(v,cv+\epsilon_1 r)\in\Sigma_{\epsilon,\epsilon_1}(c)$ we define two random vectors $t^1=(t_1^1,\cdots,t_n^1)$ and $t^2=(t_1^2,\cdots,t_n^2)$ with mutually independent coordinates, where $\mathbb{E}t_i^1=\mathbb{E}t_i^2=0,i\in[n]$, $|t_i^1|\leq 4\epsilon n^{-1/2}$, $|t_i^2|\leq 4\epsilon/\epsilon_1\cdot  n^{-1/2}$ and $v-t^1\in 4\epsilon n^{-1/2}\cdot \mathbb{Z}^n$, $r-t^2\in 4\epsilon/\epsilon_1 n^{-1/2}\cdot\mathbb{Z}^n$ almost surely. We define $v'=v-t^1$ and $r'=r-t^2$, and write $w'=cv'+\epsilon_1r'$. Then by definition, $\|v-v'\|_\infty\leq 4\epsilon n^{-1/2}$ and $\|w-w'\|_\infty\leq 4\epsilon\epsilon_1^{-1}n^{-1/2}$, and we have $\|w-w'\|_\infty\leq (1+|c|)4\epsilon n^{-1/2}\leq 12\epsilon n^{-1/2}$. Also, from the fact that $(v,r)\in \mathcal{I}(D)$, our assumption on $\epsilon,\epsilon_1$ and Fact \ref{fact703}, we can see that $(v',r')\in \mathcal{I}'(D)$ almost surely. 

Then we only need to check the conditions imposed on the threshold function. That is, we show there exists a $(v',w')$ such that 
\begin{equation}\label{checkfirstcondition}
\mathbb{P}(\|(M_nv',M_nw')\|_2\leq 4\epsilon\sqrt{n})\geq  (L^2\epsilon^2/\epsilon_1)^n,\quad \mathcal{L}_{A,op}(v',w',\epsilon\sqrt{n})\leq (2^{18}L^2\epsilon^2/\epsilon_1)^n.\end{equation}
We first check that the second condition is satisfied by all pairs $(v',w')$ in this construction. Write $\mathcal{K}:=\{\|A_n\|_{op}\leq 4\sqrt{n}\}$ and let $\tau(v'),\tau(w')\in\mathbb{R}^n$ be such that 
\begin{equation}\label{expandsmallball}\begin{aligned}
\mathcal{L}_{A,op}(v',w',\epsilon\sqrt{n}):&=\mathbb{P}(\left\|\begin{bmatrix}A_nv-A_nt^1-\tau(v')\\A_nw-A_n(ct^1+\epsilon t^2)-\tau(w')\end{bmatrix}\right\|_2\leq\epsilon\sqrt{n}\text{ and }\mathcal{K})
\\&\leq \mathbb{P}(\left\|\begin{bmatrix}A_nv-\tau(v')\\A_nw-\tau(w')\end{bmatrix}\right\|_2\leq 18\epsilon\sqrt{n})\\&\leq \mathcal{L}_A(v,w,18\epsilon\sqrt{n})\leq (2^{18}L^2\epsilon^2/\epsilon_1)^n\end{aligned}\end{equation}
where we use the assumption $(v,w)\in\Sigma_{\epsilon,\epsilon_1}(c)$ combined with Lemma \ref{lemmafs}. 

Then we check the first condition of \eqref{checkfirstcondition} holds for some $(v',w')$. Define the event $\mathcal{E}:=\{M_n:\|(M_nv,M_nw)\|_2\leq 2\epsilon\sqrt{n}\}$. Then for all $(v',w')$ we have
$$
\mathbb{P}_{M_n}(\|(M_nv',M_nw')\|_2\leq 4\epsilon\sqrt{n})\geq \mathbb{P}_{M_n}(\|(M_nt^1,cM_nt^1+\epsilon_1M_nt^2)\|_2\leq 2\epsilon\sqrt{n}\text{ and }\mathcal{E}).
$$
Taking expectations with respect to $t=(t^1,t^2)$, we get
\begin{equation}
\label{combininginto}\begin{aligned}&\mathbb{E}_t\mathbb{P}_{M_n}(\|(M_nv',M_nw')\|_2\leq4\epsilon\sqrt{n})\\&\geq(1-\mathbb{E}_t\mathbb{P}_{M_n}(\|(M_nt^1,cM_nt^1+\epsilon_1M_nt^2)\|_2\geq 2\epsilon\sqrt{n}\mid\mathcal{E}))\mathbb{P}_{M_n}(\|(M_nv,M_nw)\|_2\leq2\epsilon\sqrt{n}
).\end{aligned}\end{equation}
Then we will prove, after exchanging the order of the expectation, that 
\begin{equation}\label{that7334}
\mathbb{E}_{M_n}[\mathbb{P}_t(\|(M_nt^1,cM_nt^1+\epsilon_1 M_nt^2)\|_2\geq2\epsilon\sqrt{n})\mid\mathcal{E}]\leq\frac{1}{2}. 
\end{equation} Now use the fact that $t_i$ has zero mean and are independent, we use triangle inequality to expand the following sum as
$$\begin{aligned}
&\mathbb{E}_t\|M_nt^1\|_2^2+\mathbb{E}_t\|cM_nt^1+\epsilon_1t^2\|_2^2\leq 3\mathbb{E}_t \|M_nt^1\|_2^2+2\epsilon_1^2\mathbb{E}_t\|M_nt^2\|_2^2\\&\leq \sum_i \mathbb{E}_t[3(t_i^1)^2+2\epsilon_1^2(t_i^2)^2]\sum_j M_{ij}^2\leq 100\epsilon^2n^{-1}\|M_n\|_{HS}^2\leq \epsilon^2 n\end{aligned}
$$
where the last inequality follows from our assumption on $d$: $\|M_n\|_{HS}^2=\sum_{i,j}M_{ij}^2\leq 256B^4d(n-d)\leq n^2/100$.
This verifies \eqref{that7334} via a second moment computation.

Since $(v,w)\in\Sigma_{\epsilon,\epsilon_1}(c)$ we have by definition of the threshold function $\tau_{L,\epsilon_1}(v,w)$ that $\mathbb{P}_{M_n}(\|(M_nv,M_nw)\|_2\leq 2\epsilon\sqrt{n})\geq(2\epsilon^2L^2/\epsilon_1)^n$.

Combining these estimates into \eqref{combininginto}, we have thus verified that 
$$
\mathbb{E}_t\mathbb{P}_{M_n}(\|(M_nv',M_nw')\|_2\leq 4\epsilon\sqrt{n})\geq(1/2)((2\epsilon^2L^2/\epsilon_1)^n\geq (L^2\epsilon^2/\epsilon_1)^n,
$$ so that the first estimate in \eqref{checkfirstcondition} is satisfied by some $(v',w')$.

\end{proof}

Proposition \ref{proposition102} does not cover the range when $\epsilon_1/\epsilon$ is not large enough. For this range, and for those $(v,w)\in\Sigma_{\epsilon,0}(c)$, we use a different argument to verify the net property:

\begin{Proposition}\label{verysmallepsilon1}
    Assume that $2^{16}B^4d\leq n$.
    Let $0<\epsilon\leq\epsilon_1$ satisfy that $\epsilon\leq\kappa_0/8$ and 
    $\frac{\epsilon}{\epsilon_1}\in(\min(\kappa_0/8,c_{\rho,\delta}^4c_0^4 2^{-22}),1).$ Then for any $(v,w)\in\Sigma_{\epsilon,\epsilon_1}(c)$ there exists $(v',w')\in \mathcal{N}_{\epsilon_1,0}(c)$ with $\|v-v'\|_\infty\leq4\epsilon_1 n^{-1/2}$ and $\|w-w'\|_\infty\leq 12\epsilon_1 n^{-1/2}$.

Similarly, for any $(v,w)\in\Sigma_{\epsilon,0}(c)$ we can find $(v',w')\in\mathcal{N}_{\epsilon,0}$ with $\|v-v'\|_\infty\leq4\epsilon n^{-1/2}$, $\|w-w'\|_\infty\leq 12\epsilon n^{-1/2}$.
   
\end{Proposition}

\begin{proof}
We only prove the first case, as the second case $(v,w)\in\Sigma_{\epsilon,0}(c)$ can be obtained from a direct modification from \cite{campos2021singularity}, Lemma 8.2 or from modifying the proof of Proposition \ref{proposition102}.

In the first case, we follow similar steps as in the proof of Proposition \ref{proposition102} 
but now define $t^1,t^2$ satisfying $\mathbb{E}t_i^1=\mathbb{E}t_i^2=0\forall i\in[n]$, $|t_i^1|\leq 4\epsilon_1n^{-1/2},|t_i^2|\leq 4n^{-1/2}$ and that $v-t^1\in 4\epsilon_1 n^{-1/2}\cdot\mathbb{Z}^n$, $r-t^2\in 4n^{-1/2}\cdot\mathbb{Z}^n$. Set $v'=w'=v-t^1$, then $\|v-v'\|_\infty\leq 4\epsilon_1 n^{-1/2}$ and $\|w-w'\|_\infty\leq 12\epsilon_1n^{-1/2}$. We have by construction that $(v',w')\in\Sigma_\epsilon^0$. Then we can prove that there exists a realization $(v',w')$ satisfying that
\begin{equation}\label{771equations}
\mathbb{P}(\|(M_nv',M_nw')\|_2\leq 4\epsilon_1\sqrt{n})\geq  (L_0^2\epsilon_1)^n,\quad \mathcal{L}_{A,op}(v',w',\epsilon_1\sqrt{n})\leq (2^{18}L^2\epsilon_1)^n,\end{equation}
where to justify the second inequality we use the fact that $
\tau_{L,\epsilon_1}(v,w)\leq\epsilon_1$ and we use Lemma \ref{lemmafs} to transfer this bound to an estimate on $\mathcal{L}_{A,op}$, using similar steps as in \eqref{expandsmallball}. 

To justify the first inequality in \eqref{771equations}, we can first prove that 
\begin{equation}
\label{combininginto2}\begin{aligned}&\mathbb{E}_t\mathbb{P}_{M_n}(\|(M_nv',M_nw')\|_2\leq4\epsilon_1\sqrt{n})\\&\geq(1-\mathbb{E}_t\mathbb{P}_{M_n}(\|(M_nt^1,cM_nt^1+\epsilon_1M_nt^2)\|_2\geq 2\epsilon_1\sqrt{n}\mid\mathcal{E}))\mathbb{P}_{M_n}(\|(M_nv,M_nw)\|_2\leq2\epsilon\sqrt{n}
),\end{aligned}\end{equation} where $\mathcal{E}:=\{M_n:\|(M_nv,M_nw)\|_2\leq2\epsilon\sqrt{n}\}$. Then via a similar second moment estimate and using the assumption $(v,w)\in\Sigma_{\epsilon,\epsilon_1}(c)$, we get that this is lower bounded by 
$$
(L^2\epsilon^2/\epsilon_1)^n\geq (L_0^2\epsilon_1)^n
$$ using our definition of $L_0$. This verifies that $(v',w')\in\mathcal{N}_{\epsilon,0}$ and completes the proof.

We also note the following simple fact for future use:
\begin{fact}\label{fact4.21}
    Let $r\geq t\geq 0$ and consider $X$ a random variable in $\mathbb{R}^{n}$. Then 
    $$
\mathcal{L}(X,t)\leq\mathcal{L}(X,r)\leq (1+2r/t)^{n}\mathcal{L}(X,t).
    $$
\end{fact} This is because a ball of radius $r$ in $\mathbb{R}^{n}$ can be covered by $(1+2r/t)^{n}$ balls of radius $t$.

\end{proof}

\section{Verifying quasi-randomness conditions II}\label{verification2}
In this section we complete the proof of Theorem \ref{mainquasiramdonmess2}, our main quasi-randomness theorem. The main input in its proof is Theorem \ref{mainresultchapter3} and Proposition \ref{proposition717717}.

Recall that we need to show the subvector LCD $\hat{D}_{\alpha,\gamma,\mu}(v_1,v_2)$ is large for some sufficiently small $\mu$. As this definition involves the vectors $v_1,v_2$ on only $(1-2\mu)n$ coordinates, we first need to use the incompressibility condition to show that these restricted vectors are linearly independent when restricted to any interval $I\subset[n]$, $|I|=(1-2\mu)n$.

\begin{fact}\label{fact4.123}
    To prove Theorem \ref{mainquasiramdonmess2}, we only need to prove the theorem for vector pairs $(v_1,v_2)$ satisfying the conditions stated in (1)-(3) of Corollary \ref{corollary2.678}, and moreover satisfies the following condition: $v_2=cv_1+\epsilon_1v_3$ for some $v_3\in\mathbb{S}^{n-1}$ orthogonal to $v_1$, and that for any $I\subset [n]$ with $|I|\geq (1-2\mu)n$, we have the following estimates for all  $\mu\in(0,\kappa_3)$:
    $$
\|(v_1)_I\|_2,\|(v_3)_I\|_2\geq \kappa_4,\quad\text{ more generally } \forall \theta\in\mathbb{R},\|(\cos\theta v_1+\sin\theta v_3)_I\|_2\geq\kappa_4,
    $$ where  $\kappa_3,\kappa_4\in(0,1)$ are two fixed constants depending only on $B$.
\end{fact}

\begin{proof}
We only need to work on the event stated in Lemma \ref{smalllemma2.5bound}, that is, with probability $1-\exp(-\Omega(n))$, any linear combination $\theta_1v_1+\theta_3v_3$ satisfying that $\theta_1,\theta_3=O(1)$ and $\|\theta_1v_1+\theta_3v_3\|_2=1$ is $(\delta,\rho)$- incompressible.
Now take any $\theta\in\mathbb{R}$, by orthogonality we have $\|\cos\theta v_1+\sin\theta v_3\|_2=1$. By the incompressibility event we work on, $\cos\theta v_1+\sin \theta v_3$ is again $(\delta,\rho)$-incompressible. Then by Fact \ref{fact5.67}, we can find $ \kappa_3,\kappa_4\in(0,1)$ such that the lower bound on $\|(cos\theta v_1+\sin\theta v_3)_I\|_2$ holds for any such interval $I\subset[n]$ with $|I|\geq (1-2\mu)n$.
\end{proof}

By fact \ref{fact4.123}, for any interval $I$ with $|I|\geq (1-2\mu)n$, let $C_I^1,C_I^3$ be any two constants that satisfy $\|(C_I^1v_1+C_I^3v_3)_I\|_2=1$, then we must have \begin{equation}\label{wemusthave}\frac{1}{2}\leq\max(C_I^1,C_I^3)\leq 1/\kappa_4.\end{equation}

\begin{reduction}\label{reduction1.111}
(A family of reductions) We first let $c_*>0$ be the constant given in Lemma \ref{smalllemma2.5bound} and we abbreviate $\hat{D}_{\alpha,\gamma,\mu}(v_1,v_2)$ for $\hat{D}_{\alpha,\gamma,\mu}^{c_*}(v_1,v_2)$. Take $c_\Sigma\in(0, \min(\frac{c_*}{4},2^{-5}))$.
\begin{enumerate}
\item By Corollary \ref{corollary2.678}, we only need to consider those $(v_1,v_2)\in\mathcal{P}_{\epsilon_1}(c)$ for some $c,\epsilon_1\in\mathbb{R}$ satisfying $\epsilon_1^2+c^2=1$, and those $(v_1,v_2)\in\mathcal{P}_{\epsilon_1}^0(c)$ for some $|\epsilon_1|\leq 1$. However, by our requirement $\|\operatorname{Proj}_{v_1^\perp}(v_2)\|_2\geq e^{-c_*n/2}$ in Theorem \ref{mainquasiramdonmess2}, we only need to consider those subsets where $\epsilon_1\geq e^{-c_*n/2}$.
\item We can take a $2^{-n}$-net for $c,\epsilon_1$ in item (1). Take any $(v_1,v_2)\in\mathcal{P}_{\epsilon_1}(c)$ with $c^2+\epsilon_1^2=1$, and take $c',\epsilon_1'$ from the net such that $|c-c'|\leq 2^{-n},|\epsilon_1-\epsilon_1'|\leq 2^{-n}$. Then from the definition $v_2=cv_1+\epsilon r$ for some $r\in\mathbb{S}^{n-1}$ orthogonal to $v_1$, we define $v_2'=c'v_1+\epsilon_1'r$ and $v_1'=v_1$. Then $(v_1',v_2')\in\mathcal{P}_{\epsilon_1'}(c')$ and $\|v_2-v_2'\|_2\leq 2^{-n+2}$. 

    \item 
By the definition of the subvector LCD $\hat{D}_{\alpha,\gamma,\mu}(v_1,v_2)$, to prove Theorem \ref{mainquasiramdonmess2} we only need to rule out vector pairs $(v_1,v_3)$ such that for some $I\subset[n]$, $|I|\geq (1-2\mu)n$ and for some $C_I^1,C_I^3$ satisfying \eqref{wemusthave}, we have $D_{\alpha,\gamma}((C_I^1v_1+C_I^3v_3)_I)\leq e^{c_\Sigma n}$. That is, we do not need $C_I^1,C_I^3$ to satisfy the strict relation $\|(C_I^1v_1+C_I^3v_3)_I\|_2=1$.

\item We can take a $2^{-n}$-net for $C_I^1,C_I^3$ with \eqref{wemusthave} because the definition of LCD is stable under a small perturbation: let ${C_I^{1'}},{C_I^{3'}}$  satisfy $|C_I^{1'}-C_I^1|\leq 2^{-n},|C_I^{3'}-C_I^3|\leq 2^{-n}$, then $D_{\alpha,\gamma}((C_I^1v_1+C_I^3v_3)_I)\leq e^{c_\Sigma n}$ implies that $D_{2\alpha,2\gamma}(C_I^{1'}v_1+C_I^{3'}v_3)\leq e^{c_\Sigma n}$.
\item For similar reasons, it is not hard to check that if $\hat{D}_{\alpha,\gamma,\mu}(v_1,v_2)\leq e^{c_\Sigma n}$ and $(v_1',v_2')$ approximates $(v_1,v_2)$ in item (2), then $\hat{D}_{2\alpha,2\gamma,\mu}(v_1',v_2')\leq e^{c_\Sigma n}$.
\item We take a union bound over all disjoint $D_1,D_2,D_3\subset[n]$ of cardinality $c_0^2n/12$ each.
\end{enumerate}
\end{reduction}
We can now reduce the proof of  Theorem \ref{mainquasiramdonmess2} to the following:
\begin{corollary}\label{corollary1streduction}
    To prove  Theorem \ref{mainquasiramdonmess2}, we only need the following: let $\mathcal{K}$ denote the event $\mathcal{K}:=\{\|A_n\|_{op}\leq4\sqrt{n}\},$ then we can find $\alpha,\gamma,c_\Sigma\in(0,1)$ such that for any $w_0\in\mathbb{S}^{n-1}$,
    $$\begin{aligned}&\sup_{\epsilon_1,c,I,C_I^1,C_I^3,s_1,t_1,s_2,t_2}
\mathbb{P}^\mathcal{K}
(\exists(v_1,v_2)\in\mathcal{P}_{\epsilon_1}(c)\cup\mathcal{P}_{\epsilon_1}^0(c)\text{ associated with } D=[c_0^2n/4]:\\&
D_{\alpha,\gamma}((C_I^1v_1+C_I^3v_3)_I)\leq e^{c_\Sigma n}, \|(A_n-s_iI_n)v_i-t_iw_0\|_2\leq 2^{-n/2}\forall i=1,2)\leq 2^{-500n},
    \end{aligned}$$ with the supremum over $1/2\leq c^2+\epsilon_1^2\leq 3/2$, $\epsilon_1\geq e^{-c_*n/2}$, over $I\subset[n]$ with $|I|\geq (1-2\mu)n$, over $C_I^1,C_I^3$ satisfying \eqref{wemusthave}, and over $s_i\in[-4\sqrt{n},4\sqrt{n}]$ and $t_i\in[-8\sqrt{n},8\sqrt{n}]$, i=1,2.
\end{corollary}
This corollary follows from using the $2^{-n}$-net for $\epsilon_1,c,C_I^1,C_I^3$; using the $2^{-n}$-net for $s_1,t_1,s_2,t_2$ in Reduction \ref{reduction1.111}; taking the union of $D$; and finally using $\|A_n\|_{op}\leq4\sqrt{n}$. 

In the following we always let $\mathcal{P}_{\epsilon_1}(c)$ and $\mathcal{P}_{\epsilon_1}^0(c)$ be associated with $D=[c_0^2n/4]$. Next we rescale the norm of $C_I^1,C_I^3$ to prepare for the application of Theorem \ref{mainresultchapter3}.

\begin{corollary}\label{2ndreductionfirstsec} We may decrease the value of $c_*$ in Lemma \ref{smalllemma2.5bound} so that $c_*<10^{-5}$.
     Then to prove  Theorem \ref{mainquasiramdonmess2}, we only need to prove the following: we can find $\alpha,\gamma,c_\Sigma\in(0,1)$ sufficiently small such that 
     $$\begin{aligned}
&\sup_{\epsilon_1,c,I,C_I^1,C_I^3,s_1,s_2}\sup_{w_0^1,w_0^2\in\mathbb{R}^n}\mathbb{P}^\mathcal{K}(\exists(v_1,v_2)\in\mathcal{P}_{\epsilon_1}(c)\cup\mathcal{P}_{\epsilon_1}^0(c): D_{\alpha,\gamma}((C_I^1v_1+C_I^3v_3)_I)\leq e^{c_\Sigma n},\\&
\|(A_n-s_1I_n)v_1-w_0^1\|_2\leq 2^{-n/3},\quad\\& \|(A_n-s_2I_n)(\underline{C}_I^1v_1+\underline{C}_I^3v_3)+(s_2-s_1)\underline{C}_I^1v_1-w_0^2\|_2\leq 2^{-n/3})\leq 2^{-600n},
\end{aligned}$$
where we define $\underline{C}_I^1=\frac{C_I^1}{\sqrt{(C_I^1)^2+(C_I^3)^2}}$ and $\underline{C}_I^3=\frac{C_I^3}{\sqrt{(C_I^1)^2+(C_I^3)^2}}$.

\end{corollary}
\begin{proof}
    The reduction follows from taking a reformulation of the small ball probability in Corollary \ref{corollary1streduction}, whee we use the assumption $\epsilon_1\geq e^{-10^{-5}n}$ and we use the two-sided bound in \eqref{wemusthave} which gives $|C_I^1|,|C_I^3|\leq 1/\kappa_4$ so we can for instance bound $\mathcal{L}(\epsilon_1v_3,2^{-n/2})\leq\mathcal{L}({\underline{C}}_I^3v_3,2^{-n/3})$. We also replace the supremum over $t_1,t_2\in\mathbb{R},w_0\in\mathbb{R}^n$ by the supremum over any two vectors $w_0^1,w_0^2\in\mathbb{R}^n$.

\end{proof}

From now on, we fix a small enough value of $c_0\in(0,1)$ so that Theorem \ref{mainresultchapter3} and Proposition \ref{proposition717717} can be applied, and we fix $c_*$ satisfying Lemma \ref{smalllemma2.5bound} such that $c_*\in(0,10^{-5})$. Let $\eta=\exp(-c_\Sigma n)$ for some $c_\Sigma>0$ small enough and take $j_0$ be the largest integer such that $2^{j_0}\eta\leq\kappa_0/8$.

The probability in Corollary \ref{2ndreductionfirstsec} will be evaluated in two steps. In the first step we consider vector pairs $(v_1,v_3)$ where the threshold function $\tau_{L,\underline{C}_I^3}(v_1,\underline{C}_I^1v_1+\underline{C}_I^3v_3)\geq e^{-c_\Sigma n}$.

\begin{Proposition}\label{proposition4.21067}
Fix two vectors $w_0^1,w_0^2\in\mathbb{R}^n$ and $\underline{\Delta}\in\mathbb{R}$ a constant satisfying $|\underline{\Delta}|\leq 8\sqrt{n}$. We can choose $c_\Sigma>0$ small enough and choose $L>1$ large enough depending only on $B$ so that for each $j\in[j_0]$ and $\epsilon_1\geq e^{-c_*n/2}$, the following two estimates hold:
    \begin{equation}\begin{aligned}\label{first100bound}
\mathbb{P}^\mathcal{K}&(\exists(v_1,v_2)\in\Sigma_{2^j\eta,\epsilon_1}(c):\|(A_n-s_1I_n)v_1-w_0^1\|_2\leq 2^{-n/3},\\&\|(A_n-s_2I_n)v_2-\underline{\Delta} v_1-w_0^2\|_2\leq 2^{-n/3})\leq 2^{-700n},
\end{aligned}    \end{equation}
\begin{equation}\label{secondrelation2nd}\begin{aligned}
&\mathbb{P}^\mathcal{K}(\exists(v_1,v_2)\in\Sigma_{2^j\eta,0}(c):\|(A_n-s_1I_n)v_1-w_0^1\|_2\leq 2^{-n/3}, \\&\|(A_n-s_2I_n)v_2-\underline{\Delta} v_1-w_0^2\|_2\leq 2^{-n/3})\leq 2^{-700n},\end{aligned}
    \end{equation} for any fixed values of $\epsilon_1,c,s_1,s_2$ in the range specified in Corollary \ref{corollary1streduction}.
\end{Proposition}

\begin{proof}
We begin with the first claim and again write $\epsilon=2^j\eta$. We first assume that  $\epsilon/\epsilon_1$ satisfy the bound \eqref{satisfy3.14wes}. Then by Proposition \ref{proposition102}, for $(v_1,v_2)\in\Sigma_{\epsilon,\epsilon_1}(c)$ we can find $(\underline{v}_1,\underline{v}_2)\in \mathcal{N}_{\epsilon,\epsilon_1}(c)$ such that $\|v_i-\underline{v}_i\|_\infty\leq 12\epsilon n^{-1/2},i=1,2$. Then using $\|A_n\|_{op}\leq 4\sqrt{n}$, we get that whenever $c_\Sigma>0$ is not too large, 
\begin{equation}\label{howdoweinclude?}\begin{aligned}
&\{\|(A_n-s_1I_n)v_1-w_0^1\|_2\leq 2^{-n/3},\|(A_n-s_2I_n)v_2-\underline{\Delta} v_1-w_0^2\|_2\leq 2^{-n/3}\}\\&\subset\{\|(A_n-s_1I_n)\underline{v}_1-w_0^1\|_2\leq 300\epsilon\sqrt{n},\|(A_n-s_2I_n)\underline{v}_2-\underline{\Delta}\cdot \underline{v}_1-w_0^2\|_2\leq 300\epsilon\sqrt{n}\}.\end{aligned}
\end{equation}

By properties of the net $\mathcal{N}_{\epsilon,\epsilon_1}(c)$ and using Fact \ref{fact4.21} to change $\epsilon$ to $300\epsilon$, we have that $$\sup_{(\underline{v}_1,\underline{v}_2)\in\mathcal{N}_{\epsilon,\epsilon_1}(c)}
\mathbb{P}_A^\mathcal{K}(\|((A_n-s_1I_n)\underline{v}_1-w_0^1,(A_n-s_2I_n)\underline{v}_2-\underline{\Delta}\cdot  \underline{v}_1-w_0^2)\|_2\leq 300\epsilon\sqrt{n})\leq (2^{60}\frac{L^2\epsilon^2}{\epsilon_1})^n.$$     
By Theorem \ref{mainresultchapter3}, we have, whenever $c_\Sigma$ is not too large, the bound for $|\mathcal{N}_{\epsilon,\epsilon_1}(c)|$: 

 $$
|\mathcal{N}_{\epsilon,\epsilon_1}(c)|\leq(\frac{C_B}{c_0^{16}L^4\epsilon^2/\epsilon_1})^n,
    $$

     Then the probability \eqref{first100bound} is upper bounded by 
$$
(\frac{C_B}{c_0^{16}L^4\epsilon^2/\epsilon_1})^n(2^{60}L^2\epsilon^2/\epsilon_1)^n\leq 2^{-700n},
$$
whenever we take $L>0$ sufficiently large. This holds for all $j\in[j_0]$, so that taking a union bound over $j\in[j_0]$ and summing the probability over $j$ completes the proof.

When the assumption in \eqref{satisfy3.14wes} is not satisfied, we instead apply Proposition \ref{verysmallepsilon1}, which yields that for $(v_1,v_2)\in\Sigma_{\epsilon,\epsilon_1}(c)$ we can find $(\underline{v}_1,\underline{v}_2)\in \mathcal{N}_{\epsilon_1,0}(c)$ such that $\|v_i-\underline{v}_i\|_\infty\leq 12\epsilon_1 n^{-1/2},i=1,2$. Then using $\|A_n\|_{op}\leq 4\sqrt{n}$, we get that whenever $c_\Sigma>0$ is small enough, the inclusion \eqref{howdoweinclude?} still holds but with $\epsilon_1$ in place of $\epsilon$ in the second line of \eqref{howdoweinclude?}.

By properties of the net $\mathcal{N}_{\epsilon_1,0}(c)$ and using Fact \ref{fact4.21} to change $\epsilon$ to $300\epsilon$, we have that $$\sup_{(\underline{v}_1,\underline{v}_2)\in\mathcal{N}_{\epsilon_1,0}(c)}
\mathbb{P}_A^\mathcal{K}(\|((A_n-s_1I_n)\underline{v}_1-w_0^1,(A_n-s_2I_n)\underline{v}_2-\underline{\Delta} \underline{v}_1-w_0^2)\|_2\leq 300\epsilon_1\sqrt{n})\leq (2^{60}L^2\epsilon_1)^n.$$
By Proposition \ref{proposition717717} we have the following upper bound for $|\mathcal{N}_{\epsilon_1,0}(c)|$: 

 $$
|\mathcal{N}_{\epsilon_1,0}(c)|\leq(\frac{C_B}{c_0^{22}L^4\epsilon_1})^n.
    $$

     Then the probability \eqref{first100bound} in this case (when \eqref{satisfy3.14wes} fails) is upper bounded by 
$$
(\frac{C_B}{c_0^{22}L^4\epsilon_1})^n(2^{60}L^2\epsilon_1)^n\leq 2^{-700n},
$$
whenever we take $L>0$ sufficiently large.

The proof for \eqref{secondrelation2nd} is very similar to the proof of the previous (second) case here, and is thus omitted.

\end{proof}

We can now reduce the proof of Theorem \ref{mainquasiramdonmess2} to the following version:

\begin{corollary}\label{thirdclasscordasg}
    The proof of Theorem \ref{mainquasiramdonmess2} can be reduced to the following estimate: for a given constant $c_\Sigma>0$ and a (large, already fixed) $L>1$,
there exist $\alpha',\gamma,\mu\in(0,1)$ and $c_\Sigma'>0$ such that 
     $$\begin{aligned}
&\sup_{\epsilon_1,c,I,C_I^1,C_I^3,s_1,s_2}\sup_{w_0^1,w_0^2\in\mathbb{R}^n}\mathbb{P}^\mathcal{K}(\exists(v_1,v_2)\in\mathcal{P}_{\epsilon_1}(c)\cup\mathcal{P}_{\epsilon_1}^0(c): D_{\alpha',\gamma}((C_I^1v_1+C_I^3v_3)_I)\leq e^{c_\Sigma' n},\\&
\tau_{L,\underline{C}_I^3}(v_1,\underline{C}_I^1v_1+\underline{C}_I^3v_3)\leq e^{-c_\Sigma n} \text{ if }|\underline{C}_I^3|\geq e^{-c_\Sigma n},\quad\tau_{L,0}(v_1,\underline{C}_I^1v_1+\underline{C}_I^3v_3\left.\right)\leq e^{-c_\Sigma n} \text{else},
\\&
\|(A_n-s_1I_n)v_1-w_0^1\|_2\leq 2^{-n/3},\quad\\& \|(A_n-s_2I_n)(\underline{C}_I^1v_1+\underline{C}_I^3v_3)+(s_2-s_1)\underline{C}_I^1v_1-w_0^2\|_2\leq 2^{-n/3})\leq 2^{-600n},
\end{aligned}$$
where the supremum is over parameters of the same range as in Corollary \ref{corollary1streduction} and \ref{2ndreductionfirstsec}.
 \end{corollary}
\begin{proof}
This follows from combining Corollary \ref{2ndreductionfirstsec} with Proposition \ref{proposition4.21067}, where we use the latter result to rule out vectors whose threshold function is larger than $e^{-c_\Sigma n}$ by a dyadic decomposition. We also take $L>1$ large enough with respect to other constants, so that the value of the threshold function $\tau_{L,\epsilon_1}(v,w)$ and $\tau_{L,0}(v,w)$ has value smaller than $\kappa_0/8$ for any vector pair $(v,w)$.
\end{proof}

The final step is to find a net for vectors $\underline{C}_I^1v_1+\underline{C}_I^3v_3$ with essential LCD at most $e^{c_\Sigma' n}$. Recall that by Fact \ref{fact4.123} and its proof, we always have $(C_I^1v_1+C_I^3v_3)_I$ is $(\delta',\rho')$-incompressible for some $\delta',\rho'\in(0,1)$. Then by \cite{rudelson2008littlewood} or \cite{campos2024least}, Fact III.6 we can choose $\gamma'>0$ small enough so that for any $\alpha'>0$, $D_{\alpha',\gamma'}((C_I^1v_1+C_I^3v_3)_I)\geq (2\kappa_6)^{-1}\sqrt{n}$ for some $\kappa_6>0$. 

We then construct a net for vectors with intermediate LCD:

\begin{fact}\label{factlcdsgag}
    For any $\alpha'>0$ and $\epsilon\leq Kn^{-1/2}$ there exists an $160\epsilon\sqrt{\alpha'n}$- net $G_\epsilon'$ for those vectors $C_I^1v_1+C_I^3v_3$ where $D_{\alpha',\gamma}((C_I^1v_1+C_I^3v_3)_I)\in[(4\epsilon)^{-1},(2\epsilon)^{-1}]$ with cardinality     \begin{equation}\label{line1234}|G_\epsilon'|\leq (\frac{192K/\kappa_4}{(\alpha')^{\mu}\epsilon\sqrt{n}})^n\cdot\frac{3}{\kappa_4\epsilon}.
    \end{equation}
\end{fact}

\begin{proof}
    By assumption \eqref{wemusthave}, we have that $\frac{1}{3}\leq\|C_I^1v_1+C_I^3v_3\|_2\leq 3/\kappa_4$. We find an $\epsilon$-net $\mathcal{N}_z$ for $[\frac{1}{3},3/\kappa_4]$. Then consider the following subset
    $$
G_\epsilon:=\{\frac{zp}{\|p\|_2}:z\in\mathcal{N}_z,p\in(\mathbb{Z}^I\oplus\sqrt{\alpha}\mathbb{Z}^{I^c})\cap B_n(0,6/(\epsilon\kappa_4))\setminus\{0\}\}. 
    $$ By definition of essential LCD, let $D=D_{\alpha',\gamma}((C_I^1v_1+C_I^3v_3)_I)$, then we can find $p_I\in \mathbb{Z}^I\cap B_n(0,3/(\epsilon\kappa_4))$ with $\|D(C_I^1v_1+C_I^3V_3)_I-p_I\|_2\leq\sqrt{\alpha'n}$ and $p_I\neq 0$. We can also find $p_{I^c}\in\sqrt{\alpha'}\mathbb{Z}^{I^c}\cap B_n(0,3/(\epsilon\kappa_4))$ with $\|D(C_I^1v_1+C_I^3V_3)_{I^c}-p_{I^c}\|_2\leq\sqrt{\alpha'n}$. Taking $p=p_I\oplus p_{I^c}$ and taking $z\in\mathcal{N}_1$ be such that $|\|C_I^1v_1+C_I^3v_3\|_2-z|\leq \epsilon$, then by triangle inequality we have $|Dz-\|p\|_2|\leq 2\sqrt{\alpha'n}+1$. Then by triangle inequality again, we have
    $$\begin{aligned}&
\|C_I^1v_1+C_I^3v_3-\frac{zp}{\|p\|_2}\|_2\leq\|C_I^1v_1+C_I^3v_3-\frac{p}{D}\|_2+\|\frac{p}{D}-\frac{zp}{\|p\|_2}\|_2\\&\leq  2D^{-1}\sqrt{\alpha' n}+\|\frac{p}{D}-\frac{zp}{\|p\|_2}\|_2\leq 5D^{-1}\sqrt{\alpha'n}
\leq 80\epsilon\sqrt{\alpha'n}.\end{aligned}$$

By a volumetric argument (see \cite{rudelson2008littlewood}), the cardinality of $G_\epsilon$ is bounded by $(\frac{192K\kappa_4}{(\alpha')^{\mu}\epsilon\sqrt{n}})^n\cdot\frac{3}{\kappa_4\epsilon}$. Although this subset $G_\epsilon$ is not a net of $C_I^1v_1+C_I^3v_3$, it can easily be modified to be a $160\epsilon\sqrt{\alpha'n}$- net $G_\epsilon'$ of $C_I^1v_1+C_I^3v_3$ of the same cardinality. This completes the proof.
\end{proof}

Now we can complete the proof of Theorem \ref{mainquasiramdonmess2}.

\begin{proof}[\proofname\ of Theorem 
\ref{mainquasiramdonmess2}] We only need to justify the claimed estimates in Corollary \ref{thirdclasscordasg}.
Consider the first case in Corollary \ref{thirdclasscordasg} where $|C_I^3|\geq e^{-c_\Sigma n}$.
We take a dyadic decomposition for the possible range of essential LCD and consider each level set of LCD value $$\cup_{j=1}^{j_1}\Sigma_{\alpha',\gamma}(j),\quad \Sigma_{\alpha',\gamma}(j):=\{(v_1,v_3):{D_{\alpha',\gamma}((C_I^1v_1+C_I^3v_3)_I)\in[2^j\eta_0,2^{j+1}\eta_0]\}},$$ with $\eta_0=(2\kappa_6)^{-1}\sqrt{n}$ and $j_1\in\mathbb{N}_+$ satisfies $2^{j_1}\eta_0\geq e^{c_\Sigma' n}\geq 2^{j_1-1}\eta_0$. We fix a $j\in[j_1]$ and let $\epsilon_j$ satisfy $(2\epsilon_j)^{-1}=2^{j+1}\eta_0$. By Fact \ref{factlcdsgag}, the subset $\Sigma_{\alpha',\gamma}(j)$ has a $160\epsilon\sqrt{\alpha'n}$-net $G_{\epsilon_j}'$ of cardinality bounded by $(\frac{192K/\kappa_4}{(\alpha')^{\mu}\epsilon_j\sqrt{n}})^n\cdot\frac{3}{\kappa_4\epsilon_j}.$

Now we consider two possible cases: (I) when $|C_I^1|\geq 0.5$ and (II) when $|C_I^1|\leq 0.5$ but $|C_I^3|\geq 0.5$. By assumption \eqref{wemusthave}, we are in at least one of the two cases. In case (I), we can find a trivial $\sqrt{\alpha'n}\epsilon_j$- net for $C_I^3v_3$ of cardinality bounded by $(\frac{100 K/\kappa_4}{(\epsilon_j/C_I^3)\sqrt{\alpha'n}})^n$ which we denote by $G_{\epsilon_j}''$, and it is not hard to check using $|C_I^1|\geq 0.5$ that the product of these two nets can yield a $1000\sqrt{\alpha'n}\epsilon_j$-net for $v_1\in\mathbb{S}^{n-1}$: first note that we can find $w_1\in G_{\epsilon_j}'$, $w_2\in G_{\epsilon_j}''$ such that 

$$\|C_I^1v_1+C_I^3v_3-w_1\|_2\leq 160\sqrt{\alpha'n}\epsilon_j,\quad \|C_I^3v_3-w_2\|_2\leq \sqrt{\alpha'n}\epsilon_j,$$ then by triangle inequality 
$$
\|v_1-(w_1-w_2)/(C_I^1)\|_2\leq 400\sqrt{\alpha'n}\epsilon_j,
$$ and by a standard argument, we can modify the product of the nets $G_{\epsilon_j}'\times G_{\epsilon_j}''$ to be a $1000\sqrt{\alpha'n}\epsilon_j$- net $\mathcal{N}_{(I)}^j$ for $(v_1,C_I^1v_1+C_I^3v_3)$ of cardinality at most $|G_{\epsilon_j}'|\cdot|G_{\epsilon_j}''|$.

Similarly in case (II) which is indeed much simpler, we also use a trivial $\sqrt{\alpha' n}\epsilon_j$-net $G_{\epsilon_j}'''$ for $v_1$ and we can check using $|C_I^3|\geq 0.5$ that the product of these two nets can yield a $1000\sqrt{\alpha'n}\epsilon_j$-net $\mathcal{N}_{(II)}^j$ for $(v_1,C_I^1v_1+C_I^3v_3)$ with cardinality bounded by $|G_{\epsilon_j}'|\cdot |G_{\epsilon_j}'''|$.

Then in case (I), for $(v_1,v_3)\in\Sigma_{\alpha',\gamma}(j)$, we take $(\hat{v}_1,C_I^1\hat{v}_1+C_I^3\hat{v}_3)\in\mathcal{N}_{(I)}^j$ approximating $v_1,C_I^1v_1+C_I^3v_3$, then on the event $\|A_n\|_{op}\leq 4\sqrt{n}$ (and for $c_\Sigma'>0$ small enough), $$\begin{aligned}&\{\|(A-s_1)v_1-w_0^1\|_2,\|(A-s_2)(\underline{C}_I^1v_1+\underline{C}_I^3v_3)+(s_2-s_1)\underline{C}_I^1v_1-w_0^2\|_2\leq 2^{-n/3})\}\subseteq\\&
\{   \|(A-s_1)\hat{v}_1-w_0^1\|_2,\|(A-s_2)(\underline{C}_I^1\hat{v}_1+\underline{C}_I^3\hat{v}_3)+(s_2-s_1)\underline{C}_I^1\hat{v}_1-w_0^2\|_2\leq 
5000\sqrt{\alpha'n^2}\epsilon_j\},\end{aligned}
$$ so that the probability in question for Corollary \ref{thirdclasscordasg} for vectors in the $j$-th level set of LCD can be upper bounded by 
\begin{equation}
    \label{atdians}
(\frac{192K/\kappa_4}{(\alpha')^\mu\epsilon_j\sqrt{n}})^n\cdot(\frac{100K/\kappa_4}{(\epsilon_j/C_I^3)\sqrt{\alpha'n}})^n\cdot\frac{3}{\kappa_4\epsilon_j}\cdot (4L^210^8\alpha'n\epsilon_j^2/C_I^3)^n
\leq 2^{-700n},\end{equation} 
 where the first three terms is the cardinality of the net $\mathcal{N}_{(I)}^j$, and the fourth term arises from the threshold function being small: $\tau_{L,\underline{C}_I^3}(\hat{v}_1,\underline{C}_I^1\hat{v}_1+\underline{C}_I^3\hat{v}_3)\leq e^{-c_\Sigma n}$ (those $(\hat{v}_1,\hat{v}_3)$ not satisfying this are ruled out in Corollary \ref{thirdclasscordasg}). 
 The final inequality follows from taking $\alpha'>0$ small enough.
 For this chosen $\alpha'$, we take $c_\Sigma'>0$ to be small enough so that $\sqrt{\alpha'n}\exp({-c_\Sigma' n})\geq \exp(-c_\Sigma n)$, allowing us to use the information of the threshold function at \eqref{atdians}. The case in (II) is similar. Taking the summation over $j$ completes the proof. 

For the remaining case where $|C_I^3|\leq e^{-c_\Sigma n}$, an $\sqrt{\alpha'n}\epsilon_j$-net for $C_I^1v_1+C_I^3v_3$ yields a $2\sqrt{\alpha'n}\epsilon_j$ net for $C_I^1v_1$ when $\sqrt{\alpha'n}\epsilon_j\geq e^{-c_\Sigma n}$ ( choose $c_\Sigma'$ as previously, so this holds for all $j\in[j_1]$), which yields a $4\sqrt{\alpha'n}\epsilon_j$-net for $v_1$ because in this case $|C_I^1|\geq 0.5$. No additional nets need to be constructed for the $C_I^3v_3$ component. Then we use the information on the threshold function $\tau_{L,0}(v_1,\underline{C}_I^1v_1+\underline{C}_I^3v_3)$ to complete the proof similarly as in the first case.

\end{proof}

\section{Reducing singular value to distance estimates}\label{sectionfiveth}

The aim of this section is to reduce the lower bound estimate on two singular values to a random distance problem. This is a typical step in the invertibility via distance approach in \cite{rudelson2008littlewood}. For symmetric random matrices, the reduction is more difficult and even more so when we handle two locations. The main result of the section is the following:  

\begin{Proposition}\label{finalfuckpropositionga} Let $A_{n+1}\sim\operatorname{Sym}_{n+1}(\zeta)$, $X\sim\operatorname{Col}_n(\zeta)$  and $\lambda_1,\lambda_2\in\mathbb{R}$ . For any $p>1$ and $\delta_1,\delta_2>0$, we have the following estimate:
    $$\begin{aligned}
           &\mathbb{P}(\sigma_{min}(A_{n+1}-\lambda_i I_{n+1})\leq\delta_i n^{-1/2},i=1,2)\lesssim \delta_1\delta_2+e^{-\Omega(n)}\\&+\sup_{r_1,r_2\in\mathbb{R}}\mathbb{P}_{A_n,X}\left(
    \frac{|\langle (A_n-\lambda_i I_n)^{-1}X,X\rangle-r_i|}{\|(A_n-\lambda_i I_n)^{-1}X\|_2}\leq\delta_i,i=1,2,\frac{\mu_1(\lambda_1)\mu_1(\lambda_2)}{n}\leq(\delta_1\delta_2)^{-p}
    \right),\end{aligned}$$ where we denote $\mu_1(\lambda_i):=\sigma_{max}((A_n-\lambda_i I_n)^{-1})$ for both $i=1,2$.
    \end{Proposition}
    \begin{remark}

To prove Theorem \ref{Theorem1.1}, we can assume that $\delta_1,\delta_2\in [e^{-c_*n/2},1]$ for any given $c_*>0$. If any one of the $\delta_i$ is larger than 1 then we are reduced to the one-location estimate (see Proposition \ref{proposition6.666}). Otherwise, if one of them is less than $e^{-c_*n/2}$, then again by the one-location estimate, Proposition \ref{proposition6.666}, the resulting probability estimate is exponentially small and thus can be squeezed in the $e^{-\Omega(n)}$ error term.
        
    We shall also always assume $A_n-\lambda_i I_n$ is invertible: this can be achieved by adding a very small Gaussian noise to $A_n$ which does not change the resulting estimates.
 \end{remark}
The fact that in the stated estimate we rule out the event $\{\mu_1(\lambda_1)\mu_1(\lambda_2)\geq n(\delta_1\delta_2)^{-p}\}$ is because our estimates in the forthcoming chapters will deteriorate on this event. Rather, the estimate on this event can be worked out fairly easily via a separate argument in Proposition \ref{proposition7.1}. This can be compared to the one-location case in \cite{campos2024least}, Lemma 6.1. In our two-location estimate, the event we take away depends on the product of $\mu_1(\lambda_1)$ and $\mu_1(\lambda_2)$, and also on a constant $p>1$ that will be set arbitrarily close to 1.

We first consider the ruled-out event  $\{\mu_1(\lambda_1)\mu_1(\lambda_2)\geq n(\delta_1\delta_2)^{-p}\}$.
For this purpose, we define for any $p\geq1$, 
$$
\mathcal{P}^p:=\left\{ \prod_{i=1}^2\sigma_{min}(A^{(j)}_{n+1}-\lambda_i I_n)\leq(\delta_1\delta_2)^p n^{-1}\text{ for at least } \frac{1}{4}n \text{ values } j\in[n+1]\right\},
$$
where $A_{n+1}^{(j)}$ is the principal submatrix of $A_{n+1}$ with the $j$-th row and column removed. 

The estimate on the event $\mathcal{P}^p$ is easier to handle, as shown in the following: 
\begin{Proposition}\label{proposition7.1} Fix $B>0$, $\zeta\in\Gamma_B$, and take $A_{n+1}\sim \operatorname{Sym}_{n+1}(\zeta)$. Then for any $p>1$,
   \begin{equation}\label{probpp} \mathbb{P}(\sigma_{min}(A_{n+1}-\lambda_i I_{n+1})\leq\delta_i n^{-1/2},i=1,2,\quad \mathcal{P}^p)\lesssim \delta_1\delta_2+e^{-\Omega(n)}.\end{equation}
\end{Proposition}

\subsection{Proof of Proposition \ref{proposition7.1}}
We need two auxiliary results. The first concerns the eigenvector delocalization of Wigner matrices, taken from Rudelson and Vershynin \cite{rudelson2016no}.

\begin{theorem}\label{veragaghagag}(\cite{rudelson2016no}) Fix $B>0$ and let $\zeta\in\Gamma_B$. Consider $A_n\sim\operatorname{Sym}_n(\zeta)$ and let $v$ be any eigenvector of $A_n$. Then there is some $c_2>0$ such that for any sufficiently small $c_1>0$, for $n$ sufficiently large,
$$
\mathbb{P}(|v_j|\geq (c_2c_1)^6 n^{-1/2}\text{ for at least }(1-c_1)n \text{ indices }j)\geq 1-e^{-c_1n}.
$$    
\end{theorem}

We also need the following geometric fact from \cite{campos2024least}, Fact 6.7.

\begin{fact}\label{fact7.3}
    Given any $n\times n$ real symmetric matrix $M$ and $\lambda$ as an eigenvector of $M$ with associated eigenvector $u$. Let $j\in[n]$ and assume that $\lambda'$ is an eigenvector of the principal minor $M^{(j)}$ corresponding to a unit eigenvector $v$. Then we have
$$
|\langle v,X^{(j)}\rangle|\leq |\lambda-\lambda'|/|u_j|,
$$ where $X^{(j)}$ denotes the $j$-th column of $M$ of which the $j$-th entry is removed.
\end{fact}
We also use a standard version of the inverse Littlewood-Offord theorem \cite{rudelson2008littlewood}.

\begin{theorem}\label{littlewood1stofford}
    Fix $B>0$, $\gamma,\alpha\in(0,1)$ and $v\in\mathbb{S}^{n-1}$ with $D_{\alpha,\gamma}(v)\geq c\epsilon^{-1}$. Let $X\sim\operatorname{Col}_n(\zeta)$ with $\zeta\in\Gamma_B$. Then 
    $$
\mathbb{P}(|\langle X,v\rangle|\leq\epsilon)\lesssim\epsilon+e^{-c\alpha n}
    $$ for some $c>0$ depending only on $B$ and $\gamma$.
\end{theorem}

Now we can prove Proposition \ref{proposition7.1}. We discuss the proof heuristics first. The event $\mathcal{P}^p$ can be (approximately, though not correctly) interpreted as claiming that both $\sigma_{min}(A_{n+1}^{(j)}-\lambda_i I_n)$ are smaller than $\delta_i$, $i=1,2$. This is included in the event that an eigenvalue of $A_{n+1}$ and an eigenvalue of $A_{n+1}^{(j)}$ have a difference upper bounded by $2\delta_i$, say. Then we use Fact \ref{fact7.3} and the inverse Littlewood-Offord theorem. However, our definition of $\mathcal{P}^p$ is about the product of least singular value at two locations, so we have to discretize and consider all the possible cases of the value they may take. Also, we are considering joint estimates for two locations which are not available before, so we first use the estimate for $A_n^{(j)}$ in one location and apply the one-point estimate (Proposition \ref{proposition6.666}), then use Fact \ref{fact7.3} at the other location. This procedure enables us to successfully decouple the two locations. 

\begin{proof}[\proofname\ of Proposition \ref{proposition7.1}] Let $Q_1$ be the event where the conclusions of Theorem \ref{veragaghagag} hold true, and we work on $Q_1$. We first find some value $j\in[n+1]$ such that the event in $\mathcal{P}^p$ holds for the subscript $j$ and the unit eigenvectors (denoted $u_1$ and $u_2$) of $A_{n+1}$ corresponding to the least singular values of $A_{n+1}-\lambda_1 I_{n+1}$ and $A_{n+1}-\lambda_2 I_{n+1}$ respectively, satisfy $|(u_1)_j|,|(u_2)_j|\geq (c_2c_1)^6n^{-1/2}$ for some fixed $c_1,c_2$. The fact that such $j$ exists among the  given $\frac{1}{4}n$ values of $j$ specified in $\mathcal{P}^p$ is guaranteed by Theorem \ref{veragaghagag} on the event $Q_1$, which holds with probability $1-e^{-\Omega(n)}$.

Now we would like to find a dyadic decomposition of $[(\delta_1\delta_2)^p,1]=\cup_{j=1}^{I_n}I_j$ into $I_n$ intervals, where we denote by $I_j=[I_{j,S},I_{j,L}]$ such that $I_{j,L}=2I_{j,S}$ and $I_{j,S}=I_{j+1,L}$ (except for $I_1$, where we do slightly differently). Then we have $I_n=O(\log(\delta_1\delta_2)^{-1})$ and we can decompose the event $\{\prod_{i=1}^2\sigma_{min}(A_{n}^{(j)}-\lambda_i I_n)\leq (\delta_1\delta_2)^p n^{-1}\}$ as a union of $(I_n)^2$ events, where each event has the form $\{\sigma_{min}(A_{n}^{(j)}-\lambda_1 I_n)\leq x_1n^{-1/2}$, $\sigma_{min}(A_{n}^{(j)}-\lambda_2 I_n)\leq x_2n^{-1/2}\}$ for two real numbers $x_1,x_2>0$ satisfying that $x_1x_2\leq 4(\delta_1\delta_2)^p$. Note that if for some $i\in\{1,2\}$ we already have $\sigma_{min}(A_{n}^{(j)}-\lambda_i I_n)\leq (\delta_1\delta_2)^p n^{-1/2}$ we may directly apply the one- location estimate, Proposition \ref{proposition6.666} to $A_n-\lambda_i I_n$ and the contribution of this case to the probability on the left hand side of \eqref{probpp} is at most
$$
\mathbb{P}(\sigma_{min}(A_n-\lambda_i I_n)\leq (\delta_1\delta_2)^pn^{-1.2})\lesssim (\delta_1\delta_2)^p+e^{-cn}
,$$and then use the fact $(\delta_1\delta_2)^p\leq\delta_1\delta_2$. Or, if for some $i$ we have $\sigma_{min}(A_{n}-\lambda_i I_n)\geq n^{-1/2}$ then for the other $\ell\neq i$ we must have $\sigma_{min}(A_{n}-\lambda_\ell I_n)\leq (\delta_1\delta_2)^p n^{-1/2}$, and we are reduced to the first case. Thus we can always assume that $x_1x_2\leq 4(\delta_1\delta_2)^p$ and that $x_1,x_2\in[(\delta_1\delta_2)^p,1]$.

Now we analyze each of these $(I_n)^2$ events $\{\sigma_{min}(A_{n}-\lambda_i I_n)\leq x_in^{-1/2}, i=1,2\}$.

There are two scenarios. We first assume that for at least one $i\in\{1,2\}$ we have $x_i\geq\delta_i$. Let us take $i=1$ without loss of generality, that is, $\sigma_{min}(A_{n+1}^{(j)}-\lambda_1 I_n)\leq x_1$. We first consider the event on location $\lambda_2$: $\{\sigma_{min}(A_{n+1}^{(j)}-\lambda_2 I_n)\leq x_2n^{-1/2}$\}. By Proposition \ref{proposition6.666}, we have the bound
\begin{equation}\label{secondcond}\mathbb{P}\left(\sigma_{min}(A_{n+1}^{(j)}-\lambda_2 I_n)\leq x_2n^{-1/2}\right)\lesssim x_2+e^{-\Omega(n)}.\end{equation}
Now we condition on this event and estimate the singular value at location $\lambda_1$. By assumption that $x_1\geq\delta_1$, we need only to bound
    \begin{equation}\label{conditionone}
\mathbb{P}\left(\sigma_{min}(A_{n+1}-\lambda_1 I_{n+1})\leq x_1n^{-1/2},\sigma_{min}(A_{n+1}^{(j)}-\lambda_1 I_{n})\leq x_1n^{-1/2}\mid \mathcal{E}_{2,p} 
\right) 
    \end{equation} where we are conditioning on the event 
    $$ \mathcal{E}_{2,p}:=\{\sigma_{min}(A_{n+1}^{(j)}-\lambda_2 I_n)\leq x_2n^{-1/2}\}.$$

To bound the probability of this event, let $v^j$ be the unit eigenvector associated with   the least singular value of $A_{n+1}^{(j)}-\lambda_1 I_n$. Let $Q_3$ denote the event $Q_3:=\{D_{\alpha,\gamma}(v^j)\geq e^{c_3n}\text{ for all }j\in[n+1]\}$, where we have $\mathbb{P}(Q_3^c)\leq e^{-\Omega(n)}$ by Lemma
\ref{mainquasirandomness1}.
Now we can use Fact \ref{fact7.3} since both $\sigma_{min}(A_{n+1}-\lambda_1 I_{n+1})$ and $\sigma_{min}(A_{n}^{(j)}-\lambda_1 I_{n})$ are smaller than $x_1n^{-1/2}$, and we also use the fact that the index $j\in[n+1]$ is chosen at the beginning of the proof such that $|(u_1)_j|\geq(c_1c_2)^6n^{-1/2}$.

Then we conclude by the inverse Littlewood-Offord theorem, Theorem \ref{littlewood1stofford} that
\begin{equation}
    \eqref{conditionone}\leq\mathbb{P}(|\langle X^{(j)},v\rangle|\leq 2(c_1c_2)^{-6}x_1)+\mathbb{P}(Q_1^c\cup Q_3^c)\lesssim x_1+e^{-\Omega(n)}.
\end{equation} Then we can combine this with \eqref{secondcond} to obtain 
$$\begin{aligned}&\mathbb{P}\left(\sigma_{min}(A_{n+1}-\lambda_i I_{n+1})\leq \delta_i n^{-1/2},\sigma_{min}(A_{n+1}^{(j)}-\lambda_i I_{n})\leq x_in^{1/2},i=1,.2 
\right)\\&\quad\quad\lesssim x_1x_2+e^{-\Omega(n)}\lesssim (\delta_1\delta_2)^p+e^{-\Omega(n)}.\end{aligned}$$

Then we are left with the remaining case where $x_1<\delta_1$ and $x_2<\delta_2$. Our assumption that $x_1x_2<(\delta_1\delta_2)^p$ implies that for some $i=1,2$ we have $x_i\leq\delta_i(\delta_1\delta_2)^{(p-1)/2}$. Let us say we have $i=2$ ($i=1$ is analogous), so $x_2\leq \delta_2(\delta_1\delta_2)^{(p-1)/2}$.

We first apply Proposition \ref{proposition6.666} and obtain the same estimate as in \eqref{secondcond}. We then apply fact \ref{fact7.3} and, on the event $Q_3$, apply Theorem \ref{littlewood1stofford} to obtain
\begin{equation}\begin{aligned}
&\mathbb{P}\left(\sigma_{min}(A_{n+1}-\lambda_1 I_{n+1})\leq \delta_1n^{-1/2},\sigma_{min}(A_{n+1}^{(j)}-\lambda_1 I_{n})\leq \delta_1n^{-1/2}\mid \mathcal{E}_{2,p} 
\right)\\& \lesssim \delta_1+e^{-\Omega(n)}
.\end{aligned}    \end{equation} Combining these two estimates, we have
$$\begin{aligned}&\mathbb{P}\left(\sigma_{min}(A_{n+1}-\lambda_i I_{n+1})\leq \delta_i n^{-1/2},\sigma_{min}(A_{n+1}^{(j)}-\lambda_i I_{n})\leq x_in^{1/2},i=1,.2 
\right)\\&\quad\quad\lesssim \delta_1x_2+e^{-\Omega(n)}\lesssim (\delta_1\delta_2)^{(p+1)/2}+e^{-\Omega(n)}.\end{aligned}$$
Finally, we take a union bound over all all the $(I_n)^2$ events (i.e. choices of $x_1,x_2)$, which leads to an entropy cost of $\log((\delta_1\delta_2)^{-1})^{C_d}\lesssim (\delta_1\delta_2)^{1-p}$. This completes the proof.
\end{proof}

\subsection{Proof of of Proposition \ref{finalfuckpropositionga}}

Now we have shown that in order to prove
 $$ \mathbb{P}(\sigma_{min}(A_{n+1}-\lambda_i I_{n+1})\leq\delta_i n^{-1/2},i=1,2)\lesssim \delta_1\delta_2+e^{-\Omega(n)},$$ it suffices to prove, for any $p>1$,
  \begin{equation}\label{what'fnesa} \mathbb{P}(\sigma_{min}(A_{n+1}-\lambda_i I_{n+1})\leq\delta_i n^{-1/2},i=1,2,\quad\wedge\quad (\mathcal{P}^p)^c)\lesssim \delta_1\delta_2+e^{-\Omega(n)}.\end{equation}  First, we recall a fundamental fact (see \cite{rudelson2008littlewood}):

  \begin{fact}
      For any $n\times n$ symmetric matrix, let $v$ be the unit vector corresponding to the least singular value of $M$, that is, $\|Mv\|_2=\sigma_{min}(M)$ and $\|v\|_2=1$. Then we have
      $$ \sigma_{min}(M)\geq|v_j|\cdot d_j(M)\quad\text{ for each }j\in[n],
      $$ where $d_j(M)$ is the distance of the $j$-th column of $M$ to the subspace spanned by the other $n-1$ columns of $M$.
  \end{fact}

The distance $d_j(M)$ can be computed as follows: (see \cite{vershynin2014invertibility}, Prop.5.1)
  \begin{fact}\label{fact2.78}
Let $A_{n+1}$ be an $n+1\times n+1$ symmetric matrix where $A_n$ is its principal submatrix with the first row/ column of $A_{n+1}$ removed. Let $a_{11}$ be the entry of $A_{n+1}$ at location (1,1) and $X$ the first row of $A_{n+1}$ with the fist entry removed. Then 
$$
d_1(A_{n+1})=\frac{|\langle A^{-1}X,X\rangle-a_{11}|}{\sqrt{1+\|A^{-1}X\|_2^2}}.
$$
  \end{fact}

The following lower bound for $\|A^{-1}X\|_2$ is available: (see \cite{vershynin2014invertibility}, Prop.8.2)

\begin{fact}\label{fact2.89}
    Let $B>0$, $\zeta\in\Gamma_B$ and $A_{n+1}\sim\operatorname{Sym}_{n+1}(\zeta)$. Fix any $\lambda\in[-4\sqrt{n},4\sqrt{n}]$. Then with probability at least $1-e^{-\Omega(n)}$, we have
    $$
\|(A_n-\lambda I_n)^{-1}X\|_2\geq \frac{1}{15}.
    $$
\end{fact}

Now we are in the place to prove the main result, Proposition \ref{finalfuckpropositionga}. The proof uses a standard first moment method to express smallness of singular value by smallness of certain random distance functions. The crucial new input here is that we have two different locations $\lambda_1,\lambda_2$, and we need to guarantee that their least singular vectors have a large enough overlapping support. This is guaranteed by the no-gaps delocalization result, Theorem \ref{veragaghagag}.
\begin{proof}[\proofname\ of Proposition \ref{finalfuckpropositionga}]
 Let the eigenvectors associated with the least singular value of $A_{n+1}-\lambda_i I_{n+1}$ be denoted by $u_i$ for each $i=1,2$. Let $Q_1^c$ be the rare event that, for some $i=1,2$, $u_i$ does not satisfy the conclusion of Theorem \ref{veragaghagag} where we take $c_1=\frac{1}{4}$. Then $\mathbb{P}(Q_1^c)=e^{-\Omega(n)}$. On $Q_1$, we have that for at least $\frac{1}{2}n$ choices of $j\in[n+1]$,
 \begin{equation}\label{fuckyousga}\sigma_{min}(A_{n+1}-\lambda_i I_{n+1})\geq c_3 n^{-1/2}\cdot d_j(A_{n+1}-\lambda_iI_{n+1})\text{ for both } i=1,2,\end{equation}
 where $c_3>0$ depends only on $B$. 

Furthermore, taking the intersection of these $\frac{1}{2}n$ values of $j$ with the $\frac{3}{4}n$ values of $j$ where the claim in $(\mathcal{P}^{p})^c$ does not hold, we get that on the event $(\mathcal{P}^{p})^c$, there are at least $\frac{1}{4}n$ values of $j\in[n+1]$ such that \eqref{fuckyousga} holds for such $j$ and moreover, 
$$
\prod_{i=1}^2 \sigma_{min}(A_{n+1}^{(j)}-\lambda_i I_{n})\geq (\delta_1\delta_2)^p n^{-1}.
$$

Now let $S\subset \{1,2,\cdots,n+1\}$ denote the collection of subscripts $j\in[n+1]$ such that the following event $\mathcal{H}_j$ holds (for some universal constant $C$) 
$$\mathcal{H}_j:=
\left\{d_j(A_{n+1}-\lambda_i I_{n+1})\leq C\delta_i\text{ for } i=1,2, \prod_{i=1}^2 \sigma_{min}(A_{n+1}^{(j)}-\lambda_i I_{n})\geq (\delta_1\delta_2)^p n^{-1}\right\}.
$$
Our previous discussions imply that
$$\left\{\sigma_{min}(A_{n+1}-\lambda_i I_{n+1})\leq\delta_i n^{-1/2},i=1,2\right\}\bigcap (\mathcal{P}^p)^c\bigcap Q_1\Rightarrow |S|>\frac{1}{4}n.$$

By a standard first moment estimate, $$\mathbb{P}(|S|>\frac{1}{4}n)\leq\frac{4}{n}\sum_{j=1}^{n+1}\mathbb{P}(\mathcal{H}_j).$$
Thus there exists some $j_*\in[n+1]$ such that $\mathbb{P}(|S|\geq\frac{1}{4}n)\leq 5 \mathbb{P}(\mathcal{H}_{j_*})$, and we can take $j_*=1$ without loss of generality.

In the final step, we use Fact \ref{fact2.78} to expand the distance $d_j(A_{n+1}-\lambda_i I_{n+1})$ and use the estimate in Fact \ref{fact2.89} to complete the proof.
\end{proof}

\section{Decoupling of the quadratic forms}
The main result of this section is to prove the following decoupling lemma. 

\begin{theorem}\label{theorem3.1}
  Let $\zeta\in \Gamma_B$ for some $B>0$. Let $A_n$ be any $n\times n$ symmetric matrix with $A_n\in\mathcal{E}$ (see Lemma \ref{mainquasirandomness1} for the event $\mathcal{E}$, and let $c_*>0$, $\mu>0$ be the two constants given there) and let $X\sim \operatorname{Col}_n(\zeta)$ be a random vector independent of $A_n$. Consider two fixed locations $\lambda_1,\lambda_2\in[-4\sqrt{n},4\sqrt{n}]\subset\mathbb{R}$.

  For each $i=1,2$ denote by $\mu_1(\lambda_i):=\sigma_{\operatorname{max}}\left((A_n-\lambda_i I_n)^{-1}\right)$. Then, for any $t_1,t_2\in\mathbb{R}$ and any $u\in\mathbb{S}^{n-1}$, we have the estimate for any $\delta_1,\delta_2\in[e^{-c_*n/2},1]:$
\begin{equation}\begin{aligned}\label{eq330791}
    &\mathbb{P}_X\left(\left| \left\langle (A_n-\lambda_iI_n)^{-1}X,X\rangle-t_i\right|\leq \delta_i\mu_1(\lambda_i),\quad i=1,2;\quad \langle X,u\right\rangle\geq s\right)\\&\lesssim \delta_1\delta_2\cdot e^{-s}\cdot \int_{\{\theta\in\mathbb{R}^2:\sum_{i=1}^2 |\theta_i\delta_i|^2\leq 2\}} I(\theta)^\frac{1}{2}d\theta+e^{-\Omega(n)},\end{aligned}
\end{equation}
where $I(\theta)$ is defined as follows: 
\begin{equation}\label{Itheta}
I(\theta):=\mathbb{E}_{J,X_J,X_J'}\exp\left(\langle (X+X')_J,u\rangle-c\left\|\sum_{i=1}^2 \frac{\theta_i}{\mu_1(\lambda_i)}(A_n-\lambda_iI_n)^{-1}(X-X')_J\right\|_2^2
    \right),
\end{equation} where $X'\sim\operatorname{Col}_n(\zeta)$ is independent of $X$ and we take $J\subseteq [n]$ to be a $\mu$-random set (i.e., a subset where each element in $[n]$ has (indepndently) probability $\mu$ to belong to this subset). 
\end{theorem}

Note that, compared to the expression in Proposition \ref{finalfuckpropositionga}, we have replaced the term $\|(A_n-\lambda_i I_n)^{-1}X\|_2$ in the denominator there by the minimal singular value $\mu_1(\lambda_i)$ here, which is a lot more controllable. We will turn back to $\|A_n^{-1}X\|_2$ much later, in Section \ref{bootstrapping209}.

To prove this theorem, we need a number of preparatory results. First, we need an Esseen-type inequality with an additional decoupling of different types of constraints:
\begin{lemma}\label{lemma6.234}
For $B>0$ consider $\zeta\in\Gamma_B$ and $X\sim\operatorname{Col}_n(\zeta)$. Let $M_1,M_2$ be two $n\times n$ symmetric matrices, $u\in\mathbb{R}^n$, $t_1,t_2\in\mathbb{R}$ and $s,\delta_1,\delta_2\geq 0$. Then 
\begin{equation}\label{mainbound2s}\begin{aligned}
    &\mathbb{P}(|\langle M_iX,X\rangle-t_i|<\delta_i,i=1,2,\langle X,u\rangle\geq s)\\&\lesssim \delta_1\delta_2e^{-s}\int_{\{\theta\in\mathbb{R}^2:\theta_1^2\delta_1^2+\theta_2^2\delta_2^2\leq 2\}}|\mathbb{E}e^{2\pi i\theta_1\langle M_1X,X\rangle+2\pi i\theta_2\langle M_2X,X\rangle+\langle X,u\rangle}|d\theta.
\end{aligned}\end{equation}
\end{lemma}

\begin{proof}
The proof consists of two parts: (i) removing the soft constraint $\langle X,u\rangle\geq s$ and (ii) applying Esséen's lemma. First we can bound, for any $s>0$, 
\begin{equation}\label{1409step1}
\mathbb{P}(|\langle M_iX,X\rangle-t_i|\leq \delta_i,i=1,2,\langle X,u\rangle\geq s)\leq e^{-s}\mathbb{E}[\mathbf{1}(|\langle M_iX,X\rangle-t_i|\leq\delta_i)e^{\langle X,u\rangle}].
\end{equation}
Then define a random variable $Y\in\mathbb{R}^n$ such that 
$$
\mathbb{P}(Y\in U)=(\mathbb{E}e^{\langle X,u\rangle})^{-1}\mathbb{E}_X[1_\mathbf{U}e^{\langle X,u\rangle}],
$$ so that 
\begin{equation}\label{1409step2}
\mathbb{E}[\mathbf{1}(|\langle M_iX,X\rangle-t_i|\leq\delta_i)e^{\langle X,u\rangle}]=\mathbb{E}e^{\langle X,u\rangle}\mathbb{E}
[\mathbf{1}(|\langle M_iY,Y\rangle-t_i|\leq\delta_i,i=1,2)].\end{equation}

Then we shall use the following version of multidimensional Esséen's lemma (see \cite{esseen1966kolmogorov}): let $Z$ be a random vector in $\mathbb{R}^m$, then 
$$
\sup_{v\in\mathbb{R}^m} \mathbb{P}(\|Z-v\|_2\leq \sqrt{m})\leq C^m\int_{B(0,\sqrt{m})}|\phi_Z(\theta)|d\theta, 
$$
with $\phi_Z(\theta)=\mathbb{E}\exp( 2\pi i\langle \theta,Z \rangle)$. 

In our case we take $m=2$ and $Z=(\delta_1^{-1}\langle M_1Y,Y\rangle,\delta_2^{-1}\langle M_2Y,Y\rangle)$. Then we have 
$$
\mathbb{P}(|\langle M_iY,Y\rangle-t_i|\leq\delta_i,i=1,2)\lesssim\int_{|\theta\|_2\leq \sqrt{2}} [\mathbb{E}e^{2\pi i\frac{\theta_1}{\delta_1}\langle M_1Y,Y\rangle+2\pi i\frac{\theta_2}{\delta_2}\langle M_2Y,Y\rangle}]d\theta.
$$ Then we take the change of coordinate $(\theta_1,\theta_2)\mapsto(\delta_1\theta_1,\delta_2\theta_2)$ and deduce that 
$$
\mathbb{P}(|\langle M_iY,Y\rangle-t_i|\leq\delta_i,i=1,2)\lesssim\delta_1\delta_2\int_{\{\theta_1^2\delta_1^2+\theta_2^2\delta_2^2\leq 2\}} [\mathbb{E}e^{2\pi i\theta_1\langle M_1Y,Y\rangle+2\pi i\theta_2\langle M_2Y,Y\rangle}]d\theta,
$$ and the right hand side equals, by definition of $Y$, to 
$$(\mathbb{E}e^{\langle X,u\rangle})^{-1}
\delta_1\delta_2\int_{\{\theta_1^2\delta_1^2+\theta_2^2\delta_2^2\leq 2\}}[\mathbb{E}e^{2\pi i\theta_1\langle M_1X,X\rangle+2\pi i\theta_2\langle M_2X,X\rangle+\langle X,u\rangle}]d\theta.
$$ Combining this with \eqref{1409step1} and \eqref{1409step2} completes the proof.
\end{proof}

To control the quadratic form on the right hand side of \eqref{mainbound2s}, we use the following lemma from \cite{campos2024least}, Lemma 5.3:
\begin{lemma}\label{lemma6.256}
    Take $\zeta\in\Gamma$, $X,X'\sim\operatorname{Col}_n(\zeta)$ two independent random vectors and take $[n]=I\cup J$ be a partition. Let $u\in\mathbb{R}^n$ and $M$ an $n\times n$ symmetric matrix. Then we have
    $$
|\mathbb{E}_Xe^{2\pi i\langle MX,X\rangle+\langle X,v
\rangle}|^2\leq \mathbb{E}_{X_J,X_J'}e^{\langle (X+X')_J,u\rangle}\cdot\left|\mathbb{E}_{X_I}e^{4\pi i\langle M(X-X')_J,X_I\rangle+2\langle X_I,u\rangle}\right|.
    $$
\end{lemma}

We then have an elementary lemma (\cite{campos2024least}, Lemma 5.5) to estimate the exponential tilt:
\begin{lemma}\label{tiltlemma6.4}
    Take $B>0$, $\zeta\in\Gamma_B$ and $X\sim\operatorname{Col}_n(\zeta)$. Then we can find some $c\in(0,1)$ depending only on $B$ so that for any $u,v\in\mathbb{R}^n$,\begin{equation}
        |\mathbb{E}_Xe^{2\pi i\langle X,v\rangle+\langle X,u\rangle}|\leq \exp(-c\min_{r\in[1,c^{-1}]}\|rv\|_\mathbb{T}^2+c^{-1}\|u\|_2^2).
    \end{equation}
\end{lemma}
We then define a family of quasirandom events on $(J,X_J,X_J')$ that are satisfied by $A_n\in\mathcal{E}$ with very high probability. First define $\mathcal{F}_1$ via
\begin{equation}
    \mathcal{F}_1:=\{|J|\in[\mu n/2,2\mu n]\}.
\end{equation}
We write $\widetilde{X}=X-X'$, and then define $\mathcal{F}_2$ via
\begin{equation}
    \mathcal{F}_2:=\{\|\widetilde{X}\|_2n^{-1/2}\in[c,c^{-1}]\},
\end{equation}and define $\mathcal{F}_3(A_n)$ via
\begin{equation}
    \mathcal{F}_3(A_n):=\{\forall\theta\in\mathbb{R}^2_*,\|\theta\|_2\leq e^{c_*n}:[\theta\cdot A_n^{-1}\widetilde{X}]/\|[\theta\cdot A_n^{-1}\widetilde{X}]\|_2\in\operatorname{Incomp}(\delta,\rho)\},
\end{equation}
where we recall that $[\theta\cdot A_n^{-1}\widetilde{X}]$ was defined in \eqref{shorthandnotations}.

Finally, we take $v_1=(A_n-\lambda_1 I_n)^{-1}\widetilde{X}$, $v_2=(A_n-\lambda_2 I_n)^{-1}\widetilde{X}$ and $I:=[n]\setminus J$. Then define the event $\mathcal{F}_4(A_n)$ via
\begin{equation}
    \mathcal{F}_4(A_n):=\{\inf_{\theta\in\mathbb{R}^2_*,\|\theta\|_2\leq e^{c_*n}} D_{\alpha,\gamma}(\frac{(\theta_1v_1+\theta_2v_2)_I}{\|(\theta_1v_1+\theta_2v_2)_I\|_2})\geq e^{c_\Sigma n}\text{for those }\|(\theta_1v_1+\theta_2v_2)_I\|_2\geq 1\}.
\end{equation}
\begin{lemma}\label{lemma6.5defines} Let $A_n\sim\operatorname{Sym}_n(\zeta)$ with $A_n\in\mathcal{E}$ (see Lemma \ref{mainquasirandomness1}), and define $\mathcal{F}(A_n):=\mathcal{F}_1\cap\mathcal{F}_2\cap\mathcal{F}_3(A_n)\cap\mathcal{F}_4(A_n)$.
    We may choose $c\in(0,1)$ as a function of $B$ and $\mu$ so that 
    $$
\mathbb{P}_{J,X_J,X_J'}(\mathcal{F}^c)\lesssim e^{-cn}.
    $$
\end{lemma}

\begin{proof}[\proofname\ of Lemma \ref{lemma6.5defines}] For the event $\mathcal{F}_1$, we use Hoeffding's inequality to get $\mathbb{P}(\mathcal{F}_1^c)\leq e^{-\Omega(n)}$. For the event $\mathcal{F}_2$, since each coordinate of $\widetilde{X}/\sqrt{2\mu}$ is i.i.d. mean 0 variance 1 and has subgaussian moment at most $B/\sqrt{2\mu}$, we can use \cite{vershynin2018high}, Theorem 3.1 to show $\mathbb{P}(\mathcal{F}_2^c)\leq e^{-\Omega(n)}$ for a suitable $c>0$. The probability upper bound for $\mathcal{F}_3(A_n)$ and $\mathcal{F}_4(A_n)$ follows from the fact that $A_n\in\mathcal{E}$ and Markov's inequality.
    
\end{proof}

Now we finish the proof of the main result, Theorem \ref{theorem3.1}.

\begin{proof}[\proofname\ of Theorem \ref{theorem3.1}] By Lemma \ref{lemma6.234} and Lemma \ref{lemma6.256}, we evaluate the integral
\begin{equation}\label{evaluatetheintegrals}\int_{\{\theta_1^2\delta_1^2+\theta_2^2\delta_2^2\leq 2\}}\sqrt{\mathbb{E}_{X_J,X_J'}e^{\langle (X+X')_J,u\rangle}
|\mathbb{E}_{X_I}e^{4\pi i\langle \sum_{i=1}^2\theta_i\frac{(A_n-\lambda_i I_n)^{-1}}{\mu_1(\lambda_i)} (X-X')_J,X_I\rangle+2\langle X_I,u\rangle}}|d\theta.
\end{equation}

    We consider separately the contribution to the integral from tuples $(J,X_J,X_J')\in\mathcal{F}$ and $(J,X_J,X_J')\notin\mathcal{F}$. In the latter case, by triangle inequality, the contribution is bounded by 
    \begin{equation}
        \sqrt{\mathbb{E}_{J,X,X'}^{\mathcal{F}^c}e^{\langle (X+X')_J,u\rangle+2\langle X_I,u\rangle}}\leq\sqrt{\mathbb{E}_{J,X,X'}[e^{\langle Y,2u\rangle}]^{1/2}\mathbb{P}_{J,X_J,X_J'}(\mathcal{F}^c)^{1/2}}\lesssim e^{-\Omega(n)},
    \end{equation} where we write $Y=(X+X')_J+2X_I$, and we use that the coordinates of $Y$ are mean zero and subgaussian.
Then consider the case $(J,X_J,X_J')\in\mathcal{F}$. For each $i=1,2$ we take $v_i=v_i(X):=\frac{(A_n-\lambda_iI_n)^{-1}(X-X')_J}{\mu_1(\lambda_i)}$. Then by Lemma \ref{tiltlemma6.4}, we have
    \begin{equation}
\label{productmatrixprocess}|\mathbb{E}_{X_I}e^{2\pi i\langle X_I,2\theta_1v_1+2\theta_2v_2\rangle+2\langle X_I,u\rangle}|\lesssim \exp(-c\min_{r\in[1,c^{-1}]}\|2r(\theta_1v_1+\theta_2v_2)_I\|_\mathbb{T}^2+c^{-1}\|u\|_2^2).
    \end{equation}

By assumption, $|\theta_1|,|\theta_2|$ are smaller than $e^{c_*n/2}$. We always have $|v_1|\leq c^{-1}\sqrt{n}$ and $|v_2|\leq c^{-1}\sqrt{n}$ by definition of the two singular values $\mu_1(\lambda_i),\mu_2(\lambda_i)$ and on the event $\mathcal{F}_2$, so that $\|\theta_1v_1+\theta_2v_2\|_2\leq e^{c_\Sigma n}$ if $|\theta_1|,|\theta_2|\leq e^{c_\Sigma n/2}$.

There are two cases: if $\|(\theta_1v_1+\theta_2v_2)_I\|_2\leq\frac{1}{2}$ for any $|\theta_1|,|\theta_2|\leq e^{c_\Sigma n}$, then we can simply replace $\|(\theta_1v_1+\theta_2v_2)_I\|_\mathbb{T}=\|(\theta_1v_1+\theta_2v_2)_I\|_2$. Or if for some $|\theta_1|,|\theta_2|\leq e^{c_\Sigma n/2}$ we have $\|(\theta_1v_1+\theta_2v_2)_I\|_2\geq 1$, then on the event $\mathcal{F}_4(A_n)$, using $\|(\theta_1v_1+\theta_2v_2)_I\|_2\leq e^{c_\Sigma n}$, 
$$
\|(\theta_1v_1+\theta_2v_2)_I\|_\mathbb{T}=\|\frac{\|(\theta_1v_1+\theta_2v_2)_I\|_2(\theta_1v_1+\theta_2v_2)_I}{\|(\theta_1v_1+\theta_2v_2)_I\|_2}\|_\mathbb{T}\geq\min( \alpha|I|,\gamma\|(\theta_1v_1+\theta_2v_2)_I\|_2)
$$ by definition of the essential LCD. On the event $\mathcal{F}_3(A_n)$, the vector $\theta_1v_1+\theta_2v_2$ is $(\delta,\rho)$-incompressible, so we can find some $C>0$ depending only on $B$ such that 
$$
\|(\theta_1v_1+\theta_2v_2)_I\|_2\geq C\|\theta_1v_1+\theta_2v_2\|_2. 
$$ We take this back to the integral \eqref{evaluatetheintegrals}, consider separately $\gamma\|(\theta_1v_1+\theta_2v_2)_I\|_2\leq\alpha|I|$ and the opposite case (where the right hand side of \eqref{productmatrixprocess} is already exponentially small), and the contribution to the integral for pairs $(J,X_J,X_J')\in\mathcal{F}$ is bounded by the integral of 
$$
\sqrt{\exp\left(\langle(X+X')_J,u\rangle-c\|\theta_1v_1+\theta_2v_2\|_2^2\right)+e^{-\Omega(n)}}
$$ over $\{\theta:\theta_1^2\delta_1^2+\theta_2^2\delta_2^2\leq 2\}$, which gives the claimed result by Lemma \ref{lemma6.234}.

\end{proof}

\section{Global properties of the spectrum} \label{chap3chap3chap3}
In this section we establish a few eigenvalue rigidity estimates for Wigner matrices. We make crucial use of the local semicircle law, as well as the super-exponential concentration of the empirical measure guaranteed by a finite log-Sobolev constant. 

Recall that for any $\lambda_i\in\mathbb{R}$, we denote by $\mu_1(\lambda_i)\geq\mu_2(\lambda_i)\cdots\geq \mu_n(\lambda_i)\geq 0$ the singular values of $(A_n-\lambda_i I)^{-1}$ in decreasing order. Additionally, we define a modified norm $\|\cdot\|_*$ for $(A_n-\lambda_i I_n)^{-1}$ via
\begin{equation}
    \left\|(A_n-\lambda_i I_n)^{-1}\right\|_*^2=\sum_{k=1}^n \mu_k(\lambda_i)^2(\log(1+k))^2.
\end{equation}
The use of this norm will not be apparent until in Section \ref{removalladfaga}: before this we can safely replace $\|M\|_*$ by $\|M\|_{HS}$. Informally, we can think of $\|M\|_*$ as a biased version of $\|M\|_{HS}$, and that when $\|M\|_*$ deviates significantly from $\|M\|_{HS}$, this would signify a macroscopic event taking place in the spectrum. 

\subsection{Statement of main estimates }
We first need the following two results that generalize Lemma 8.1 and 8.2 of \cite{campos2024least} to nonzero eigenvalues in the bulk. 

\begin{lemma}\label{lemma4.11} For any $p>1$, $B>0$, $\kappa>0$ and $\zeta\in\Gamma_B$, let $A\sim\operatorname{Sym}_n(\zeta)$. Then we can find a constant $C_p$ depending on $B$, $p$ and $\kappa$ such that for any $\lambda_i\in[-(2-\kappa)\sqrt{n},(2-\kappa)\sqrt{n}]$ we have, for all $k=1,2,\cdots,n$,
$$
\mathbb{E}\left[\left(\frac{\sqrt{n}}{\mu_k(\lambda_i)\cdot k}\right)^p\right]\leq C_p.
$$
\end{lemma}

\begin{corollary}\label{lemma4.22}
    For a given $p>1$, $B>0$, $\kappa>0$ and $\zeta\in\Gamma_B$, $A\in\operatorname{Sym}_n(\zeta)$. Then we can find a constant $C_P>0$ depending on $B$, $p$, $\kappa$ such that for any $\lambda_i\in[(-(2-\kappa)\sqrt{n},(2-\kappa)\sqrt{n}]$, we have
    \begin{equation}
        \mathbb{E}\left[\left(\frac{\|(A_n-\lambda_i I_n)^{-1}\|_*}{\mu_1(\lambda_i)}\right)^p\right]\leq C_p.
    \end{equation}
\end{corollary}

We also need the following high-probability estimates for large values of $k$:

\begin{lemma}\label{corollary3.6chap9} Let $B>0$, $\zeta\in\Gamma_B$ and $A_n\in\operatorname{Sym}_n(\zeta)$. Fix any $\kappa>0$. Fix two locations $\lambda_1,\lambda_2\in[-(2-\kappa)\sqrt{n},(2-\kappa)\sqrt{n}]$ and any $\sigma\in(0,1)$. Then for any $c_0>0$ we can find constants $C_1,C_2,C_3>0$ depending only on $B$, $\kappa$ (they do not depend on $c_0$) such that
\begin{equation}\label{estimate4s}
\mathbb{P}\left(  \frac{C_1\sqrt{n}}{k}  \leq \mu_{k}(\lambda_i)\leq  \frac{C_2\sqrt{n}}{k}\text{ for both $i\in\{1,2\}$ and all } c_0n^\sigma\leq k\leq C_3n\right)\geq 1-e^{-\Omega(n^\sigma/2)},
\end{equation} where the constant in the error term $\exp(-\Omega(n^\frac{\sigma}{2}))$ depends also on $c_0>0$.

In particular, we can find $C_4>1$ depending only on $\kappa$ and $B$ so that for any $c_0>0$,
$$
\left(\mu_{\frac{k}{C_4}}(\lambda_i)\geq {10}\mu_k(\lambda_j) \text{ for all } c_0n^\sigma\leq k\leq C_3n\text{ and  }\text{ for all } i,j\in\{1,2\}\right)\geq 1-e^{-\Omega(n^{\sigma/2})}.
$$
\end{lemma}

Assuming further that the entries of $A$ have a finite log-Sobolev constant, we can prove stronger concentration estimates:

\begin{lemma}\label{corollary4.88}
    Given any $B>0$, a random variable $\zeta\in \Gamma_B$ that has a finite, $n$-independent log-Sobolev constant, and a random matrix $A_n\sim\operatorname{Sym}_n(\zeta)$. Fix $\kappa>0$. Then we can find constants $C_1,C_2,C_3>0$ depending only on $\kappa$, $B$ and the log-Sobolev constant (and not on $c_0$) such that, for any $c_0>0$ and  $\lambda_1,\lambda_2\in[-(2-\kappa)\sqrt{n},(2-\kappa)\sqrt{n}]$,
    \begin{equation}\label{coro4.8first}
   \mathbb{P}\left( \frac{C_1\sqrt{n}}{k}\leq 
   \mu_k(\lambda_i)\leq \frac{C_2\sqrt{n}}{k}\text{ for both $i\in\{1,2\}$ and all } c_0n\leq k\leq C_3n
   \right)\geq 1-e^{-\Omega(n)}.     
    \end{equation}

In particular, we can find $C_4>1$ depending only on $\kappa$, $B$ and the log-Sobolev constant such that for any $c_0>0$ and for each $\lambda_i$ in the given interval,
$$
\left(\mu_{\frac{k}{C_4}}(\lambda_i)\geq {10}\mu_k(\lambda_j) \text{ for all } c_0n\leq k\leq C_3n\text{ and  }\text{ for all } i,j\in\{1,2\}\right)\geq 1-e^{-\Omega(n)}.
$$
   
\end{lemma}

\subsection{Proofs via local semicircle law}

The proofs of Lemma \ref{lemma4.11}, Corollary \ref{lemma4.22} and Corollary \ref{corollary3.6chap9} rely on the following local semicircular law estimates for Wigner matrices by Erdös, Schlein and Yau \cite{erdHos2011universality}, Theorem 1.11. The Wigner semicircle law $\rho_{sc}$ is defined as 
\begin{equation}\label{semicircle}
\rho_{sc}(x):=\frac{1}{2\pi}\sqrt{(4-x^2)_+},\quad x\in\mathbb{R}.
\end{equation}

\begin{theorem}\label{theorem4.3}(\cite{erdHos2011universality}, Theorem 1.11.) Let $H\in\operatorname{Sym}_n(\zeta)$ where $\zeta\in\Gamma_B$ for some $B>0$. Let $\kappa>0$ and consider $E\in[-2+\kappa,2-\kappa]$. Denote by $\mathcal{N}_{\eta^*}(E)=\mathcal{N}_{I^*}$ the number of eigenvalues of $n^{-\frac{1}{2}}H$ in $I^*:=[E-\eta^*/2,E+\eta^*/2]$. Then we can find a universal constant $c_1>0$ and two constants $C$, $c>0$ depending only on $\kappa$ such that for any $\delta\leq c_1\kappa$ we can find a constant $K_\delta$ depending only on $\delta$ such that 
\begin{equation}
    \mathbb{P}\left\{\left|\frac{\mathcal{N}_{\eta^*}(E)}{n\eta^*}-\rho_{sc}(E)\right|\geq\delta\right\}\leq Ce^{-c\delta^2\sqrt{n\eta^*}},
\end{equation}
for all $\eta^*$ satisfying $K_\delta/n\leq\eta^*\leq 1$ and all $n\geq 2$.
\end{theorem}
Since $\rho_{sc}(x)$ is maximal at $x=0$ and decreases as $|x|$ increases, we deduce the following immediate corollary:
\begin{corollary}\label{corollary425}
    In the setting of Theorem \ref{theorem4.3}, for any $\kappa>0$ we have the following two estimates uniform in all $E\in[-2+\kappa,2-\kappa]$: for any $\eta^*\in[C_{\ref{corollary425}}n^{-1},1]$,
\end{corollary}
\begin{equation}\label{1stcor4.4}
    \mathbb{P}\left\{\frac{\mathcal{N}_{\eta^*}(E)}{n\eta^*}\geq\pi\right\}\lesssim \exp(-c_1\sqrt{n\eta^*}),
\end{equation}
and \begin{equation}\label{2ndcor4.4}
    \mathbb{P}\left\{\frac{\mathcal{N}_{\eta^*}(E)}{n\eta^*}\leq \frac{1}{4}\rho_{sc}(2-
{\kappa})\right\}\lesssim \exp(-c_1\sqrt{n\eta^*}),
\end{equation}
where $C_{\ref{corollary425}}$ and $c_1$ are constants that depend only on $\kappa>0$.

From this estimate we deduce two immediate corollaries.

\begin{corollary}\label{corollary8.5}
Fix $B>0$, $\zeta\in\Gamma_B$ and $A_n\sim\operatorname{Sym}_n(\zeta)$. Fixing $\kappa>0$, then we can find positive constants $C_{\ref{corollary8.5}}$ and $c$ depending only on $\kappa$ such that, for all $s\geq C_{\ref{corollary8.5}}$ and $k\in\mathbb{N}$ that satisfy $sk\leq n$, uniformly for all $\lambda_i\in[-(2-\kappa)\sqrt{n},(2-\kappa)\sqrt{n}]$,
$$
\mathbb{P}\left(\frac{\sqrt{n}}{\mu_k(\lambda_i)\cdot k}\geq s\right)\lesssim\exp(-c(sk)^\frac{1}{2}).
$$
We also have, uniformly for all $\lambda_i\in[-(2-\kappa)\sqrt{n},(2-\kappa)\sqrt{n}]$, for all $k\in\mathbb{N},$
$$
\mathbb{P}\left(\mu_k(\lambda_i)\geq \frac{\sqrt{n}}{C_{\ref{corollary8.5}}k}\right)\lesssim \exp(-ck^\frac{1}{2}).
$$
\end{corollary}

\begin{proof} Denote by $C_{\ref{corollary8.5}}:=\max(\pi, \frac{4}{\rho_{sc}(2-\kappa)},C_{\ref{corollary425}})$.
    If $\frac{\sqrt{n}}{\mu_k(\lambda_i)\cdot k}\geq s$, then $N_{n^{-\frac{1}{2}}A}(-skn^{-1}+\lambda_i n^-\frac{1}{2},skn^{-1}+\lambda_i n^-\frac{1}{2})\leq k$, where we use the symbol $N_{H}(I)$ to denote the number of eigenvalues of $H$ in the interval $I$. 
    
    Then the first claim follows by applying estimate \ref{2ndcor4.4} with $\eta_*=skn^{-1}\geq sn^{-1}\geq C_{\ref{corollary425}}n^{-1}$ and noticing $s^{-1}\leq \frac{\rho_{sc}(2-\kappa)}{4}$. The second claim follows from applying the estimate \ref{1stcor4.4} with $\eta_*=C_{\ref{corollary8.5}}kn^{-1}\geq C_{\ref{corollary8.5}}n^{-1}$ and noticing $C_{\ref{corollary8.5}}\geq \pi$. 
\end{proof}

Now we are in the place to prove Lemma \ref{lemma4.11} and Corollary \ref{lemma4.22}.

\begin{proof}[\proofname\ of Lemma \ref{lemma4.11}] From the tail estimate on $\|A_n\|_{op}$ in Lemma \ref{operatorbound}, we deduce that for any $k\geq n/C_{\ref{corollary8.5}}$ and any $\lambda_i\in[-(2-\kappa)\sqrt{n},(2-\kappa)\sqrt{n}]$,
$$
\mathbb{E}\left(\frac{\sqrt{n}}{\mu_k(\lambda_i)\cdot k}\right)^p\leq\mathbb{E}_{A_n}\left(\frac{\sigma_1(A_n+\lambda_i I_n)\sqrt{n}}{k}
\right) ^p=O_p((n/k))^p=O_p(1).
$$ Then for $k\leq n/C_{\ref{corollary8.5}}$, we consider separately the following three events
$$
E_1^i=\left\{\frac{\sqrt{n}}{\mu_k(\lambda_i)\cdot k}\leq C_{\ref{corollary8.5}}\right\},\quad 
E_2^i=\left\{\frac{\sqrt{n}}{\mu_k(\lambda_i)\cdot k}\in [C_{\ref{corollary8.5}},\frac{n}{k}]\right\},
\quad 
E_3^i=\left\{\frac{\sqrt{n}}{\mu_k(\lambda_i)\cdot k}\geq \frac{n}{k}\right\}
$$ and bound 
\begin{equation}
  \mathbb{E}\left(\frac{\sqrt{n}}{\mu_k(\lambda_i)\cdot k}\right)^p\leq C_{\ref{corollary8.5}}^p+ \mathbb{E}\left(\frac{\sqrt{n}}{\mu_k(\lambda_i)\cdot k}\right)^p \mathbf{1}_{E_2^i}+\mathbb{E}\left(\frac{\sqrt{n}}{\mu_k(\lambda_i)\cdot k}\right)^p \mathbf{1}_{E_3^i}. 
\end{equation}

For the second term, we apply Corollary \ref{corollary8.5} to bound 
$$
\mathbb{E}\left(\frac{\sqrt{n}}{\mu_k(\lambda_i)\cdot k}\right)^p\lesssim\int_{C_{\ref{corollary8.5}}}^{n/k}ps^{p-1}e^{-c\sqrt{sk}}ds=O_p(1).
$$

  For the third term, since $n/k\geq C_{\ref{corollary8.5}}$, we apply Corollary \ref{corollary8.5} with $s=n/k$ and deduce that $\mathbb{P}(E_3^i)\lesssim e^{-c\sqrt{n}}$. Then applying the Cauchy-Schwartz inequality,
$$\begin{aligned}
\mathbb{E}&\left(\frac{\sqrt{n}}{\mu_k(\lambda_i)\cdot k}\right)^p\mathbf{1}_{E_3^i}\leq \left(\mathbb{E}(\frac{\sigma_1(A_n-\lambda_i I_n)\sqrt{n}}{k})^{2p}\right)^{1/2}\mathbb{P}(E_3^i)^{1/2}\\\quad&\leq O_p(1)n^p e^{-c\sqrt{n}}=O_p(1).\end{aligned}
$$ The bound on $\sigma_1$ follows from Lemma \ref{operatorbound}. This completes the proof.
\end{proof}

\begin{proof}[\proofname\ of Corollary \ref{lemma4.22}]
    Recall from definition that 
    $$
\|(A_n-\lambda_i I_n)^{-1}\|_*^2=\sum_{k=1}^n \mu_k^2(\lambda_i)(\log(1+k))^2.
    $$ We may take without loss of generality $p\geq 2$ thanks to Hölder's inequality, and apply the triangle inequality to obtain
    \begin{equation}\label{intothesummation}
\left[\mathbb{E}\left(\sum_{k=1}^n \frac{\mu_k^2(\lambda_i)(\log(1+k))^2}{\mu_1^2(\lambda_i)}\right)^{p/2}\right]^{2/p}
\leq\sum_{k=1}^n(\log(1+k))^2\mathbb{E}\left[\frac{\mu_k^p(\lambda_i)}{\mu_1^p(\lambda_i)}\right]^{2/p}.
    \end{equation}
   Then we apply Lemma \ref{lemma4.11} and Corollary \ref{corollary8.5} to the upper bound 
    \begin{equation}\label{totheupperbounds}
\mathbb{E}\left[\frac{\mu_k^p(\lambda_i)}{\mu_1^p(\lambda_i)}\right]\leq (C_{\ref{corollary8.5}}k)^{-p}\mathbb{E}\left[\left(\frac{\sqrt{n}}{\mu_1(\lambda_i)}\right)^p\right]+\mathbb{P}(\mu_k\geq \frac{\sqrt{n}}{C_{\ref{corollary8.5}}k})\lesssim C_p^pk^{-p} 
    \end{equation}for some constant $C_p>0$ depending only on $\kappa$, whose value may change from the value $C_p$ appearing in Lemma \ref{lemma4.11}. Then the corollary follows by taking \eqref{totheupperbounds} into the summation \eqref{intothesummation} and showing the summation converges to a finite value depending only on $p$.
\end{proof}

\subsection{Proofs via Log-Sobolev inequality}
When $\sigma=1$, the estimate from Lemma \ref{corollary3.6chap9} has the disadvantage that it only gives an error of order $e^{-c\sqrt{n}}$, falling short of the desired $e^{-cn}$ error. To remedy this, we invoke the following stronger concentration estimate, which relies on the fact that $\zeta$ has a finite log-Sobolev constant. 

\begin{Proposition}\label{proposition4.7}[\cite{WOS:000542157900013}, Lemma 6.1]
    Let $B>0$, and $\zeta\in\Gamma_B$ be a random variable having a finite, $n$-independent log-Sobolev constant. Consider $A_n\sim\operatorname{Sym}_n(\zeta)$, and denote by $\sigma_1,\cdots,\sigma_n$ the eigenvalues of $A_n$. Define the empirical measure $\rho_{A_n}^n$ via 
$$\rho_{A_n}^n:=\frac{1}{n}\sum_{i=1}^n \delta_{\frac{\sigma_i}{\sqrt{n}}},$$ 
where $\delta_\cdot$ is the delta measure. Then there exists $\kappa_{\ref{proposition4.7}}\in(0,\frac{1}{10})$ such that 
\begin{equation}\label{refinedconcentration}
    \lim_{n\to\infty}\frac{1}{n}\ln\mathbb{P}[d(\rho_{A_n}^n,\rho_{sc})>n^{-\kappa_{\ref{proposition4.7}}}]=-\infty,
\end{equation}
    where $\rho_{sc}$ is the semicircle law \eqref{semicircle} and the distance $d$ is the Dudley (i.e., bounded Lipschitz) distance defined as follows: for two probability measures $\mu$ and $\nu$,
    $$
d(\mu,\nu)=\sup_{\|f\|_L\leq 1}\left|\int f(x)d\mu(x)-\int f(x)d\nu(x)\right|,
    $$
    where $\|f\|_L:=\sup_{x\neq y}\frac{|f(x)-f(y)|}{|x-y|}+\sup_x |f(x)|$.
    \end{Proposition}

Now we are in the position to prove Lemma \ref{corollary4.88}.

\begin{proof}[\proofname\ of Lemma \ref{corollary4.88}] We shall use two constants $C_1$ and $C_2$ whose value will not be determined until the end of the proof. The fact that 
$\mu_k(\lambda_i)\geq \frac{C_2\sqrt{n}}{k}$ implies that on the interval $[\lambda_i-\frac{k}{C_2\sqrt{n}},\lambda_i+\frac{k}{C_2\sqrt{n}}]$, there are more than $k$ eigenvalues of $A_n$. Likewise, $\mu_k(\lambda_i)\leq \frac{C_1\sqrt{n}}{k}$ implies that on the interval $[\lambda_i-\frac{k}{C_1\sqrt{n}},\lambda_i+\frac{k}{C_1\sqrt{n}}]$, there are fewer than $k$ eigenvalues of $A_n$. 

We now normalize by dividing the interval by $\sqrt{n}$ and consider $I_2:=[\frac{\lambda_i}{\sqrt{n}}-\frac{k}{C_2n},\frac{\lambda_i}{\sqrt{n}}+\frac{k}{C_2n}]$. Since the semicircle law $\rho_{sc}(E)$ has bounded density, this would imply that $\rho_{sc}(I_2)\leq \rho_{sc}(0)|I_2|=\rho_{sc}(0)\frac{2k}{C_2n}$. Moreover, $\mu_k(\lambda_i)\geq \frac{C_2\sqrt{n}}{k}$ implies that $\rho_{A_n}^n(I_2)\geq \frac{k}{n}$. Therefore, if we choose $\rho_{sc}(0)\frac{2}{C_2}<1$, we must have $|\rho_{A_n}^n(I_2)-\rho_{sc}(I_2)|\geq c_0(1-\rho_{sc}(0)\frac{2}{C_2})>0$. Since $|I_2|\geq c_0>0$, this event is super-exponentially unlikely according to Proposition \ref{proposition4.7}.

Likewise, we consider $I_1:=[\frac{\lambda_i}{\sqrt{n}}-\frac{k}{C_1n},\frac{\lambda_i}{\sqrt{n}}+\frac{k}{C_1n}]$. Assume without loss of generality that $\lambda_i\geq 0$. If $\frac{k}{{C_1n}}\leq 2$, then $\mu_{sc}(I_1)\geq \frac{k}{C_1n}\rho_{sc}(2-\kappa)$. However the assumption that $\mu_k(\lambda_i)\leq \frac{C_1\sqrt{n}}{k}$ tells us that $\mu_{A_n}^n(I_1)\leq\frac{k}{n}$. Therefore, if we choose $C_1=\frac{1}{2}\rho_{sc}(2-\kappa)$, this event is superexponentially unlikely by Proposition \ref{proposition4.7}. This already implies the second claim if we choose $C_4=10C_2/C_1$. Indeed, we can apply the first claim and obtain $\mu_{\frac{k}{C_4}}(\lambda_i)\geq \frac{C_1C_4\sqrt{n}}{k}
=10\frac{C_2\sqrt{n}}{k}\geq 10\mu_k(\lambda_j)$ on an event with probability $1-e^{-\Omega(n)}$, for any $i,j\in\{1,2\}$.

\end{proof}

The proof of Lemma \ref{corollary3.6chap9} is almost parallel, except that we use Corollary \ref{corollary8.5} in place of Proposition \ref{proposition4.7}.

\begin{proof}[\proofname\ of Lemma \ref{corollary3.6chap9}]
    The first claim follows from taking a union bound for $k$ in the range $c_0n^\sigma\leq k\leq n$ in the conclusion of Corollary \ref{corollary8.5}. The second claim follows similarly.
\end{proof}

\section{Inner product with many vectors at two locations}
\label{chap4chap4chap4}
In this section we estimate the norm of a crucial term in the functional $I(\theta)$, where we recall that $I(\theta)$ is a function defined in \eqref{Itheta}. 

The main difficulty in estimating $I(\theta)$ is two-fold. First, by definition of $I(\theta)$, we have to estimate the inner product of a random vector with many fixed orthogonal vectors instead of just one vector, and the inner product with different vectors have different weights. This technical difficulty already appeared in the one-location case \cite{campos2024least}. The second difficulty is that we have two locations $\lambda_1,\lambda_2$, and we need to decouple the event of small inner product at these two different locations. 
The first main result of the section is as follows.

\begin{lemma}\label{lemma6.61}
Fix $B>0,\zeta\in\Gamma_B$ and $A_{n}\sim\operatorname{Sym}_n(\zeta)$. We further assume that $\zeta$ has a finite, $n$-independent Log-Sobolev constant. Let $X$ be a random vector with $X\sim\operatorname{Col}_n(\zeta)$, and define the random vector $\widetilde{X}:=(X-X')_J$, where $X'$ is an independent copy of $X$ and $J$ is a $\mu$-subset of $[n]$.

Fix some $\kappa>0$ and $\Delta>0$. 
    Fix two locations
   $\lambda_1,\lambda_2\in[-(2-\kappa)\sqrt{n},(2-\kappa)\sqrt{n}]$ with $|\lambda_1-\lambda_2|\geq \Delta\sqrt{n}$.

   Then there exists an event $\Omega(A_n)$ depending on $\kappa,\Delta$ with probability $\mathbb{P}(A_n\in\Omega(A_n))\geq1-\exp(-\Omega(n))$ such that whenever $A_n\in\Omega(A_n)$, the following hold:
   \begin{enumerate}
       \item Fix any given $(\theta_1,\theta_2)\in\mathbb{R}^2$ satisfying $|\theta_1\theta_2|=1$,

   Then there exists some $J\in\{1,2\}$ (which is chosen in \eqref{choiceofbigJ} in a way depending on $A_n$ and the relative ratio $|\theta_1/\theta_2|$) such that, for any $1\leq k\leq {c_0}_{\ref{lemma6.61}}n$ and for any \begin{equation}\label{wheredoesslies}s\in\left(e^{-cn}, \frac{C_{\ref{lemma6.61}}\cdot \mu_k(\lambda_J)}{{\left(\prod_{i=1}^2\mu_1(\lambda_i)\right)^{1/2}}}\right),\end{equation} we have the estimate
    \begin{equation}\label{212212212}
\mathbb{P}_{\widetilde{X}}\left(\left\|\sum_{i=1}^2 \frac{\theta_i}{\mu_1(\lambda_i)}(A_n-\lambda_i I_n)^{-1}\widetilde{X}\right\|_2\leq s
\right) \lesssim s e^{-ck},
    \end{equation} where ${c_0}_{\ref{lemma6.61}}>0$ and $C_{\ref{lemma6.61}}>0$ are two constants that depend only on $\kappa,\Delta,B$ and the Log-Sobolev constant of $\zeta$. The constant $c>0$ depends only on $B$.
    \item  Fix any given $(\theta_1,\theta_2)\in\mathbb{R}^2$ satisfying $\max(|\theta_1|,|\theta_2|)=1$. Then there exist some $I,J\in\{1,2\}$ depending on $A_n$ and the relative ratio $|\theta_1/\theta_2|$ such that, for any $1\leq k\leq {c_0}_{\ref{lemma6.61}}n$ and any 
    \begin{equation}\label{secondsasess}
s\in(e^{-cn},\frac{\mu_k(\lambda_J)}{\mu_1(\lambda_I)})
    \end{equation} we have the same estimate \eqref{212212212}.
     \end{enumerate}
\end{lemma}

It might seem a surprise that in the estimate \eqref{212212212} we do not aim for the optimal power in $s$, as one may expect a power $s^2$. The reason is two-fold: first, whether or not the $s^2$ power can be reached depends on the relative size of $|\theta_1|,|\theta_2|$, which is not assumed here and we wish to cover the case when one $\theta_i$ is very small; and second, it will be clear that an estimate with merely a power $s$ (rather than $s^2$) is already enough to conclude the proof.

To illustrate the main ideas, we first recall how to compute this expectation with only one location $\lambda_1=0$. This is the content of \cite{campos2024least}, Chapter 9. Rephrased in the setting here, we first prove that for any $s\in(e^{-cn},\mu_k/\mu_1)$ (where we abbreviate $\mu_k:=\mu_k(0)$),
$$\mathbb{P}_{\widetilde{X}}\left(\|A^{-1}\widetilde{X}\|_2\leq s\mu_1\right)\lesssim se^{-ck}.$$

The idea is to expand 
$$\|A^{-1}\widetilde{X}|_2^2=\sum_{j=1}^n \mu_j^2 \langle v_j,\widetilde{X}\rangle 
^2$$
where $v_j$ is the eigenvector of $A^{-1}$ corresponding to $\mu_j$.
Then 
$$\mathbb{P}_{\widetilde{X}}(\|A^{-1}\widetilde{X}\|_2\mu_1^{-1}\leq s)\leq\mathbb{P}_{\widetilde{X}}\left(|\langle v_1,\widetilde{X}\rangle|\leq s,\quad \sum_{j=2}^k\frac{\mu_j^2}{\mu_1^2}\langle v_j,\widetilde{X}\rangle^2\leq s^2\right).$$
By the assumption on $\mu_k/\mu_1$, we further deduce that
$$
\mathbb{P}_{\widetilde{X}}(\|A^{-1}\widetilde{X}\|_2\mu_1^{-1}\leq s)\leq\mathbb{P}_{\widetilde{X}}\left(|\langle v_1,\widetilde{X}\rangle|\leq s,\quad \sum_{j=2}^k\langle v_j,\widetilde{X}\rangle^2\leq 1\right).
$$
The last term can be estimated via the novel inverse Littlewood-Offord inequality in \cite{campos2024least}.

New challenges arise as we have two locations $\lambda_1,\lambda_2$. We introduce new notations:

\begin{Definition}\label{definition4.2}
    Fix two distinct real numbers $\lambda_1,\lambda_2\in\mathbb{R}$. For any $i\neq j\in \{1,2\}$ and any $k\geq 1$ we denote by 
    $$
\mu_{c_j(i;k)}(\lambda_j)
    $$ the $c_j(i,k)$-th largest singular value of $(A_n-\lambda_j I_n)^{-1}$ whose associated eigenvector coincides with the eigenvector of $\mu_k(\lambda_i)$, which is the $k$-th largest singular value of $(A_n-\lambda_i I_n)^{-1}.$ 
\end{Definition}
In other words, we use $c_j(i;k)$ as a mapping that correlates the $k$-th singular value at location $\lambda_i$ to the $c_j(i,k)$-th singular value at  location $\lambda_j$.

This definition makes sense because $(A_n-\lambda_1 I_n)^{-1}$ and $(A_n-\lambda_2 I_n)^{-1}$ share the same set of eigenvectors: both these matrices have the same set of eigenvectors with $A_n-\lambda_i I_n$,i=1,2. This observation allows us to correlate the two events at locations $\lambda_1,\lambda_2$.

    When a singular value of $(A_n-\lambda_j I_n)^{-1}$ is associated with two orthogonal eigenvectors, which is the case when $(A_n-\lambda_j I_n)^{-1}$ has repeated singular values, then we just assign each of these (mutually orthogonal) eigenvectors to each of these (equal) singular values. This choice does not change the ensuing computations.

The proof of Lemma \ref{lemma6.61} also relies on a Littlewood-Offord theorem with a family of constraints from \cite{campos2024least}, Theorem 9.2. 
\begin{theorem}\label{theorem7.4littlewoodchapter2}
    Fix $n\in\mathbb{N}$ and fix $\alpha,\gamma\in(0,1)$, $B>0,\mu\in (0,2^{-15})$. We can find $c,R>0$ depending only on these constants so that the following holds. For $0\leq k\leq c\alpha n$ and $\epsilon\geq\exp(-c\alpha n)$, fix $v\in \mathbb{S}^{n-1}$ and orthogonal vectors $w_1,
\cdots,w_k\in\mathbb{S}^{n-1}$. Let $\zeta\in\Gamma_B$, $\zeta'$ be an independent copy of $\zeta$ and $Z_\mu$ a Bernoulli variable with mean $\mu$. Let $\widetilde{X}\subset\mathbb{R}^n$ be a random vector with i.i.d. coordinates of law $(\zeta-\zeta')Z_\mu$. Then if $D_{\alpha,\gamma}(v)\geq 1/\epsilon$ then 
\begin{equation}
    \mathbb{P}_{\widetilde{X}}(|\langle\widetilde{X},v\rangle|\leq\epsilon\text{ and }\sum_{j=1}^k\langle \widetilde{X},w_j\rangle^2\leq ck)\leq R\epsilon\cdot e^{-ck}.
\end{equation}
\end{theorem}

Now we are ready to prove Lemma \ref{lemma6.61}.

\begin{proof}[\proofname\ of Lemma \ref{lemma6.61}] The proof in cases (1) and (2) are the same up until the last step.
   \textbf{Part 1: Expanding the inner product as a sum.}

We first expand the inner product in the following form, where $v_k(\lambda_i)$ denotes the eigenvector associated with the singular value $\mu_k(\lambda_i):$
\begin{equation}\label{mainmainmian}\begin{aligned}&
\left\|\sum_{i=1}^2 \frac{\theta_i}{\mu_1(\lambda_i)}(A_n-\lambda_i I_n)^{-1}\widetilde{X}\right\|_2=\sum_{i=1}^2 \left|\theta_i+\sum_{j\neq i}\theta_j\frac{\mu_{c_j(i;1)(\lambda_j)}}{\mu_1(\lambda_j)}\right|^2 \langle \widetilde{X},v_1(\lambda_i)\rangle^2
\\&+\frac{1}{2}\sum_{i=1}^2\sum_{k=2}^n \left|\theta_i \frac{\mu_k(\lambda_i)}{\mu_1(\lambda_i)}+\sum_{j\neq i} \theta_j \frac{\mu_{c_j(i;k)(\lambda_j)}}{\mu_1(\lambda_j)}\right|^2\langle \widetilde{X},v_k(\lambda_i)\rangle ^2.
\end{aligned}\end{equation}
    This follows from expanding the inner product on the left hand side of \eqref{mainmainmian} with respect to the orthonormal basis of eigenvalues of $A_n-\lambda_i I_n$. The factor $\frac{1}{2}$ on the second line is because there is an overcounting as we sum over both $i=1,2$ and all $k\in[2,n]$. We also assume without loss of generality that $\lambda_1\leq\lambda_2$.
    
    The main difficulty in evaluating this expression is that for each eigenvector, say $v_1(\lambda_i)$, its weight in the inner product expansion depends on contributions from parameters at both $\lambda_1$ and $\lambda_2$, and we need to prove the two contributions do not cancel one another two much. We would like to classify the constraints into two sub-classes: the hard constraints involving the inner product of $\widetilde{X}$ with the two eigenvectors $v_1(\lambda_1)$ and $v_1(\lambda_2)$ (i.e. those associated with the two least singular values); as well as the soft constraints involving the inner product of $\widetilde{X}$ with other eigenvectors of $A_n$.
    
   In our expansion, the right hand side of the first line of \eqref{mainmainmian} consists of the hard constraints. The second line of \eqref{mainmainmian} are the set of soft constraints. 

\textbf{Part 2: Singling out the main contributions.} As seen from the expression \eqref{mainmainmian},  there can be many cancellations in the weight of the inner products. To single out the major contribution, we fix some notations. First, let $I\in\{1,2\}$ be such that $$|\theta_I|=\max\{|\theta_1|,|\theta_2|\}$$ and let $J\in\{1,2\}$ be such that 
\begin{equation}\label{choiceofbigJ}
\frac{|\theta_J|}{\mu_1(\lambda_J)}=\max\left\{\frac{|\theta_1|}{\mu_1(\lambda_1)},\frac{|\theta_2|}{\mu_1(\lambda_2)}\right\}.
\end{equation}

Now we apply the estimates in Section \ref{chap3chap3chap3} to single out the main contributions.

First consider the hard constraint. By the assumption that $|\lambda_1-\lambda_2|\geq\Delta\sqrt{n}$, we may assume that on an event $\Omega_1(A_n)$ measurable with respect to $A_n$ that holds with probability $\mathbb{P}(\Omega_1(A_n))\geq 1-\exp(-\Omega(n))$, we may find some $c_7>0$ depending only on $\kappa,\Delta$, $B$ and the log-Sobolev constant of $\zeta$ such that, for any $k\leq c_7n$ and any given label $J\in\{1,2\}$, we have the following estimate for the other label $j\neq J$:
\begin{equation}\label{goodbetterbest}
\mu_k(\lambda_j)\geq 10\mu_{c_j(J;k)}(\lambda_j).
\end{equation} and 
\begin{equation}\label{goodbetterbest2nd}
\mu_k(\lambda_J)\geq 10\mu_{c_j(J;k)}(\lambda_j).
\end{equation}
(To check this in detail, we first deduce that whenever $c_5>0$ is sufficiently small relative to $\Delta$ and $k\leq c_5n$, we must have that $c_j(J;k)\geq c_6n$ for some $c_6>0$ depending on $\Delta$  and this holds with probability $1-\exp(\Omega(n))$. Indeed, suppose to the contrapositive, then there are at most $(c_5+c_6)n$ eigenvalues of $A_n$ in $[\lambda_1,\lambda_2]$, which is super-exponentially unlikely when $c_5+c_6$ is too small, by Proposition \ref{proposition4.7}. Then we apply Lemma 
\ref{corollary4.88} to find some $c_7>0$ with $c_7<\max(c_5,c_6)$, such that for any $k\leq c_7n$, we have $\mu_k(\lambda_j)\geq \mu_{c_7n}(\lambda_j)\geq 10\mu_{c_6n}(\lambda_j)\geq 10\mu_{c_j(J;k)}(\lambda_j)$. This verifies \eqref{goodbetterbest}. For \eqref{goodbetterbest2nd}, we similarly use $\mu_k(\lambda_J)\geq \mu_{c_7n}(\lambda_J)\geq 10\mu_{c_6n}(\lambda_j)\geq 10\mu_{c_j(J;k)}(\lambda_j)$.)

Then by definition of $I$ we have $|\theta_I|\geq\sqrt{\theta_1\theta_2}=1$, and then using $\mu_{c_j(I;1)}(\lambda_j)\leq \mu_1(\lambda_j)/10$ by \eqref{goodbetterbest}, we deduce that
$$
\left|\theta_I+\sum_{j\neq I}\theta_j\frac{\mu_{c_j(I;1)(\lambda_j)}}{\mu_1(\lambda_j)}\right|\geq 0.9|\theta_I|\geq 0.9.
$$ 

Then we consider the soft constraint. Using \eqref{goodbetterbest2nd} and the definition of $J$, we have that for any $k\leq c_7n$ and the index $j\neq J$,
$$
|\theta_J|\frac{\mu_k(\lambda_J)}{\mu_1(\lambda_J)}\geq 10 |\theta_j|\frac{\mu_{c_j(J;k)}(\lambda_j)}{\mu_1(\lambda_j)}.
$$
That is, for any $k\leq c_7n$,
$$
\left|\theta_J\frac{\mu_k(\lambda_J)}{\mu_1(\lambda_J)}+\sum_{j\neq J}\theta_j\frac{\mu_{c_j(J;k)}(\lambda_j)}{\mu_1(\lambda_j)}\right|\geq 0.9\left|\theta_J\frac{\mu_k(\lambda_J)}{\mu_1(\lambda_J)}\right|.
$$

Applying both bounds to \eqref{mainmainmian}, we conclude that on $\Omega_1(A_n)$, we have
\begin{equation}\label{reduction1s2}\begin{aligned}&
\left\|\sum_{i=1}^2 \frac{\theta_i}{\mu_1(\lambda_i)}(A_n-\lambda_i I_n)^{-1}\widetilde{X}\right\|_2\geq 0.9  \langle \widetilde{X},v_1(\lambda_I)\rangle^2
 +0.4\sum_{k=2}^{c_7n} \left|\theta_J\frac{\mu_k(\lambda_J)}{\mu_1(\lambda_J)}\right|^2\langle \widetilde{X},v_k(\lambda_J)\rangle ^2.
\end{aligned}\end{equation}

 \textbf{Part 3: Applying the inverse Littlewood-Offord theorem.} We are now in the position to apply the conditioned Littlewood-Offord theorem, Theorem \ref{theorem7.4littlewoodchapter2}. We claim that whenever $A_n\in\mathcal{E}$ where $\mathcal{E}$ is the quasirandomness event defined in Lemma \ref{mainquasirandomness1}, then
 
 \begin{equation}\label{finsexpan}
\mathbb{P}_{\widetilde{X}}\left(\left\|\sum_{i=1}^2 \frac{\theta_i}{\mu_1(\lambda_i)}(A_n-\lambda_i I_n)^{-1}\widetilde{X}\right\|_2\leq s
\right) \lesssim s e^{-ck},\end{equation}
for any $e^{-cn}\leq s\leq  |\theta_J|\mu_k(\lambda_J)/\mu_1(\lambda_J)$ and any $k\leq c_7n$. This is because, using the reduction \eqref{reduction1s2}, the event in \eqref{finsexpan} can be decomposed as requiring 
$$|\langle \widetilde{X},v_1(\lambda_I)\rangle| \lesssim s,\quad\text{ and } \sum_{j=2}^{k} \left\|\theta_J\frac{\mu_j(\lambda_J)}{\mu_1(\lambda_J)}\right|^2\langle \widetilde{X},v_j(\lambda_J)\rangle ^2\lesssim s^2$$ for any $0\leq k\leq c_7n$, on an event with probability $1-\exp(-\Omega(n))$. By Lemma \ref{mainquasirandomness1}, on the event $\mathcal{E}$ we may assume all the eigenvectors $v$ of $A_{n}$ satisfy $D_{\alpha,\gamma}(v)\geq e^{c_3n}$.
Then using our assumption on $s$, we get that for the $k>0$ satisfying $s\leq|\theta_J|\mu_k(\lambda_J)/\mu_1(\lambda_J)$, we have by Theorem \ref{theorem7.4littlewoodchapter2} that
$$\begin{aligned}
\operatorname{LHS}\text{ of }\eqref{finsexpan}&\leq \mathbb{P}\left(|\langle\widetilde{X},v_1(\lambda_I)\rangle| \lesssim s,\quad \sum_{j=2}^{k} \langle \widetilde{X},v_j(\lambda_J)\rangle ^2\lesssim 1\right)+\exp(-\Omega(n))\\&\leq se^{-ck}+\exp(-\Omega(n)).\end{aligned}
$$
We finally take $\Omega(A_n):=\mathcal{E}\cap\Omega_1(A_n)$.

\textbf{Part 4: The choice of parameters.} Here, case (1) and (2) are estimated differently. For case (1), by definition of $J$ and our normalization on $\theta$, we have 
$$
|\theta_J|\frac{\mu_k(\lambda_J)}{\mu_1(\lambda_J)}\geq \mu_k(\lambda_J)\left(\prod_{j=1}^2|\theta_j|\frac{1}{\mu_1(\lambda_j)}\right)^{1/2}= \frac{\mu_k(\lambda_J)}{\left(\prod_{j=1}^2 \mu_1(\lambda_j)\right)^{1/2}}.
$$
Therefore, the estimate \eqref{finsexpan} holds for all $s$ in the range \eqref{wheredoesslies}.
This completes the proof, and remarkably, we get an estimate that does not quantitatively depend on $\theta$, and the dependence on $\theta$ is absorbed in the choice of the label $J$.

For case (2), suppose that $|\theta_J|=1$ then we take $I=J$ in the statement. Otherwise, we take $I\neq J$, then by definition we have $|\theta_I|=1$ and 
$$
|\theta_J|\frac{\mu_k(\lambda_J)}{\mu_1(\lambda_J)}\geq |\theta_I|\frac{\mu_k(\lambda_J)}{\mu_1(\lambda_I)}=\frac{\mu_k(\lambda_J)}{\mu_1(\lambda_I)} 
$$ and thus the estimate holds for $s$ in the range \eqref{secondsasess}.

\end{proof}

\section{A machinery for the bootstrap argument}

We have made enough preparations for proving the least singular value estimate, although the remaining steps are still highly technical. In the next two sections, we will adopt a bootstrap argument to upgrade the existing estimates we have in hand, and apply the bootstrap procedure finitely many times until we obtain the desired estimate. In this section we set up a general machinery for carrying out this bootstrapping argument.
 
 The main result of this chapter is the following lemma:

\begin{lemma}\label{lemma6.111} Let $\zeta\in\Gamma_B$ for some $B>0$ and have a finite Log-Sobolev constant, and let $A_n\sim\operatorname{Sym}_n(\zeta)$.  Fix any $\kappa>0$, $\Delta>0,$ and any two $\lambda_1,\lambda_2\in[(-(2-\kappa)\sqrt{n},(2-\kappa)\sqrt{n}]$ with $|\lambda_1-\lambda_2|\geq \Delta\sqrt{n}$. Then for any $\delta_1,\delta_2\geq e^{-cn}$, any unit vector $u\in\mathbb{S}^{n-1}$, and any $p>0$, we have the following estimate: 
\begin{equation}\label{section6mainests}\begin{aligned}
&\mathbb{E}_{A_n}\sup_{r_1,r_2\in\mathbb{R}}\mathbb{P}_X\left(\begin{cases}
    \frac{|\langle (A_n-\lambda_i I_n)^{-1}X,X\rangle-r_i|}{\|(A_n-\lambda_i I_n)^{-1}\|_*}\leq\delta_i,\text{  for each }i=1,2;\\
    \langle X,u\rangle\geq s,\\\frac{\mu_1(\lambda_1)\mu_1(\lambda_2)}{n}\leq(\delta_1\delta_2)^{-p}\end{cases}
    \right)\\&\lesssim e^{-s}\delta_1\delta_2+e^{-s}\delta_1\delta_2\mathbb{E}\left[\left(\frac{\mu_1(\lambda_1)\mu_1(\lambda_2)}{n}\right)^{9/10}
    \mathbf{1}\{ \frac{\mu_1(\lambda_1)\mu_1(\lambda_2)}{n}\leq (\delta_1\delta_2)^{-p}\}\right]^{80/81}\\&+e^{-\Omega(n)},\end{aligned} 
\end{equation} where the constant $c>0$ appearing in the condition $\delta_1,\delta_2\geq e^{-cn}$ only depends on $B,\kappa,\Delta$ and the Log-Sobolev constant of $\zeta$.
\end{lemma}

To better understand the meaning of this lemma, one may take $u=0$ for simplicity, so that the lemma says we can squeeze the small ball probability estimate of the complicated term $\langle (A_n-\lambda_i I_n)^{-1}X,X\rangle$ into the expectation of terms involving only the two quantities we care about: $\mu_1(\lambda_1)$ and $\mu_1(\lambda_2)$. All the other quantities have been averaged over and thus disappeared from the final estimate. To see why this lemma provides us with a bootstrap machinery, suppose that we already have a (sub-optimal) small ball probability estimate for $\sigma_{min}(A_n-\lambda_1I_n)\sigma_{min}(A_n-\lambda_2I_n)$ (the reciprocal of $\mu_1(\lambda_1)\mu_1(\lambda_2)$), then plugging this estimate to the expectation on the second line will often yield an improved estimate.
In later applications, we will set $p>1$ to be a constant sufficiently close to 1.

The proof of Lemma \ref{lemma6.111} essentially follows from combining Theorem \ref{theorem3.1} and Lemma \ref{lemma6.61}. The main technical work is to use different estimates for $\theta$ in different regions of the integral \eqref{eq330791}. This is explicitly spelled out in the following subsections.

\subsection{Initial decomposition}
Let $I(\theta)$ be the function defined in Theorem \ref{theorem3.1}. 

We apply Hölder's inequality to the definition of $I(\theta)$ and get \begin{equation}\label{holders}
I(\theta)\lesssim \left(\mathbb{E}_{J,X,X'}e^{9\langle (X+X')_J,u\rangle}\right)^{1/9}\left(\mathbb{E}_{J,X,X'}e^{-c'\|\sum_{i=1}^2 \frac{\theta_i}{\mu_1(\lambda_i)}(A_n-\lambda_i I_n)^{-1}(X-X')_J\|_2^2}
\right)^\frac{8}{9}
\end{equation} where $c'=c'(B)$ depends only on $B>0$ and $J$ is the $\mu$-random subset of $[n]$.
The first term on the right hand side is $O(1)$ since $\|u\|_2=1$ and $X,X'$ are sub-Gaussian random vectors. To estimate the expectation of the second term, we will do the estimate in a way depending on the location of $\theta\in\mathbb{R}^2:$ denoting by $\widetilde{X}:=(X-X')_J,$ then
$$\begin{aligned}
&I(\theta)^{9/8}\lesssim_B \mathbb{E}_{J,X,X'}e^{-c'\|\sum_{i=1}^2 \frac{\theta_i}{\mu_1(\lambda_i)}(A_n-\lambda_i I_n)^{-1}(X-X')_J\|_2^2}\\&\lesssim 
\begin{cases} 1,& |\theta_1|<1,|\theta_2|<1
\\e^{-c'|\theta_1|^{1/5}}+\mathbb{P}_{\widetilde{X}}\left(\left\|\sum_{i=1}^2 \frac{\theta_i}{\mu_1(\lambda_i)}(A_n-\lambda_i I_n)^{-1}\widetilde{X}\right\|_2\leq |\theta_1|^{\frac{1}{10}}\right),& |\theta_1|\geq 1,|\theta_2|< 1\\
e^{-c'|\theta_2|^{1/5}}+\mathbb{P}_{\widetilde{X}}\left(\left\|\sum_{i=1}^2 \frac{\theta_i}{\mu_1(\lambda_i)}(A_n-\lambda_i I_n)^{-1}\widetilde{X}\right\|_2\leq |\theta_2|^{\frac{1}{10}}\right),& |\theta_1|< 1,|\theta_2|\geq 1\\
e^{-c'|\theta_1\theta_2|^{1/11}}+\mathbb{P}_{\widetilde{X}}\left(\left\|\sum_{i=1}^2 \frac{\theta_i}{\mu_1(\lambda_i)}(A_n-\lambda_i I_n)^{-1}\widetilde{X}\right\|_2\leq |\theta_1\theta_2|^{\frac{1}{22}}\right),& |\theta_1|\geq 1,|\theta_2|\geq 1.
\end{cases}
\end{aligned}
$$

We split the integration of $I(\theta)^{1/2}$ over four regions, noticing that the four functions $1$, $e^{-c'|\theta_1|^{1/5}}$, $e^{-c'|\theta_2|^{1/5}}$ and $e^{-c'|\theta_1\theta_2|^{1/11}}$ all integrate to $O(1)$ on the respective regions:
\begin{equation}\label{decompositionformula}
\begin{aligned}
&\int_{\{\theta\in\mathbb{R}^2:\sum_{i=1}^2|\theta_i\delta_i|^2\leq 2\}}I(\theta)^{1/2}d\theta\\&
\lesssim 1+\int_{|\theta_1|\leq 1,1\leq|\theta_2|\leq2\delta_2^{-1}}\mathbb{P}_{\widetilde{X}}(\|\sum_{i=1}^2 \frac{\theta_i}{\mu_1(\lambda_i)}(A_n-\lambda_i I_n)^{-1}\widetilde{X}\|_2\leq  |\theta_2|^\frac{1}{10})^{4/9}d\theta\\&
+ \int_{1\leq |\theta_1|\leq 2\delta_1^{-1},|\theta_2|\leq 1}\mathbb{P}_{\widetilde{X}}(\|\sum_{i=1}^2 \frac{\theta_i}{\mu_1(\lambda_i)}(A_n-\lambda_i I_n)^{-1}\widetilde{X}\|_2\leq |\theta_1|^\frac{1}{10})^{4/9}d\theta\\&+
\int_{1\leq |\theta_1|\leq 2\delta_1^{-1},1\leq|\theta_2|\leq2\delta_2^{-1}}\mathbb{P}_{\widetilde{X}}(\|\sum_{i=1}^2 \frac{\theta_i}{\mu_1(\lambda_i)}(A_n-\lambda_i I_n)^{-1}\widetilde{X}\|_2\leq (\theta_1\theta_2)^\frac{1}{22})^{4/9}d\theta\\&:=1+I_1+I_2+I_3.
\end{aligned}\end{equation}

Informally speaking, the integrals $I_1$ and $I_2$ represent terms that rely essentially on one location $\lambda_i$, and the integral $I_3$ consists of terms that rely genuinely on both locations.

\subsection{Estimating essentially 1-d terms} 
The computations in this section are very similar to those in \cite{campos2024least}, Section 9 and Section 10, reflecting that $I_1$ and $I_2$ are essentially 1-d objects.

We can rewrite the term $I_1$ as follows: 
$$\begin{aligned}
I_1&=\int_{|\theta_1|\leq 1,1\leq|\theta_2|\leq2\delta_2^{-1}}\mathbb{P}_{\widetilde{X}}\left(\|
\sum_{i=1}^2\frac{\theta_i/\theta_2}{\mu_1(\lambda_i)}(A_n-\lambda_iI_n)^{-1}\widetilde{X}
\|_2\leq|\theta_2|^{-\frac{9}{10}}\right)^{4/9}d\theta\\&
\lesssim \int_1^{\frac{2}{\delta_2}}\int_{-1}^1\mathbb{P}_{\widetilde{X}} \left(\|
\sum_{i=1}^2\frac{\theta_1/s\mathbf{1}(i=1)+\mathbf{1}(i=2)}{\mu_1(\lambda_i)}(A_n-\lambda_iI_n)^{-1}\widetilde{X}
\|_2\leq s^{-\frac{9}{10}}\right)^{4/9}d\theta_1ds\\&
\leq\int_{\frac{\delta_2}{2}}^1 s^{-19/9} \int_{-1}^1\mathbb{P}_{\widetilde{X}} \left(\|
\sum_{i=1}^2\frac{\theta_1s^{\frac{10}{9}}\mathbf{1}(i=1)+\mathbf{1}(i=2)}{\mu_1(\lambda_i)}(A_n-\lambda_iI_n)^{-1}\widetilde{X}
\|_2\leq s\right)^{4/9} d\theta_1ds.
\end{aligned}$$
 
Recall that the quantitative estimate of Lemma \ref{lemma6.61},(2) is independent of $\theta$, and the dependence on $\theta$ is via the subscript $I,J\in\{1,2\}$ which depends only on $A_n$ and the relative ratios of $\theta_1$ and $\theta_2$. To clarify notations, let $\Sigma_{ij},i,j=1,2$  be subsets of the parameter space $(\theta_1,\theta_2)$, $|\theta_1|\leq 1$ and $|\theta_2|=1$, such that on $\Sigma_{ij}$, Lemma \ref{lemma6.61} yields $I=i,J=j$. The subsets $\Sigma_{ij}$ are measurable with respect to $A_n$ but independent of $\widetilde{X}$.

Now we use the shorthand notation $\theta_1'=\theta_1s^\frac{10}{9}$, $\theta_2'=1$ and write 
$$
I_1\lesssim \sum_{\ell\in\{1,2\}^2} \int_{\frac{\delta_2}{2}}^1 \int_{-1}^1s^{-19/9} 1_{(\theta_1',\theta_2')\in\Sigma_\ell}\mathbb{P}_{\widetilde{X}} \left(\|
\sum_{i=1}^2\frac{\theta_i'}{\mu_1(\lambda_i)}(A_n-\lambda_iI_n)^{-1}\widetilde{X}
\|_2\leq s\right)^{4/9} d\theta_1ds.
$$

For any $I,J\in\{1,2\}$ the interval $[\delta_2/2,1]$ is contained in a union of intervals $$[e^{-cn},\frac{\mu_{{c_0}_{\eqref{lemma6.61}}n}(\lambda_J)}{\mu_1(\lambda_I)}]\cup \left(\cup_{k=2}^{{c_0}_{\eqref{lemma6.61}}n}[\frac{\mu_k(\lambda_J)}{\mu_1(\lambda_I)},\frac{\mu_{k-1}(\lambda_J)}{\mu_1(\lambda_I)}]\right),$$ and we upper bound the integral in $I_1$ by extending the range of integration  from $s\in[\delta_2/2,1]$ to $s\in[e^{-cn},1]$.
 Then by Lemma \ref{lemma6.61},(2) we have after integrating over $\theta_1\in[-1,1]$ that, whenever $A_n\in\Omega(A_n)$, (where $\Omega(A_n)$ was defined in Lemma \ref{lemma6.61})
\begin{equation}\label{whatisi1}
\begin{aligned}
I_1&\lesssim\sum_{I,J=1}^2\sum_{k=1}^{{c_0}_{\eqref{lemma6.61}}n}\int_{\mu_{k}(\lambda_J)/\mu_1(\lambda_I)}^{\mu_{k-1}(\lambda_J)/\mu_1(\lambda_I)} e^{-ck} s^{-19/9+4/9}ds+\int_{e^{-cn}}^{\mu_{{c_0}_{\eqref{lemma6.61}}n}(\lambda_J)/\mu_1(\lambda_I)} e^{-cn}s^{-19/9+4/9} ds\\&
\leq \sum_{I,J=1}^2 \sum_{k=1}^{{c_0}_{\eqref{lemma6.61}}n}e^{-ck}(\mu_1(\lambda_J)/\mu_k(\lambda_{I}))^{2/3}+2e^{-cn},
\end{aligned}\end{equation} where the constant $c>0$ depends only on $B$.
    The same upper bound also applies to $I_2$.

 We then prove the following technical lemma that gives an estimate for $I_1$:
\begin{lemma}\label{whatislemma1} For any $B>0,$ let $\zeta\in\Gamma_B$, and $A_n\sim\operatorname{Sym}_n(\zeta)$.
For any $\delta_1,\delta_2\geq e^{-cn}$, define $$\mathcal{E}(\delta_1\delta_2):=\left\{\frac{\mu_1(\lambda_1)}{\sqrt{n}}\leq (\delta_1\delta_2)^{-1},\quad \frac{\mu_1(\lambda_2)}{\sqrt{n}}\leq (\delta_1\delta_2)^{-1}\right\}.$$     Then for any $p>0$ and any given $c'\in(0,1]$,
\begin{equation}\label{6.3wowsog}
\mathbb{E}_{A_n}^{\mathcal{E}(\delta_1^p\delta_2^p)}\left[ \prod_{i=1}^2 \frac{\|(A_n-\lambda_iI_n)^{-1}\|_*}{\mu_1(\lambda_i)}\sum_{\ell_1=1}^2 \sum_{\ell_2=1}^2\sum_{k=1}^{c'n}e^{-ck}\left(\frac{\mu_1(\lambda_{\ell_1})}{\mu_k(\lambda_{\ell_2})}\right)^{2/3}+e^{-\Omega(n)}\right]\lesssim 1,
\end{equation} where for a measurable event $\mathcal{E}$, we use the symbol $\mathbb{E}_{A_n}^\mathcal{E}$ to denote the expectation of the probability law of $A_n$ taken over those $A_n\in\mathcal{E}$.
\end{lemma}

\begin{proof}[\proofname\ of Lemma \ref{whatislemma1}]  We use the abbreviation $\mathcal{E}_0:=\mathcal{E}(\delta_1^p\delta_2^p)$.
    For any fixed $\ell_1,\ell_2\in\{1,2\}$ we first estimate via Hölder's inequality that
    $$\begin{aligned}
&\mathbb{E}_{A_n}^{\mathcal{E}(\delta_1^p\delta_2^p)}
\prod_{i=1}^2 \frac{\|(A_n-\lambda_iI_n)^{-1}\|_*}{\mu_1(\lambda_i)}\left(\frac{\mu_1(\lambda_{\ell_1})}{\mu_k(\lambda_{\ell_2})}\right)^{2/3}\\&\leq \prod_{i=1}^2\mathbb{E}_{A_n}^{\mathcal{E}_0}\left[\left(\frac{\|(A_n-\lambda_i I_n)^{-1}\|_*}{\mu_1(\lambda_i)}\right)^{28}\right]^{1/28}\\&\quad \cdot\mathbb{E}_{A_n}^{\mathcal{E}_0}\left[(\frac{\sqrt{n}}{k\mu_k(\lambda_{\ell_1})})^{28/3}\right]^{1/14}\mathbb{E}_{A_n}^{\mathcal{E}_0}\left[(\frac{\mu_1(\lambda_{\ell_2})}{\sqrt{n}})^{7/9}
\right]^{6/7} \cdot k^{2/3}
\\&\lesssim k^{2/3}\mathbb{E}_{A_n}^{\mathcal{E}_0}\left[(\frac{\mu_1(\lambda_{\ell_2})}{\sqrt{n}})^{7/9}
\right]^{6/7}
    \end{aligned}$$ where we use Lemma \ref{lemma4.11} and Lemma \ref{lemma4.22} to show that the first two expectations after the first inequality are $O(1)$.

Since $\sum_{k\geq 1}k^{2/3}e^{-ck}=O(1)$, taking the summation over $k$ in the previous estimate, 
\begin{equation}
    \text{LHS of } \eqref{6.3wowsog} \lesssim \mathbb{E}_{A_n}^{\mathcal{E}_0}\left[(\frac{\mu_1(\lambda_{1})}{\sqrt{n}})^{7/9}\right]^{6/7}+\mathbb{E}_{A_n}^{\mathcal{E}_0}\left[(\frac{\mu_1(\lambda_{2})}{\sqrt{n}})^{7/9}\right]^{6/7}.
\end{equation}

   Finally, we apply Proposition \ref{proposition6.666} and get
    $$\begin{aligned}
\mathbb{E}_{A_n}^{\mathcal{E}_0}[(\frac{\mu_1(\lambda_{1})}{\sqrt{n}})^{7/9}]&\lesssim 1+\int_{1}^{(\delta_1\delta_2)^{-7p/9}}\mathbb{P}(\mu_1(\lambda_1)\geq s^{9/7}\sqrt{n})ds\\&\lesssim 1+\int_1^{(\delta_1\delta_2)^{-7p/9}} (s^{-9/7}+e^{-cn}) ds\lesssim 1,
    \end{aligned}$$ The same computation holds for $\mu_1(\lambda_2)$, thus completing the proof of the lemma.
\end{proof}

\subsection{Estimating genuinely 2-d terms}

In this section we bound $I_3$. Recall that 
$$I_3=\int_{1\leq |\theta_1|\leq 2\delta_1^{-1},1\leq|\theta_2|\leq2\delta_2^{-1}}\mathbb{P}_{\widetilde{X}}\left(\|\sum_{i=1}^2 \frac{\theta_i}{\mu_1(\lambda_i)}(A_n-\lambda_i I_n)^{-1}\widetilde{X}\|_2\leq (\theta_1\theta_2)^\frac{1}{22}\right)^{4/9}d\theta.$$

First we normalize the probability inside so that $\theta_1\theta_2=1$. That is, 
$$
I_3=\int_{1}^{2/\sqrt{\delta_1\delta_2}}\int_{|\theta_1\theta_2|=s^2,1\leq|\theta_i|}\mathbb{P}_{\widetilde{X}}
\left(\|\sum_{i=1}^2 \frac{\theta_i(A_n-\lambda_i I_n)^{-1}\widetilde{X}}{\mu_1(\lambda_i)\sqrt{\theta_1\theta_2}}\|_2\leq s^{-\frac{10}{11}}\right)^{4/9}ds,$$where we integrate over the hyperplane $\Lambda_s$ defined as $$\Lambda_s:=\{(\theta_1,\theta_2)\in\mathbb{R}^2:|\theta_1\theta_2|= s^2,|\theta_1|\geq 1,|\theta_2|\geq 1\}.$$

We will use Lemma \ref{lemma6.61} to estimate the probability inside. We again define $\Sigma_1$, $\Sigma_2$ as the two complementary subsets that partition the hyperplane $$\{\theta\in\mathbb{R}^2:|\theta_1\theta_2|=1\},$$ so that on $\Sigma_1$,  Lemma \ref{lemma6.61} yields $J=1$ and on $\Sigma_2$,  Lemma \ref{lemma6.61} yields $J=2$ (note that $\Sigma_1,\Sigma_2$ are measurable with respect to $A_n$). Then we further bound
$$\begin{aligned}
I_3\lesssim &\int_1^{2/\sqrt{\delta_1\delta_2}}\int_{\Lambda_s}\sum_{\ell=1}^2\mathbf{1}_{\{(\frac{\theta_1}{\sqrt{\theta_1\theta_2}},\frac{\theta_2}{\sqrt{\theta_1\theta_2}})\in\Sigma_\ell\}}\\&\cdot\mathbb{P}_{\widetilde{X}}\left(\|\sum_{i=1}^2 \frac{\theta_i/\sqrt{\theta_1\theta_2}}{\mu_1(\lambda_i)}(A_n-\lambda_i I_n)^{-1}\widetilde{X}\|_2\leq s^{-\frac{10}{11}}\right)^{4/9}ds.\end{aligned}
$$

We make a change of coordinate $(\theta_1,\theta_2)\mapsto(\theta_1,\sqrt{\theta_1\theta_2})$ and integrate over $\theta_1$ and $s=\sqrt{\theta_1\theta_2}$. Then $I_3$ is bounded as follows: where we require that $\theta_2$ satisfies $\theta_1\theta_2=s^2$,
$$\begin{aligned}
I_3\lesssim&\int_1^{2/\sqrt{\delta_1\delta_2}}\int_1^{s^2}\sum_{\ell=1}^2 \frac{s}{\theta_1} \cdot \mathbf{1}_{\{(\frac{\theta_1}{s},\frac{\theta_2}{s})\in \Sigma_\ell\}}\\&\cdot\mathbb{P}_{\widetilde{X}}\left(\|\sum_{i=1}^2 \frac{\theta_i/s}{\mu_1(\lambda_i)}(A_n-\lambda_i I_n)^{-1}\widetilde{X}\|_2\leq s^{-\frac{10}{11}}\right)^{4/9}
d\theta_1ds,
\end{aligned}$$ where the $\frac{s}{\theta_1}$ factor comes from the Jacobian of the change of coordinate $(\theta_1,\theta_2)\mapsto(\theta_1,s)$.

Changing the coordinates the second time by setting $s\mapsto s^{-\frac{11}{10}}$, we obtain 
\begin{equation}\label{integration97000}\begin{aligned}
I_3\lesssim& 
\sum_{\ell=1}^2\int_{\sqrt{\delta_1\delta_2}/2}^1 \frac{s^{-32/10}}{\theta_1}\int_1^{s^{-22/10}} \mathbf{1}_{\{(\frac{\theta_1}{s},\frac{\theta_2}{s})\in\Sigma_\ell\}} \\&\cdot\mathbb{P}_{\widetilde{X}}\left(\|\sum_{i=1}^2 \frac{\theta_i/s}{\mu_1(\lambda_i)}(A_n-\lambda_i I_n)^{-1}\widetilde{X}\|_2\leq s\right)^{4/9}d\theta_1ds.
\end{aligned}\end{equation} 

Now we prepare for the use of lemma \ref{lemma6.61}. We consider the following decomposition, taking $J=1$ on $\Sigma_1$ and $J=2$ on $\Sigma_2$: \begin{equation}\label{decomposition2d}[\sqrt{\delta_1\delta_2}/2,1]\subset\left[e^{-cn},\frac{\mu_{{c_0}_{\eqref{lemma6.61}}n}(\lambda_J)}{\sqrt{\mu_1(\lambda_1)\mu_1(\lambda_2})}\right]\cup \bigcup_{k=2}^{{c_0}_{\eqref{lemma6.61}}n}\left[\frac{\mu_{k}(\lambda_J)}{\sqrt{\mu_1(\lambda_1)\mu_1(\lambda_2)}},\frac{\mu_{k-1}(\lambda_J)}{\sqrt{\mu_1(\lambda_1)\mu_1(\lambda_2})}\right].\end{equation} (Note: if we already have $\frac{\mu_{k-1}(\lambda_J)}{\sqrt{\mu_1(\lambda_1)\mu_1(\lambda_2})}\geq 1$ for some $k\geq 0$, then we terminate the right end of the interval at 1; if the whole interval lies on the right hand side of 1 then we discard that interval).

For $s\in(0,1)$ and $J=1$ or 2 let $k_s(J)$  be the largest integer $k$ in $[1,{c_0}_{\eqref{lemma6.61}}n]$ such that $s\leq \frac{\mu_k(\lambda_J)}{\sqrt{\mu_k(\lambda_1)\mu_k(\lambda_2)}}$. If no such $k$ exists, then set $k_s(J)={c_0}_{\eqref{lemma6.61}}n$. Whenever $A_n\in\Omega(A_n)$, proceed as follows.
Apply Lemma \ref{lemma6.61}(1) for $s$ integrated over each interval of the decomposition on the right hand side of \eqref{decomposition2d} and for both cases $\ell=1,2$ such that $(\frac{\theta_1}{s},\frac{\theta_2}{s})\in\Sigma_\ell$, we can upper bound $I_3$ by
$$
I_3\lesssim\sum_{\ell=1}^2\int_{s=e^{-cn}}^1\frac{s^{-32/10}}{\theta_1}\int_{\theta_1=1}^{s^{-22/10}}s^{4/9}e^{-ck_s(\ell)}d\theta_1ds.
$$

Then we integrate over $\theta_1$, decompose the range of integration of $s$ and obtain
\begin{equation}\begin{aligned}\label{definitionI3}
I_3\lesssim&\sum_{\ell=1}^2\sum_{k=2}^{{c_0}_{\eqref{lemma6.61}}n}e^{-ck}\int_{\frac{\mu_{k}(\lambda_\ell)}{\sqrt{\mu_1(\lambda_1)\mu_1(\lambda_2})}}^{\frac{\mu_{k-1}(\lambda_\ell)}{\sqrt{\mu_1(\lambda_1)\mu_1(\lambda_2})}} s^{-124/45}\ln s^{-1} ds\\&+\sum_{\ell=1}^2 e^{-c{c_0}_{\eqref{lemma6.61}}n}\int_{e^{-cn}}^{\frac{\mu_{{c_0}_{\eqref{lemma6.61}}n}(\lambda_\ell)}{\sqrt{\mu_1(\lambda_1)\mu_1(\lambda_2})}}s^{-124/45}\ln s^{-1}ds,\end{aligned}
\end{equation} where the $\ln s^{-1}$ term arises from the integral of $\theta_1$. In the following, we use the elementary bound $s^{-124/45}\ln s^{-1}\lesssim s^{-25/9}$ for $s\in(0,1)$.
The second term on the right hand side is $e^{-\Omega(n)}$ . For the first term, we integrate over $s$ and obtain 
\begin{equation}\label{whatisi3}
I_3\lesssim\sum_{\ell=1}^2\sum_{k=2}^{{c_0}_{\eqref{lemma6.61}}n}e^{-ck}\left(\frac{\sqrt{\mu_1(\lambda_1)\mu_1(\lambda_2)}}{\mu_k(\lambda_\ell)}\right)^{16/9}+e^{-\Omega(n)}.
\end{equation}
At this point notice that each individual $\mu_i(\lambda_1),i=1,2$ has a power of $\frac{8}{9}<1$, but the overall power is greater than one. In the following lemma we use a much finer estimate to obtain an effective bound of $I_3$.

\begin{lemma}\label{whatislemma2}
    For any $B>0$, and any $\zeta\in\Gamma_B$, consider the random matrix $A_n\sim\operatorname{Sym}_n(\zeta)$. For any 
    $p>0$ and $\delta_1,\delta_2>0$ define 
    $$\mathcal{E}(p,\delta_1\delta_2):=\left\{\frac{\mu_1(\lambda_1)\mu_1(\lambda_2)}{n}\leq (\delta_1\delta_2)^{-p}\right\}.$$ Then for any given $c'\in(0,1]$, 
    \begin{equation}\begin{aligned}
    \mathbb{E}_{A_n}^{\mathcal{E}(p,\delta_1\delta_2)} &\left[\prod_{i=1}^2 \frac{\|(A_n-\lambda_iI_n)^{-1}\|_*}{\mu_1(\lambda_i)}\sum_{\ell=1}^2\sum_{k=2}^{c'n}e^{-ck}\left(\frac{\sqrt{\mu_1(\lambda_1)\mu_1(\lambda_2)}}{\mu_k(\lambda_\ell)}\right)^{16/9}
        \right]
        \\&\lesssim 1+\mathbb{E}_{A_n}^{\mathcal{E}(p,\delta_1\delta_2)} \left[\left(\frac{\mu_1(\lambda_1)\mu_1(\lambda_2)}{n}\right)^{9/10}
\right]^{80/81}.
        \end{aligned}
    \end{equation}
\end{lemma}

\begin{proof} Denote by $\mathcal{E}_{0,p}:=\mathcal{E}(p,\delta_1\delta_2)$.
    The claim again follows from computing
       $$\begin{aligned}
\mathbb{E}_{A_n}^{\mathcal{E}_0,p}&
\left[\prod_{i=1}^2 \frac{\|(A_n-\lambda_iI_n)^{-1}\|_*}{\mu_1(\lambda_i)}\cdot \frac{\mu_1(\lambda_{1})^{8/9}\mu_1(\lambda_2)^{8/9}}{\mu_k(\lambda_\ell)^{16/9}}\right]\\&\leq \prod_{i=1}^2\mathbb{E}_{A_n}^{\mathcal{E}_{0,p}}\left[\left(\frac{\|(A_n-\lambda_i I_n)^{-1}\|_*}{\mu_1(\lambda_i)}\right)^{324}\right]^{1/324}\cdot \mathbb{E}_{A_n}^{\mathcal{E}_{0,p}}\left[\left(\frac{\sqrt{n}}{k\mu_k(\lambda_{\ell})}\right)^{288}\right]^{1/162}\\&\quad\quad \cdot \mathbb{E}_{A_n}^{\mathcal{E}_{0,p}}\left[\left(\frac{\mu_1(\lambda_1)\mu_1(\lambda_2)}{n}\right)^{9/10}
\right]^{80/81} \cdot k^{16/9}
\\&\lesssim k^2\mathbb{E}_{A_n}^{\mathcal{E}_{0,p}}\left[\left(\frac{\mu_1(\lambda_1)\mu_1(\lambda_2)}{n}\right)^{9/10}
\right]^{80/81},
    \end{aligned}$$ where we use Lemma \ref{lemma4.11} and Lemma \ref{lemma4.22} to show that the two terms on the second line are $O(1)$. By the elementary estimate $\sum_{k\geq 1}k^2e^{-ck}<+\infty,$ the proof is complete.
\end{proof}

\subsection{Conclusion}
Now we are ready to prove Lemma \ref{lemma6.111}.

\begin{proof}[\proofname\ of Lemma \ref{lemma6.111}]
    We make a further conditioning and write 

    $$\begin{aligned}
&\mathbb{E}_{A_n}\sup_{r_1,r_2}\mathbb{P}_X\left(
    \frac{|\langle (A_n-\lambda_i I_n)^{-1}X,X\rangle-r_i|}{\|(A_n-\lambda_i I_n)^{-1}\|_*}\leq\delta_i
    ,\langle X,u\rangle\geq s,\frac{\mu_1(\lambda_1)\mu_1(\lambda_2)}{n}\leq(\delta_1\delta_2)^{-p}
    \right)\\& \leq \mathbb{P}(\sigma_{min}(A_n-\lambda_1I_n)
\leq \delta_1\delta_2 n^{-1/2})+ \mathbb{P}(\sigma_{min}(A_n-\lambda_2I_n)
\leq \delta_1\delta_2 n^{-1/2})\\&+\mathbb{E}_{A_n}
\sup_{r_1,r_2}\mathbb{P}_{X}\left(\bigcap_{i=1}^2 \left\{\left\|\frac{\langle (A_n-\lambda_i I_n)^{-1}X,X\rangle-r_i}{\|(A_n-\lambda_i I_n)^{-1}\|_*}\right\|\leq \delta_i\right\}\wedge \mathcal{E}(p;2)\right)\mathbf{1}(A_n\in\Omega(A_n))\\&+\mathbb{P}(A_n\notin\Omega(A_n)),
    \end{aligned}$$
where we define $$\mathcal{E}(p;2):=\{\sigma_{min}(A_n-\lambda_iI_n)\geq\delta_1\delta_2n^{-1/2},i=1,2;\quad \frac{\mu_1(\lambda_1)\mu_1(\lambda_2)}{n}\leq (\delta_1\delta_2)^{-p}\}.$$

    The probability $\mathbb{P}(\sigma_{min}(A_n-\lambda_1I_n)
\leq \delta_1\delta_2 n^{-1/2})$ is easily bounded by $C\delta_1\delta_2+e^{-\Omega(n)}$ by an application of Proposition \ref{proposition6.666}.

To compute the probability on the third line, we use Theorem \ref{theorem3.1} to write the probability in question as the integral of $I(\theta)^{1/2}$ over the given region. Then we use the Hölder inequality estimate \eqref{holders} combined with decomposition \eqref{decompositionformula}, and the computations in \eqref{whatisi1}, and \eqref{whatisi3}, as well as estimates in Lemma \ref{whatislemma1} and Lemma \ref{whatislemma2} to obtain the desired result. The last fact that $\mathbb{P}(A_n\notin\Omega(A_n))\leq\exp(-\Omega(n))$, follows from its definition in Lemma \ref{lemma6.61}.

\end{proof}

\section{Bootstrapping the estimate and proof completion}\label{bootstrapping209}

In this section we complete the proof of our main theorem, Theorem \ref{Theorem1.1}. We will use Lemma \ref{lemma6.111} to iteratively improve the accuracy of our estimate.

\subsection{Initial estimate and bootstrapping}
We take a first step towards the estimate by obtaining a factor $(\delta_1\delta_2)^{1/10}$ for the probability $\mathbb{P}(\sigma_{min}(A_n-\lambda_i I_n)\leq\delta_in^{-1/2},i=1,2)$. Although this is far from the optimal estimate which should have a factor $\delta_1\delta_2$, this sub-optimal estimate is a very important input for further arguments.

We shall use the concentration inequality named after Hanson and Wright \cite{wright1973bound} \cite{hanson1971bound}, see also \cite{vershynin2018high}, Theorem 6.2.1.

\begin{theorem}[Hanson-Wright]\label{hansonwright8.1}
    Given $B>0$, $\zeta\in\Gamma_B$ and consider a random vector $X\sim\operatorname{Col}_n(\zeta)$. Let $M$ be an $n\times n$ matrix. Then for any given $t>0$, we have
    $$
\mathbb{P}_X(|\|MX\|_2-\|M\|_{HS}|>t)\leq 4\exp\left(-c\min(\frac{t^2}{B^4\|M\|_{HS}^2},\frac{t}{B^2\|M\|})\right)
    $$ for a universal constant $c>0$.
\end{theorem}

We will use this Hanson-Wright inequality to upper bound $\|(A_n-\lambda_i I_n)^{-1}X\|_2$ by the Hilbert-Schmidt norm $\|(A_n-\lambda_i I_n)^{-1}\|_{HS}$. However, the concentration in Hanson-Wright inequality is not strong enough to guarantee a probability upper bound $O(\delta_1\delta_2)$ unless we take $t$ sufficiently large: this is precisely what we will do by taking $t\sim\log(\delta_1\delta_2)^{-1}\|(A_n-\lambda_i I_n)^{-1}\|_{HS}$. This usage of Hanson-Wright inequality clearly leads to some suboptimality as we introduce an additional log factor, and we will remove this factor in Section \ref{removalladfaga}.

We can then perform the initial estimate as follows:

\begin{lemma}\label{lemma7.27} 
Fix $B>0,\zeta\in\Gamma_B$ with a finite Log-Sobolev constant and $A_n\sim\operatorname{Sym}_n(\zeta)$.
Fix $\kappa>0,\Delta>0$ and consider two locations $\lambda_1,\lambda_2\in[-(2-\kappa)\sqrt{n},(2-\kappa)\sqrt{n}]$ with $|\lambda_1-\lambda_2|\geq\Delta\sqrt{n}$. Then we have the following initial estimate:
for any $\delta_1,\delta_2>0$,
$$
\mathbb{P}\left(\sigma_{min}(A_{n+1}-\lambda_i I_{n+1})\leq \delta_i n^{-1/2}, \quad i=1,2 \right)\lesssim (\delta_1\delta_2)^{1/10}+e^{-\Omega(n)}.
$$
\end{lemma}
The rate $\epsilon^{1/10}$ is rather weak, but we will bootstrap it to the optimal rate later.

\begin{proof}
By Proposition \ref{finalfuckpropositionga}, it suffices to prove that for any $(r_1,r_2)\in\mathbb{R}^2$ and some $p>1$,
\begin{equation}\begin{aligned}
    &\mathbb{P}_{A_n,X}\left(
    \frac{|\langle (A_n-\lambda_i I_n)^{-1}X,X\rangle-r_i|}{\|(A_n-\lambda_iI_n)^{-1}X\|_2}\leq C\delta_i,  i=1,2,  \prod_{i=1}^2\sigma_{min}(A_n-\lambda_iI_n)\geq (\delta_1\delta_2)^pn^{-1}
    \right)\\&\lesssim (\delta_1\delta_2)^{1/10}+e^{-\Omega(n)}.
\end{aligned}\end{equation}
Applying Hanson-Wright (Theorem \ref{hansonwright8.1}), there is a $C'>0$ satisfying, for $i=1,2$,

\begin{equation}
    \mathbb{P}_X(\|(A_n-\lambda_i I_n)^{-1}X\|_2\geq C'\log(\delta_1\delta_2)^{-1}\cdot\|(A_n-\lambda_i I_n)^{-1}\|_{HS})\leq \delta_1\delta_2,
\end{equation}
so it suffices to bound, for any $(\theta_1,\theta_2)\in\mathbb{R}^2$,
\begin{equation}\label{sufficestobound}
\mathbb{P}_{A_n,X}\left(\frac{|\langle (A_n-\lambda_i I_n)^{-1}X,X \rangle-r_i|}{\|(A_n-\lambda_i I_n)^{-1}\|_{HS}}\leq\bar{\delta}_i,i=1,2,\prod_{i=1}^2\sigma_{min}(A_n-\lambda_i I_n)\geq \delta_in^{-1/2}\right),\end{equation} where we set \begin{equation}\label{whatisbardelta}\bar{\delta}_i:=C^{'}\delta_i\log(\delta_1\delta_2)^{-1},\quad i=1,2.\end{equation} Using the fact that $\|M\|_*\geq \|M\|_{HS}$ for any matrix $M$, it then suffices to bound
\begin{equation}
\mathbb{P}_{A_n,X}\left(\frac{|\langle (A_n-\lambda_i I_n)^{-1}X,X \rangle-r_i|}{\|(A_n-\lambda_i I_n)^{-1}\|_*}\leq\bar{\delta}_i,i=1,2,\prod_{i=1}^2\sigma_{min}(A_n-\lambda_i I_n)\geq \delta_in^{-1/2}\right).\end{equation}

Then we apply Lemma \ref{lemma6.111} (where we choose $p=1.01$) and conclude that
\begin{equation}
    \eqref{sufficestobound}\lesssim \bar{\delta}_1\bar{\delta_2}(\delta_1\delta_2)^{-8p/9}+e^{-\Omega(n)}\lesssim (\delta_1\delta_2)^{1/10}+e^{-\Omega(n)}.
\end{equation}
More precisely, in the second line of the estimate \eqref{lemma6.111}, we directly use the upper bound $\mu_1(\delta_1)\mu_2(\delta_2)/n\leq(\delta_1\delta_2)^{-1.01}$ to obtain an upper bound for the entire expectation.
This completes the proof.
\end{proof}

\subsection{A bootstrap machinery} When we wish to run a bootstrap argument, an incompatibility arises. The estimates we already have (such as Lemma \ref{lemma7.27}) are in terms of the probability $\mathbb{P}(\sigma_{min}(A_{n+1}-\lambda_i I_{n+1})\leq\delta_in^{-1/2},i=1,2)$; but the conditioning in the estimate of Lemma \ref{lemma6.111} is via the product singular value $\mu_1(\lambda_1)\mu_1(\lambda_2)/n$. This is not a major obstacle as an estimate of the former immediately yields an estimate for the latter.
This is done in the following elementary lemma:

\begin{lemma}\label{singularvalueproduct} Fix some constants $\sigma>0,c>0,C>0$ and $\tau\in(0,1].$ Let $M_1,M_2$ be two $n\times n$ random matrices (not necessarily independent) that satisfy, for any $\delta_1,\delta_2>0$,
\begin{equation}\label{assumption}
    \mathbb{P}(\sigma_{min}(M_1)\leq\delta_1 n^{-1/2}, \sigma_{min}(M_2)\leq\delta_2 n^{-1/2})\leq C(\delta_1\delta_2)^\tau+e^{-cn^\sigma},
\end{equation}
and for any $k\in\{1,2\}$ it satisfies 
\begin{equation}\label{assumption2}
    \mathbb{P}(\sigma_{min}(M_{k})\leq\delta_{k} n^{-1/2})\leq C(\delta_{k})^\tau+e^{-cn^\sigma}.
\end{equation}
Then for any $\delta>0$, we have
$$
\mathbb{P}\left(\sqrt{\sigma_{min}(M_1) \sigma_{min}(M_2)}\leq \delta n^{-1/2}\right)\leq C_0C(\delta^\tau\log(\delta^{-1}))^2+e^{-cn^\sigma},
$$where $C_0>0$ is a universal constant.

\end{lemma}

\begin{proof}
    Assume that $\frac{1}{2}\geq \delta\geq e^{-cn^\sigma}$; otherwise, the claim is trivial. 
    
    We decompose $[\delta^2,1]$ into dyadic intervals 
    \begin{equation}\label{dyadic}[\delta^2,1]=\cup_{j=1}^{L_\delta} [I_{j,S},I_{j,L}]
    \end{equation} where for each $j\geq 2$, we assume $I_{j,L}/I_{j,S}=2$ and we assume  $1\leq I_{1,L}/I_{1,S}\leq 2$. Then it is elementary to check that $L_\delta\leq C_0\log\delta^{-1}$, for a universal constant $C_0>0$.

The event $\{\sigma_{min}(M_1)\sigma_{min}(M_2)\leq\delta^2n^{-1}\}$ lies in the union of a collection of events $\{\sigma_{min}(M_1)\leq\delta_1n^{-1/2},\sigma_{min}(M_2)\leq\delta_2n^{-1/2}\}$ for some $\delta_1,\delta_2>0$. We first claim that we can assume $\delta_1,\delta_2\in[\delta^2,1]$. Suppose the otherwise, then either for some $i$ (say, $i=1$) $\sigma_{min}(M_i)\geq n^{-1/2}$, then the event $\sigma_{min}(M_1)\sigma_{min} (M_2)\leq\delta^2 n^{-1}$ implies $\sigma_{min}(M_2)\leq \delta^2 n^{-1/2}$, then the claimed estimate follows from \eqref{assumption2}. Or, if for some $i$ we have $\sigma_{min}(M_i)
\leq\delta^{2}n^{-1/2}$, the claim also follows from \eqref{assumption2}. Thus, we assume that $\delta_1,\delta_2\in[\delta^2,1]$. We can use the dyadic decomposition \eqref{dyadic} and take $\delta_1,\delta_2$ to be one of the $I_{j,S}$'s or $I_{j,L}$'s, and there are at most $(2L_\delta)^2$ choices of pairs from the decomposition \eqref{dyadic} approximating $\delta_1,\delta_2$. 

Therefore we only need to consider $(2L_\delta)^2$ events of the form $$
\sigma_{min}(M_1)\leq x_1n^{-1/2}, \sigma_{min}(M_2)\leq x_2n^{-1/2},\quad 
x_1x_2\leq 4 \delta^2,
$$ where $x_1,x_2$ are chosen from the left or right end points of the intervals in the decomposition \eqref{dyadic}. The probability of each such event is at most $(4C\delta^2)^\tau+e^{-cn^\sigma}$ by assumption \ref{assumption}.
Then the lemma follows from taking a union bound over the $(2L_\delta)^2$ events.
\end{proof}

Now we can set up a bootstrap machinery as follows. 
\begin{lemma}\label{bootstraplemma} Let the random matrix $A_{n}$ and $\lambda_1,\lambda_2\in\mathbb{R}$ satisfy the same assumptions as in Lemma \ref{lemma7.27}.
Suppose that there exists some $\tau\in(0,1)$ such that for any $\delta_1,\delta_2>0$ we have the estimate
\begin{equation}\label{inductionhypothesis}
    \mathbb{P}\left(\sigma_{min}(A_{n+1}-\lambda_iI_{n+1})\leq\delta_i n^{-1/2},i=1,2\right)\lesssim(\delta_1\delta_2)^\tau+e^{-\Omega(n)}. 
\end{equation} Then for any $p>1$ and $\omega>0$, the following estimate holds for any $\delta_1,\delta_2>0$:
\begin{equation}\begin{aligned}&
    \mathbb{P}(\sigma_{min}(A_{n+1}-\lambda_iI_{n+1})\leq\delta_i n^{-1/2},i=1,2)\\&\lesssim
    (\delta_1\delta_2)^{\min(1,\frac{80}{81}(1-\omega)\tau p-\frac{8}{9}p+1)}(\log(\delta_1\delta_2)^{-1})^2+e^{-\Omega(n)}.
\end{aligned}\end{equation}
\end{lemma}

\begin{proof}
    Assume that $\delta_1,\delta_2\geq e^{-cn}$, otherwise the claim is trivial.
    We first apply Lemma \ref{singularvalueproduct}  to our induction hypothesis \eqref{inductionhypothesis} and obtain the following: for any $\epsilon>0$ and any $\omega>0$,
    $$\begin{aligned}&
\mathbb{P}(\prod_{i=1}^2\sigma_{min}(A_{n+1}-\lambda_iI_{n+1})\leq \epsilon^\tau n^{-1})\lesssim \epsilon(\log\epsilon^{-1})^{2}+e^{-\Omega(n)}\lesssim \epsilon^{\tau(1-\omega)}+e^{-\Omega(n)}.\end{aligned}
    $$
    
    Following the same lines as in the proof of the base step, Lemma \ref{lemma7.27}, we have the following worst-case estimates
$$\begin{aligned}
&\mathbb{E}_{A_n}\left[\left(\frac{\mu_1(\lambda_1)\mu_1(\lambda_2)}{n}\right)^{9/10}\mathbf{1}\left\{\frac{\mu_1(\lambda_1)\mu_1(\lambda_2)}{n}\leq\epsilon^{-p}\right\}\right]
\\&
\lesssim \int_0^{\epsilon^{-9p/10}}\mathbb{P}\left(\sigma_{min}(A_n-\lambda_1I_n)\cdot\sigma_{min}(A_n-\lambda_2I_n)\leq x^{-10/9}n^{-1}\right)dx,\\&
\lesssim 1+\int_1^{\epsilon^{-9p/10}}(x^{-(1-\omega)10\tau/9}+e^{-cn})dx\lesssim \max\{1,\epsilon^{(1-\omega)\tau p-9p/10}\}.
\end{aligned}$$
Then we apply Hanson-Wright inequality (Theorem \ref{hansonwright8.1}) for each $i=1,2$ and obtain
\begin{equation}\label{greathansonwright}
\mathbb{P}_X\left(\cup_{i=1,2}\{\|(A_n-\lambda_i I_n)^{-1}X\|_2\geq C'\|(A_n-\lambda_i I_n)^{-1}\|_{HS}\cdot\log(\delta_1\delta_2)^{-1}\}\right)\leq\delta_1\delta_2
\end{equation} for some fixed constant $C'>0$. We define $\bar{\delta}_1,\bar{\delta}_2$ in the same way as in \eqref{whatisbardelta}.

Then we apply Lemma \ref{lemma6.111} with $s=0$, $u=0$ and get: for any fixed $r_1,r_2\in\mathbb{R}$,
\begin{equation}\label{1114firstrows}\begin{aligned}&\mathbb{P}_{A_n,X}\left( \frac{|\langle (A_n-\lambda_i I_n)^{-1}X,X\rangle-r_i|}{\|(A_n-\lambda_i I_n)^{-1}X\|_2}\leq\delta_i,i=1,2;\quad\frac{\mu_1(\lambda_1)\mu_1(\lambda_2)}{n}\leq(\delta_1\delta_2)^{-p}
\right)\\&\leq \mathbb{P}_{A_n,X}\left( \frac{|\langle (A_n-\lambda_i I_n)^{-1}X,X\rangle-r_i|}{\|(A_n-\lambda_i I_n)^{-1}\|_{HS}}\leq\bar{\delta}_i,i=1,2;\quad\frac{\mu_1(\lambda_1)\mu_1(\lambda_2)}{n}\leq(\delta_1\delta_2)^{-p}
\right)+\delta_1\delta_2\\&\lesssim \max\{\delta_1\delta_2,(\delta_1\delta_2)^{\frac{80}{81}(1-\omega)\tau p-\frac{8}{9}p+1}\}\log((\delta_1\delta_2)^{-1})^2+e^{-\Omega(n)},
\end{aligned}\end{equation}
where the first inequality is via \eqref{greathansonwright} and the second inequality is via Lemma \ref{lemma6.111}.
We have also used the fact that $\|(A_n-\lambda_i I_n)^{-1}\|_{HS}\leq \|(A_n-\lambda_i I_n)^{-1}\|_*$. Note that, if we let $\mathcal{A}$ denote the event in the parenthesis in the first row of \eqref{1114firstrows}, then $\sup_{r_1,r_2\in\mathbb{R}}\mathbb{P}_{A_n,X}(\mathcal{A})= \sup_{r_1,r_2\in\mathbb{R}}\mathbb{E}_{A_n}\mathbb{P}_X(\mathcal{A})\leq\mathbb{E}_{A_n}\sup_{r_1,r_2\in\mathbb{R}}\mathbb{P}_X(\mathcal{A})$, and the latter is given by Lemma \ref{lemma6.111}.

Finally, we plug in estimate \eqref{greathansonwright} into Proposition \ref{finalfuckpropositionga} and conclude the proof.
\end{proof}

Now we apply Lemma \ref{bootstraplemma} finitely many times to improve the exponent of $(\delta_1\delta_2)^\tau$ all the way up to $\tau=1$. After this step, we have reached the optimal smallest singular value estimates except for the additional $\log(\delta_1\delta_2)^{-1}$ factor.

\begin{lemma}\label{lemma8.4wefinish}
    In the setting of Lemma \ref{lemma7.27}, for any $\delta_1,\delta_2>0$,
$$
\mathbb{P}\left(\sigma_{min}(A_{n+1}-\lambda_i I_{n+1})\leq \delta_i n^{-1/2}, \quad i=1,2 \right)\lesssim \delta_1\delta_2(\log(\delta_1\delta_2)^{-1})^2+e^{-\Omega(n)}.
$$

\end{lemma}

\begin{proof} By lemma \ref{lemma7.27}, we start with the initial value $\tau=0.1$.
We take $\omega$ arbitrarily small (say $10^{-10}$) and take $p=\frac{81}{80(1-\omega)}$ such that Lemma \ref{bootstraplemma} yields self-improvement $\tau\to (\tau+0.09)\wedge 1$ in the estimate (we upper bound $\epsilon^{\tau+1-0.9/(1-\omega)}(\log\epsilon^{-1})^{1/2}$ by $\epsilon^{\tau+0.09}$). Then after ten applications of Lemma \ref{bootstraplemma} we increase the value of $\tau$ from 0.1 to $0.19,0.28,\cdots,0.91,1$.
\end{proof}

\subsection{Completing the proof: eliminating the log factor}\label{removalladfaga}

Our final goal is to remove the additional (sub-optimal) log factor in the estimate of Lemma \ref{lemma8.4wefinish}.

The reason why the $\log(\delta_1\delta_2)^{-1}$ factor appears in Lemma \ref{lemma8.4wefinish} is because we have applied Hanson-Wright inequality to control the rare event $\|(A_n-\lambda_i I_n)^{-1}X\|_2\gg \|(A_n-\lambda_i I_n)^{-1}\|_{HS}$. However, it is not hard to observe that on this rare event, we can find some eigenvector $v$ of $A_n$ such that $\langle X,v\rangle$ is very large. We will now make a better use of this information.

By Lemma \ref{finalfuckpropositionga}, to conclude the proof of Theorem \ref{Theorem1.1} it suffices to prove that we can find some $C>0$ so that for any $r_1,r_2\in\mathbb{R}$, any $\delta_1,\delta_2>0$ and any $p>1$,
\begin{equation}\label{section11base}\begin{aligned}
    \mathbb{P}_{A_n,X}&\left(\frac{|\langle (A_n-\lambda_i I_n)^{-1}X,X\rangle-r_i|}{\|(A_n-\lambda_i I_n)^{-1}X\|_2}\leq C\delta_i,i=1,2,\prod_{i=1}^2\sigma_{min}(A_{n}-\lambda_i I_n)\geq (\delta_1\delta_2)^pn^{-1}\right)\\&\lesssim\delta_1\delta_2+e^{-\Omega(n)}.\end{aligned}
\end{equation}
We introduce the following notations (we keep implicit the dependence on $r_1,r_2$)
$$
Q(A_n,X,\lambda_i):=\frac{|\langle (A_n-\lambda_i I_n)^{-1}X,X\rangle-r_i|}{\|(A_n-\lambda_i I_n)^{-1}X\|_2},$$
$$ Q_*(A_n,X,\lambda_i):=\frac{|\langle (A_n-\lambda_i I_n)^{-1}X,X\rangle-r_i|}{\|(A_n-\lambda_i I_n)^{-1}\|_*},
$$
where we recall the definition 
$$
\|(A_n-\lambda_iI_n)^{-1}\|_*^2:=\sum_{k=1}^n \mu_k^2(\lambda_i)(\log(1+k))^2.
$$

Let $\mathcal{E}:=\left\{\prod_{i=1}^2 \sigma_{min}(A_n-\lambda_i I_n)\geq(\delta_1\delta_2)^p n^{-1}\right\}.$
We now decompose the left hand side of \eqref{section11base} using our new notations:
\begin{equation}\label{whatistheleftside?}\begin{aligned}
    &\mathbb{P}^{\mathcal{E}}\left(Q(A_n,X,\lambda_i)\leq C\delta_i,i=1,2\right)\leq\mathbb{P}^{\mathcal{E}}(Q_*(A_n,X,\lambda_i)\leq C\delta_i,\forall i=1,2)\\&+\mathbb{P}^{\mathcal{E}}\left(Q(A_n,X,\lambda_i)\leq C\delta_i,\forall i=1,2;\quad \frac{\|(A_n-\lambda_i I_n)^{-1} X\|_2}{\|(A_n-\lambda_i I_n)^{-1}\|_*}\geq 2,\forall i=1,2.\right)\\&+\mathbb{P}^{\mathcal{E}}\left(Q(A_n,X,\lambda_i)\leq C\delta_i,\forall  i=1,2;\quad \frac{\|(A_n-\lambda_i I_n)^{-1} X\|_2}{\|(A_n-\lambda_i I_n)^{-1}\|_*}\geq 2,
    \text{for one }i\in\{1,2\}\right).\end{aligned}
\end{equation}

Applying Lemma \ref{lemma8.4wefinish} together with Lemma \ref{lemma6.111} (taking $u=0$, $s=0$) we can easily prove the following lemma, which bounds the RHS of the first line of \eqref{whatistheleftside?}.
\begin{lemma}\label{lemmaa7.5fff} Let $A_n$ and $\lambda_1,\lambda_2$ satisfy the same assumptions as in Lemma \ref{lemma7.27}. Then there exists some $p_0>1$ such that for any $p\in(1,p_0)$ we have
    $$
    \mathbb{P}^{\mathcal{E}}(Q_*(A_n,X,\lambda_i)\leq C\delta_i,i=1,2)\lesssim\delta_1\delta_2+e^{-\Omega(n)}$$ where the implicit constant in $\lesssim$ also depends on $p$.
\end{lemma}
The proof of this lemma is essentially the same as the proof of Lemma \ref{lemma8.4wefinish} but without the step of using Hanson-Wright inequality. Thus there is no $\log(\delta_1\delta_2)^{-1}$ factor here.

In the following we evaluate the probability in the second line of \eqref{whatistheleftside?}, i.e.
\begin{equation}\label{secondline}
\mathbb{P}^{\mathcal{E}}\left(Q(A_n,X,\lambda_i)\leq C\delta_i,\forall i=1,2;\quad \frac{\|(A_n-\lambda_i I_n)^{-1} X\|_2}{\|(A_n-\lambda_i I_n)^{-1}\|_*}\geq 2,\forall i=1,2\right).\end{equation}

We consider a dyadic partition of $[2,\infty)$ so we can find two integers $j_1,j_2\geq 1$ with $$2^{j_i}\leq\|(A_n-\lambda_i I_n)^{-1}X\|_2/\|(A_n-\lambda_i I_n)^{-1}\|_* \leq 2^{j_i+1},\quad i=1,2.$$ We can terminate the summation of $j_1,j_2$ up to $\log n$ and upper bound \eqref{secondline} by 
\begin{equation}\label{bybounds}
     \sum_{j_1=1}^{\log n}\sum_{j_2=1}^{\log n}\mathbb{P}^{\mathcal{E}}\left(Q_*(A_n,X,\lambda_i)\leq 2^{j_i+1}C\delta_i,\quad \frac{\|(A_n-\lambda_i I_n)^{-1}X\|_2}{\|(A_n-\lambda_i I_n)^{-1}\|_*}\geq 2^{j_i},\quad i=1,2\right)+e^{-\Omega(n)}.
\end{equation}
Note that we can terminate the summation over $j_1,j_2$ at $\log n$ because, by Hanson-Wright (Theorem \ref{hansonwright8.1}) and $\|M\|_*\geq \|M\|_{HS}$ for any symmetric matrix $M$, we have
for both $i=1,2$,
$$
\mathbb{P}_X(\|(A_n-\lambda_i I_n)^{-1}X\|_2\geq n\|(A_n-\lambda_i I_n)^{-1}\|_*)\lesssim e^{-\Omega(n)}.
$$
Now we estimate each individual probability in the summation, via the following lemma:
\begin{lemma}\label{lemma7.6alreadytire} For any $t_i,\delta_i>0,i=1,2$ we have the following bound
$$\begin{aligned}
    &\mathbb{P}_X\left(Q_*(A_n,X,\lambda_i)\leq 2Ct_i\delta_i,\quad \frac{\|(A_n-\lambda_i I_n)^{-1}X\|_2}{\|(A_n-\lambda_i I_n)^{-1}\|_*}\geq t_i,\quad \forall i=1,2 
    \right)\\& \leq 4\sum_{k_1,k_2=1}^n\mathbb{P}_X(Q_*(A_n,X,\lambda_i)\leq 2Ct_i\delta_i,\langle X,v_{k_i}(\lambda_i)\rangle\geq t_i\log(1+k_i),\forall i=1,2),
\end{aligned}$$ where $v_k(\lambda_i)$ is the eigenvector associated with the $k$-th largest singular value  of $(A_n-\lambda_i I_n)^{-1}$ ((i.e., $\mu_k(\lambda_i)$), for each $i=1,2$ and $k\in[n]$. The set of eigenvectors $v_k(\lambda_i),k\in[n]$ form an orthonormal basis of $\mathbb{R}^n$.
\end{lemma}
\begin{proof}
We apply spectral decomposition to $\|(A_n-\lambda_i I_n)^{-1}X\|_2\geq t_i\|(A_n-\lambda_i I_n)^{-1}\|_*$ and obtain
$$
t_i^2\sum_k \mu_k^2(\lambda_i)(\log(k+1))^2\leq \sum_k \mu_k^2(\lambda_i)\langle v_k(\lambda_i),X\rangle^2.
$$
That is, 
$$
\{\|(A_n-\lambda_i I_n)^{-1}X\|_2\geq t_i\|(A_n-\lambda_i I_n)^{-1}\|_*\}\subset\bigcup_k \{|\langle X,v_k(\lambda_i)\rangle|\geq t_i\log(k+1)\}.
$$ We apply this decomposition to both $i=1$ and $i=2$. That is, there exists $k_1,k_2\in[1,n]$ such that $|\langle X,v_{k_i}(\lambda_i)\rangle|\geq t_i\log(k_i+1)$. 
\end{proof}

Finally, to estimate the third line of \eqref{whatistheleftside?}, it suffices to consider
$$\mathbb{P}^{\mathcal{E}}\left(Q_*(A_n,X,\lambda_1)\leq C\delta_1, Q(A_n,X,\lambda_2)\leq C\delta_2;\quad \frac{\|(A_n-\lambda_2 I_n)^{-1} X\|_2}{\|(A_n-\lambda_2 I_n)^{-1}\|_*}\geq 2\right),$$
(the case when we swap subscripts 1 and 2 is exactly the same). This probability in question is bounded by $e^{-cn}$ plus the following: 
$$
\sum_{j=1}^{\log n}\mathbb{P}^{\mathcal{E}}\left( Q_*(A_n,X,\lambda_1)\leq 2C\delta_1,  Q_*(A_n,X,\lambda_2)\leq 2^{j+1}C\delta_2,
\frac{\|(A_n-\lambda_2 I_n)^{-1}X\|_2}{\|(A_n-\lambda_2 I_n)^{-1}\|_*}
\geq 2^j\right). 
$$ Following the proof of Lemma \ref{lemma7.6alreadytire}, we can check that this probability is bounded by $e^{-cn}$ plus the following summation:
\begin{equation}\label{deductionsigagg}
\sum_{j=1}^{\log n}\sum_{k=1}^n\mathbb{P}^{\mathcal{E}}\left( Q_*(A_n,X,\lambda_1)\leq 2C\delta_1,  Q_*(A_n,X,\lambda_2)\leq 2^{j+1}C\delta_2,
\langle X,v_k(\lambda_i)\rangle\geq 2^j\log(1+k)\right)\end{equation}

Now we can finish the proof of Theorem \ref{Theorem1.1}.

\begin{proof}[\proofname\ of Theorem \ref{Theorem1.1}]
We combine decomposition \eqref{whatistheleftside?}, Lemma \ref{lemmaa7.5fff} with the expansion \eqref{bybounds}, Lemma \ref{lemma7.6alreadytire} and the expansion \eqref{deductionsigagg} to obtain
$$\begin{aligned}
&\mathbb{P}^{\mathcal{E}}(Q(A_n,X,\lambda_i)\leq C\delta_i,i=1,2)\lesssim e^{-\Omega(n)}+\delta_1\delta_2\\&+\sum_{j_1,j_2=1}^{\log n}\sum_{k_1,k_2=1}^{n}\mathbb{P}^{\mathcal{E}}\left(Q_*(A_n,X,\lambda_i)\leq 2^{j_i+1}C\delta_i;\langle X,v_{k_i}(\lambda_i)\rangle\geq  2^{j_i}\log((1+k_i))
\right)\\&+\sum_{j_1=1,j_2\equiv 0}^{j_1=\log n}\sum_{k_1=1,k_2\equiv 0}^{k_1=n}\mathbb{P}^{\mathcal{E}}\left(Q_*(A_n,X,\lambda_i)\leq 2^{j_i+1}C\delta_i;\langle X,v_{k_i}(\lambda_i)\rangle\geq  2^{j_i}\log((1+k_i))
\right)\\&+\sum_{j_1\equiv 0,j_2=1}^{j_2=\log n}\sum_{k_1\equiv0,k_2=1}^{k_2=n}\mathbb{P}^{\mathcal{E}}\left(Q_*(A_n,X,\lambda_i)\leq 2^{j_i+1}C\delta_i;\langle X,v_{k_i}(\lambda_i)\rangle\geq  2^{j_i}\log((1+k_i))
\right).
\end{aligned}$$
with the convention $v_0(\lambda_i):=0$. In the third and fourth lines we use summation indices that are identically zero, just to make a formal similarity to the sum on the second line.

We can combine the two conditions $\langle X,v_{k_i}(\lambda_i)\rangle\geq 2^{j_i}\log(1+k_i)$, $i=1,2$ into $\langle X,v_{k_1k_2}\rangle\geq 2^{j_1}\log(1+k_1)+2^{j_2}\log(1+k_2)$,
where we set $u_{k_1k_2}=v_{k_1}(\lambda_1)+v_{k_2}(\lambda_2)$. Note that $v_{k_1}(\lambda_1)$ and $v_{k_2}(\lambda_2)$ are both eigenvectors of $A_n$, so either we have $\|u_{k_1k_2}\|=\sqrt{2}$ if they correspond to two orthogonal eigenvectors, or if they correspond to co-linear vectors we must have that $v_{k_1}(\lambda_1)=v_{k_2}(\lambda_2)$ (by assumption of Lemma \ref{lemma7.6alreadytire} they are either co-linear or orthogonal) so that $\|u_{k_1k_2}\|=2$. We can now apply Lemma \ref{lemma6.111} for all $t_i>0$ where we choose $\delta_i$ to be $2Ct_i\delta_i$, $s=t_1\log(k_1+1)+t_2\log(k_2+1)$ and $u=u_{k_1k_2}$. Then Lemma \ref{lemma6.111} implies
\begin{equation}\begin{aligned}
    &\mathbb{P}^{\mathcal{E}_0}(Q_*(A,X,\lambda_i)\leq 2Ct_i\delta_i,\langle X,u_{k_1k_2}\rangle\geq t_1\log(1+k_1)+t_2\log(1+k_2))\\&\lesssim \delta_1\delta_2 t_1t_2(k_1+1)^{-t_1/2}(k_2+1)^{-t_2/2}\cdot I^{81/80}+e^{-\Omega(n)},
\end{aligned}\end{equation}
where $$I:=\mathbb{E}_{A_n}\left[\left(\frac{\mu_1(\lambda_1)\mu_1(\lambda_2)}{n}\right)^{9/10}\mathbf{1}\left\{\frac{\mu_1(\lambda_1)\mu_1(\lambda_2)}{n}\leq(\delta_1\delta_2)^{-p}\right\} \right]$$
and $p>1$ can be chosen to be arbitrarily close to 1. 
Applying Lemma \ref{lemma8.4wefinish} combined with Lemma \ref{singularvalueproduct} and a suitable choice of parameters (i.e. setting $p$ sufficiently close to 1 and when applying Lemma \ref{singularvalueproduct}, setting $\tau$ sufficiently close to 1), we can show that $$I\lesssim 1.$$
Combining all the above, we obtain 
$$\begin{aligned}
\mathbb{P}^{\mathcal{E}_0}(Q(A,X,\lambda_i)\leq C\delta_i,i=1,2)&\lesssim \delta_1\delta_2 \sum_{j_1,j_2=1}^{\log n}\sum_{k_1,k_2=1}^n 2^{j_1+j_2}(k_1+1)^{-2^{j_1}}(k_2+1)^{-2^{j_2}}
\\&+2\delta_1\delta_2\sum_{j=1}^{\log n}\sum_{k=1}^n 2^j(k+1)^{-2^j}
+e^{-\Omega(n)}.
\end{aligned}$$

Now the proof is completed by the elementary computation
$$\sum_{j_1=1}^\infty\sum_{j_2=1}^\infty\sum_{k_1=1}^\infty \sum_{k_2=1}^\infty 2^{j_1+j_2}(k_1+1)^{-2^{j_1}}(k_2+1)^{-2^{j_2}}=O(1),$$ 
and $\sum_{j=1}^\infty\sum_{k=1}^\infty 2^j(k+1)^{-2^j}=O(1)$.
\end{proof}

\section{Locations with a mesoscopic separation without assuming LSI}
\label{section9section9}
In this section, we prove Theorem \ref{Theorem1.2}, which states that all the results in Theorem \ref{Theorem1.1} hold but with an error term $e^{-\Omega(n^{\sigma/2})}$ in place of $e^{-\Omega(n)}$, assuming that the locations $\lambda_1,\lambda_2$ are separated only by distance $\Delta n^{-1/2+\sigma}$, for some fixed values of $\Delta>0$ and $\sigma>0$. Also, we no longer need to assume that $\zeta$ has a finite Log-Sobolev constant.
The proof essentially follows the proof of Theorem \ref{Theorem1.1}, but we no longer use the log-Sobolev constant.

We use again the notation $\mu_{c_j(i;k)}(\lambda_j)$ in Definition \ref{definition4.2}, where $c_j(i;k)$ is the subscript of the singular value of $(A_n-\lambda_j I_n)^{-1}$ whose associated eigenvector coincides with the eigenvector associated to the $k$-th largest singular value of $(A_n-\lambda_i I_n)^{-1}$.

We prove the following further estimate as a corollary of Lemma \ref{corollary3.6chap9}: 
\begin{lemma}\label{threeeventsnsigma} Let $A_n$ be as in Lemma \ref{corollary3.6chap9}. Let $\Delta>0,\kappa>0$ and $\sigma\in(0,1)$ be fixed constants. Fix any $\lambda_1,\lambda_2\in[-(2-\kappa)\sqrt{n},(2-\kappa)\sqrt{n}]$ satisfying $|\lambda_1-\lambda_2|\geq\Delta n^{\sigma-\frac{1}{2}}$. Then \begin{enumerate}
\item We can find some constant $\Delta_0>0$ and some $C_0>0$ depending on $\Delta$ such that, for any $1\leq k\leq \Delta_0n^\sigma$ and $j\neq i$, we have
    \begin{equation}
        \mathbb{P}(c_j(i;k)\leq C_0n^\sigma)\leq \exp(-\Omega(n^{\sigma/2})).
    \end{equation} \item For any $i\neq j\in\{1,2\}$, on an event with probability $1-e^{-\Omega(n^{\sigma/2})},$ we have
\begin{equation}\label{proofnameagaga}
    \frac{\mu_{c_i(j;1)}(\lambda_i)}{\mu_1(\lambda_i)}\leq 10^{-1}.
\end{equation}
\item Fix a given $(\theta_1,
\theta_2)\in\mathbb{R}^2$ and some $j\neq i\in\{1,2\}$ that satisfy $|\frac{\theta_j}{\mu_1(\lambda_j)}|\geq |\frac{\theta_i}{\mu_1(\lambda_i)}|$. We can find some $\Delta_1>0$ depending on $\Delta$ such that on an event with probability $1-e^{-\Omega(n^{\sigma/2})},$ for any $1\leq k\leq \Delta_1n^{\sigma}$, we have 
\begin{equation}
   |\theta_i| \frac{\mu_{c_i(j;k)}(\lambda_i)}{\mu_1(\lambda_i)}\leq 10^{-1}|\theta_j|\frac{\mu_k(\lambda_j)}{\mu_1(\lambda_j)}
\end{equation}
\end{enumerate} The constant term in the error $\exp(-\Omega(n^{\sigma/2}))$ depends on all the constants $\Delta,\kappa,\sigma>0$.
\end{lemma}
\begin{proof}
    For the first claim, suppose that $c_j(i;k)\leq C_0n^\sigma$. Then there are at most $k+C_0n^\sigma\leq(\Delta_0+C_0)n^\sigma$ number of eigenvalues of $A_n$ in the interval $[\lambda_i,\lambda_j]$ (or $[\lambda_j,\lambda_i]$). However, if $\Delta_0+C_0$ is chosen to be sufficiently small, this event has probability $e^{-\Omega(n^{\sigma/2})}$ according to Theorem \ref{theorem4.3}.

For the second and third claim, it suffices to prove that we can find $\Delta_7>0$ so that for all $1\leq k\leq \Delta_7 n^\sigma$ and any $i,j\in\{1.2\}$, we have 
$$
\mu_k(\lambda_j)\geq 10\mu_{c_i(j;k)}(\lambda_i),
$$ on an event of probability $e^{-\Omega(n^{\sigma/2})}$. Indeed, by the conclusion of Lemma \ref{corollary3.6chap9}, on an event with such probability we can select $\Delta_7>0$ so that $\mu_k(\lambda_j)\geq \mu_{\Delta_7 n^\sigma}(\lambda_j)\geq 10\mu_{C_0n^\sigma}(\lambda_i)\geq 10\mu_{c_i(j;k)}(\lambda_i)$, completing the proof.\end{proof}

We can prove the following analogue of Lemma \ref{lemma6.61}:

\begin{lemma}\label{9--lemma6.61}
 Assume that $A_n$ satisfies the assumptions in Theorem \ref{Theorem1.2} and $X\sim\operatorname{Col}_n(\zeta)$. Fix any $\kappa>0$, $\Delta>0$ and some $\sigma\in(0,1)$. Then there exists an event $\mathcal{E}_\sigma(A_n)$ with $\mathbb{P}(A_n\in\mathcal{E}_\sigma(A_n))\geq 1-\exp(-\Omega(n^{\sigma/2}))$ such that for any $A_n\in\mathcal{E}_\sigma(A_n)$, the following holds:

   \begin{enumerate}
       \item Fix any given $(\theta_1,\theta_2)\in\mathbb{R}^2$ satisfying $|\theta_1\theta_2|=1$.   Then there exists some $J\in\{1,2\}$ depending on $A_n$ and $|\theta_1/\theta_2|$ such that, for any $1\leq k\leq {c_0}_{\ref{9--lemma6.61}}n^\sigma$ and for any \begin{equation}s\in\left(e^{-cn}, \frac{C_{\ref{9--lemma6.61}}\cdot \mu_k(\lambda_J)}{{\left(\prod_{i=1}^2\mu_1(\lambda_i)\right)^{1/2}}}\right),\end{equation} we have the estimate
    \begin{equation}\label{2122122121122}
\mathbb{P}_{\widetilde{X}}\left(\left\|\sum_{i=1}^2 \frac{\theta_i}{\mu_1(\lambda_i)}(A_n-\lambda_i I_n)^{-1}\widetilde{X}\right\|_2\leq s
\right) \lesssim s e^{-ck},
    \end{equation} where ${c_0}_{\ref{9--lemma6.61}}>0$ and $C_{\ref{9--lemma6.61}}>0$ are two constants that depend only on $\kappa,\Delta$ and $B$. The constant $c>0$ depends only on $B$.
    \item  Fix any given $(\theta_1,\theta_2)\in\mathbb{R}^2$ satisfying $\max(|\theta_1|,|\theta_2|)=1$. Then there exist $I,J\in\{1,2\}$ depending on $A_n$ and $|\theta_1/\theta_2|$ such that, for any $1\leq k\leq {c_0}_{\ref{9--lemma6.61}}n^\sigma$ and any 
    \begin{equation}
s\in(e^{-cn},\frac{\mu_k(\lambda_J)}{\mu_1(\lambda_I)})
    \end{equation} we have the same estimate \eqref{2122122121122}.
     \end{enumerate}
\end{lemma}

\begin{proof}(Sketch) The only necessary change to the proof of Lemma \ref{lemma6.61} is that we intersect the quasi-randomness event $\mathcal{E}$ (Lemma \ref{mainquasirandomness1}) with the three events in Lemma \ref{threeeventsnsigma} to define the event $\mathcal{E}_\sigma(A_n)$. This in turn restricts the possible range of $k$ to $1\leq k\leq {c_0}_{\ref{9--lemma6.61}}n^\sigma$.
    
\end{proof}

Under the assumption of Theorem \ref{Theorem1.2}, Proposition \ref{finalfuckpropositionga} holds without change. Using Lemma \ref{9--lemma6.61}, we can prove the following analogue of Lemma \ref{lemma6.111}:

\begin{lemma}\label{lemma10.444}  Under the assumption of Theorem \ref{Theorem1.2}, for any $\delta_1,\delta_2\geq e^{-cn^{\sigma}}$, any $u\in\mathbb{S}^{n-1}$, and any $p>1$, we have the following estimate:
\begin{equation}
\begin{aligned}
&\mathbb{E}_{A_n}\sup_{r_1,r_2\in\mathbb{R}}\mathbb{P}_X\left(
    \frac{|\langle (A_n-\lambda_i I_n)^{-1}X,X\rangle-r_i|}{\|(A_n-\lambda_i I_n)^{-1}\|_*}\leq\delta_i
    ,\langle X,u\rangle\geq s,\frac{\prod_{i=1}^2\mu_1(\lambda_i)}{n}\leq(\delta_1\delta_2)^{-p}
    \right)\\&\lesssim e^{-s}\delta_1\delta_2+e^{-\Omega(n^{\sigma/2})}
    \\&+e^{-s}\delta_1\delta_2\mathbb{E}_{A_n}\left[\left(\frac{\prod_{i=1}^2\mu_1(\lambda_i)}{n}\right)^{\frac{9}{10}}
    \wedge\left\{ \frac{\prod_{i=1}^2\mu_1(\lambda_i)}{n}\leq (\delta_1\delta_2)^{-p}\right\}\right]^{\frac{80}{81}},\end{aligned}
\end{equation}
    where the constant $c>0$ in the assumption $\delta_i\geq e^{-cn^\sigma}$ and the implied constant in the inequality $\lesssim$ depend on all the constants $B,\kappa,\Delta,\sigma$.
\end{lemma}
\begin{proof}(Sketch)
This essentially follows the proof of Lemma \ref{lemma6.111}. The main place for a change is that we take the following decomposition, for a suitable $c>0$ and $\delta_1,\delta_2\geq e^{-cn^\sigma}$:
\begin{equation}\label{decompositiongaag}[\prod_{i=1}^2\delta_i,1]\subset\left[e^{-2cn^\sigma},\frac{\mu_{{c_0}_{\ref{9--lemma6.61}} n^\sigma}(\lambda_j)}{(\prod_{i=1}^2\mu_1(\lambda_i))^{1/2}}\right]\cup \bigcup_{k=2}^{{c_0}_{\ref{9--lemma6.61}} n^\sigma}\left[\frac{\mu_{k}(\lambda_j)}{(\prod_{i=1}^2\mu_1(\lambda_i))^{1/2}},\frac{\mu_{k-1}(\lambda_j)}{(\prod_{i=1}^2\mu_1(\lambda_i))^{1/2}}\right]\end{equation} Then we apply Lemma \ref{9--lemma6.61}. Observe that the largest possible value of $k$ here is only $c_0n^\sigma$, so that in order for the last part in the integral of $I_3$ in \eqref{definitionI3} to converge, we must require that $\delta_1,\delta_2\geq e^{-cn^{\sigma}}$ for an appropriate $c>0$ and we replace the integration range from $[e^{-cn},*]$ to $[e^{-2cn^\sigma},*]$ in the last term of $I_3$ \eqref{definitionI3} because its leading multiplicative term is only $e^{cn^\sigma}$ instead of $e^{cn}$.
Finally, use $\mathbb{P}(A_n\notin\mathcal{E}_\sigma(A_n))\leq e^{-\Omega(n^{\sigma/2})}$.
\end{proof}

\begin{remark}\label{remarkweaks}
    The proof of Lemma \ref{lemma10.444} suggests that when $\sigma\in(0,1)$ we can not hope for the optimal error term $e^{-cn}$ in Theorem \ref{Theorem1.2} as compared to Theorem \ref{Theorem1.1}: in the former case there are typically only $O(n^\sigma)$ eigenvalues in the interval $[\lambda_1,\lambda_2]$, so an exponentially decaying factor would only yield $e^{-cn^\sigma}$ rather than $e^{-cn}$. This accounts for the reason that we impose the restriction  $\delta_1,\delta_2\geq e^{-cn^\sigma}$ in Lemma \ref{lemma10.444}.

    Meanwhile, Lemma \ref{lemma10.444} has an error term $e^{-\Omega(n^{\sigma/2})}$ as compared to our restriction $\delta_1,\delta_2\geq e^{-cn^\sigma}$: this error term is likely to be suboptimal and may be improved to $e^{-\Omega(n^{\sigma})}$.
\end{remark}

Now all preparations have been made for the proof of Theorem \ref{Theorem1.2}. 

\begin{proof}[\proofname\ of Theorem \ref{Theorem1.2}] The proof here is identical to the proof of Theorem \ref{Theorem1.1}. We can derive the corresponding versions of Lemma \ref{lemma7.27} and \ref{bootstraplemma} where the only difference is that we replace all $e^{-\Omega(n)}$ terms by $e^{-\Omega(n^{\sigma/2})}$. Then we remove the log factor in exactly the same way as the proof of Theorem \ref{Theorem1.1}. The details are omitted.
\end{proof}

\section{Leftover technical proofs}\label{whatchap23g>?}

This section contains the proof to a few technical results of this paper. The first is the proof of Theorem \ref{theorem12.234567}, our main inverse Littlewood-Offord theorem for vector pairs with conditioning on the rank. The proof is an adaptation of the corresponding theorem for one vector, given in \cite{campos2025singularity} and \cite{campos2024least}, and proceeds all the way up to Section \ref{upgradingtos}. Finally, Section \ref{lastbooks} contains the proof to some other technical results in the main text.

We first introduce some notations. For an $2d\times k$ matrix $W$ and an $\ell$-tuple of vectors $Y_1,\cdots,Y_\ell\in\mathbb{R}^d$, we define an augmented matrix $W_\mathbf{Y}$ as follows, where we abbreviate $\mathbf{Y}:=(Y_1,\cdots,Y_\ell)$: 
$$
W_\mathbf{Y}=\begin{bmatrix}
    W, \begin{bmatrix}
        Y_1\\\mathbf{0}_d
    \end{bmatrix},\begin{bmatrix}
        \mathbf{0}_d\\Y_1
    \end{bmatrix}
   ,\cdots,\begin{bmatrix}
        Y_\ell\\\mathbf{0}_d
    \end{bmatrix},\begin{bmatrix}
        \mathbf{0}_d\\Y_\ell
    \end{bmatrix}
\end{bmatrix}.
$$

\subsection{Initial geometric reductions}

We define $\gamma_\ell$ as the $\ell$-dimensional Gaussian measure defined by $\gamma_\ell(S)=\mathbb{P}(g\in S)$ for every Borel set $S\subset\mathbb{R}^\ell$ and $g\sim\mathcal{N}(0,(2\pi)^{-1}I_\ell)$ with $I_\ell$ the $\ell\times\ell$ identity matrix. We also denote by $\operatorname{Vol}_\ell(\cdot)$ the Lebesgue measure on $\mathbb{R}^\ell$. 

 The following fact is well-known:

\begin{fact}\label{factsphere}
    The Lebesgue volume of an $n$-dimensional ball of radius $R>0$ is
    $$
\operatorname{Vol}_n(R)=\frac{\pi^{n/2}R^n}{\Gamma(n/2+1)},
    $$ which can be upper bounded by the following for any $n\in\mathbb{N}$: $$
\operatorname{Vol}_n(R)\leq (\frac{3}{\sqrt{n}})^nR^n.
    $$ More generally, for fixed $s_1,\cdots,s_n\in\mathbb{R}^n_+$ consider the following ellipse
    \begin{equation}\label{ellipsevolume}
\mathcal{E}_{s_1,\cdots,s_n}:=\{x\in\mathbb{R}^n:\sum_{i=1}^n\frac{x_i^2}{s_i^2}\leq 1\}.
    \end{equation} Then its Lebesgue volume $\operatorname{Vol}_n(\mathcal{E}_{s_1,\cdots,s_n})$ satisfies  $$\operatorname{Vol}_n(\mathcal{E}_{s_1,\cdots,s_n})\leq (\frac{3}{\sqrt{n}})^ns_1\cdots s_n.$$
\end{fact}
\begin{proof}
    The first two facts are well-known. The last fact follows from integrating over the volume $\mathcal{E}_{s_1,\cdots,s_n}$ and applying a change of variable $x_i\mapsto s_ix_i$.
\end{proof}

For an $2\ell$-tuple of real numbers $\mathbf{s}=(s_1,\cdots,s_{2\ell})\in\mathbb{R}_+^{2\ell}$ and $x\in\mathbb{R}^{2\ell}$, define a norm 
$$
\|x\|_\mathbf{s}^2:=\sum_{i=1}^{2\ell}\frac{x_i^2}{s_i^2}.
$$

\begin{fact}\label{fixanfact4.3}
    Fix an $2\ell$-tuple of real numbers $\mathbf{s}=(s_1,\cdots,s_{2\ell})\in\mathbb{R}_+^{2\ell}$ with $\|s\|_\infty\leq 1$. Fix also $t\in(0,1)$. Let  $S\subset\mathbb{R}^{2\ell}$ be a measurable subset satisfying $$\gamma_{2\ell}(S)\geq 2^{2\ell+1}(\prod_{i=1}^{2\ell}s_i)t^{2\ell}(\frac{3}{\sqrt{2\ell}})^{2\ell}.$$ Then there exists $x,y\in S$ such that $\|x-y\|_\infty\leq 16$ but $\|x-y\|_\mathbf{s}\geq t$.
\end{fact}
\begin{proof}
   We may cover $\mathbb{R}^{2\ell}=\cup_{p\in 16\cdot\mathbb{Z}^{2\ell}}Q_p$ and each block $Q_p:=p+[-8,8]^{2\ell}$. Then $$\gamma_{2\ell}(S)\leq\sum_{p\in 16\cdot\mathbb{Z}^{2\ell}}\gamma_{2\ell}(S\cap Q_p).$$ By our assumption on $S$, for each $Q_p$ there exists a $x=x(p)\in Q_p$ with 
    $$
\gamma_{2\ell}(S\cap Q_p)\leq \gamma_{2\ell}(S\cap Q_p\cap(x(p)+\mathcal{E}_\mathbf{s}(t)))\leq\gamma_{2\ell}(x(p)+\mathcal{E}_\mathbf{s}(t)) 
    $$ where $\mathcal{E}_\mathbf{s}(t):=\{v\in\mathbb{R}^{2\ell}:\|v\|_\mathbf{s}\leq t\}$ is the $t$-translate of the ellipse defined in \eqref{ellipsevolume}.
   
    Taking $g\sim\mathcal{N}(0,(2\pi)^{-1})$, then 
    $$
\gamma_{2\ell}(x+\mathcal{E}_\mathbf{s}(t))\leq \operatorname{Vol}_{2\ell}(\mathcal{E}_\mathbf{s}(t))\exp(-\pi\|p\|_2^2/16)
    $$ where we use the fact that the Gaussian density on $x+\mathcal{E}_\mathbf{s}(t)$ is bounded from above by $\exp(-\pi\|p\|_2^2/16)$, which uses the inequality $(x_i-s_i)^2\geq p_i^2/8$ (this can be checked separately for $p_i=0$ and $p_i\neq 0$ and where we use $s_i\in(0,1)$). Then we can bound
    $$
\gamma_{2\ell}(S)\leq\sum_{p\in 16\cdot\mathbb{Z}^{2\ell}}\gamma_{2\ell}(S\cap Q_p)\leq \operatorname{Vol}_{2\ell}(\mathcal{E}
_\mathbf{s}(t))\sum_{p\in 16\cdot\mathbb{Z}^{2\ell}}\exp(-\pi\|p\|_2^2/16)<2^{2\ell+1}\operatorname{Vol}_{2\ell}(\mathcal{E}_\mathbf{s}(t)),
    $$ where in the last step we used the numerical estimate 
    $
\sum_{x\in\mathbb{Z}}\exp(-16\pi x^2)\leq\sqrt{2}.
    $ Finally we use  
    $\operatorname{Vol}_{2\ell}(\mathcal{E}_\mathbf{s}(t))\leq (\frac{3}{\sqrt{2\ell}})^{2\ell}(\prod_{i=1}^{2\ell}s_i)t^{2\ell}.$ This completes the proof.
\end{proof}

For $S\subseteq\mathbb{R}^{k+2\ell}$ and $\theta_{[k]}\in\mathbb{R}^k$, define the vertical fiber
\begin{equation}
    S(\theta_{[k]}):=\{(\theta_{k+1},\cdots,\theta_{k+2\ell})\in\mathbb{R}^{k+2\ell}:(\theta_{[k]},\theta_{k+1},\cdots,\theta_{k+2\ell})\in S\}.
\end{equation}

Now for any $r>0$ and sequence $\mathbf{s}\in\mathbb{R}_+^{2\ell}$, we define the following cylinder for any $t>0$
\begin{equation}\label{definition1st}
    \Gamma_{r,\mathbf{s},t}:=\left\{\theta\in\mathbb{R}^{k+2\ell}:\|\theta_{[k]}\|_2\leq r,\|\theta_{[k+1,k+2\ell]}\|_\mathbf{s}\leq t\right\}.
\end{equation}We also define, for any real number $t>0$,
\begin{equation}\label{definition2nd}
    \Gamma_{r,t}:=\left\{\theta\in\mathbb{R}^{k+2\ell}:\|\theta_{[k]}\|_2\leq r,\|\theta_{[k+1,k+2\ell]}\|_\infty\leq t\right\}. 
\end{equation} The two notations will hopefully not lead to confusion: while we use the same $\Gamma$, the first notation has three parameters and the second only has two. When using \eqref{definition2nd} we will only take $t=8$ or $t=16$, but when using \eqref{definition1st} we take many different values $t>0$. Then:
\begin{lemma}\label{lemma4.44}
    For given $k\in\mathbb{N}$ and $t>0$, let $S\subset\mathbb{R}^{k+2\ell}$ satisfy that for any $x\in S$,
    $$
(\Gamma_{r,16}\setminus\Gamma_{r,\mathbf{s},t}+x)\cap S=\emptyset.
    $$Then
    \begin{equation}\label{lastbounds}
\max_{\theta_{[k]}\in\mathbb{R}^k}\gamma_{2\ell}(S(\theta_{[k]}))\leq 2^{2\ell+1}(\prod_{i=1}^{2\ell}s_i)t^{2\ell}(\frac{3}{\sqrt{2\ell}})^{2\ell}.
    \end{equation}
\end{lemma}

\begin{proof}
    We proceed with the contrapositive. Assume that $\psi_{[k]}\in\mathbb{R}^k$ satisfies $\gamma_{2\ell}(S(\psi_{[k]}))$ is larger than the right hand side of \eqref{lastbounds}. Then by Fact \ref{fixanfact4.3} we can find $\mathbf{\theta}_{[k+1,k+2\ell]}:=(\theta_{k+1},\cdots,\theta_{k+2\ell})$ and $\mathbf{\theta}'_{[k+1,k+2\ell]}:=(\theta_{k+1}',\cdots,\theta_{k+2\ell}')\in S(\psi_{[k]})$ satisfying 
    $$
\|\theta_{[k+1,k+2\ell]}-\theta'_{[k+1,\cdots,k+2\ell]}\|_\mathbf{s}\geq t,\quad \|\theta_{[k+1,k+2\ell]}-\theta'_{[k+1,\cdots,k+2\ell]}\|_\infty\leq 16.
    $$ Then we can define $\theta,\theta'\in S$ with $\theta_{[k]}=\theta_{[k]}'=\psi_{[k]}$ and $\theta_{[k+1,k+2\ell]}$, $\theta'_{[k+1,k+2\ell]}$ are previously specified. We have that $\theta\in(\theta'+\Gamma_{r,16}\setminus\Gamma_{r,\mathbf{s},t})$.
\end{proof}

\subsection{Decoupling via Gaussian measures}
We now decouple the consideration on the first $k$ coordinates of $S$ and the last $2\ell$ coordinates of $S$.
First, we prove that the upper bound of $\gamma_{k+2\ell}(S)$ can be reduced to a local problem.

\begin{lemma}\label{lemma4.5}
    Let $k,\ell\geq 0,r>0$ and $S\subset\mathbb{R}^{k+2\ell}$ be a measurable subset. Then we can find some $x\in S$ and $h\in\Gamma_{r,8}$ such that
    $$
\gamma_{k+2\ell}(S\cap B)\leq 2^{2
\ell}\gamma_{k+2\ell}((S-x)\cap\Gamma_{2r,16}+h)
    $$ where we take $B:=\{\theta\in\mathbb{R}^{k+2\ell}:\|\theta_{[k]}\|_2\leq r\}$.
\end{lemma}

\begin{proof}  Taking the translates $\Gamma_{r,8}+y$ with $y\in 16\mathbb{Z}^{2\ell}$, we can write
\begin{equation}
    \gamma_{k+2\ell}(S\cap B)=\sum_{y\in\{0\}^k\times 16\mathbb{Z}^{2\ell}}\gamma_{k+2\ell}(S\cap(\Gamma_{r,8}+y)).
\end{equation}
    We expand $\gamma_{k+2\ell}(S\cap (\Gamma_{r,8}+y))$ into integrals as follows:
    $$\gamma_{k+2\ell}(S\cap(\Gamma_{r,8}+y))=
\int_{\mathbb{R}^{k+2\ell}}\mathbf{1}[\phi\in(S-y)\cap\Gamma_{r,8}]e^{-\pi\|\phi+y\|_2^2/2}d\phi.
    $$

Now we expand the expression of $\|\phi+y\|_2^2$ via 
$$
\|\phi+y\|_2^2=\|\phi\|_2^2+\|y\|_2^2+2\sum_{j=1}^{2\ell}\phi_{k+j}y_{k+j}\geq \|\phi\|_2^2+\|y\|_2^2-16\sum_{j=1}^{2\ell}|y_{k+j}|.
$$  where in the last inequality we use $|\phi_{k+i}|\leq 8$ for all $i\in[2\ell]$. Then we have
\begin{equation}
    \gamma_{k+2\ell}(S\cap(\Gamma_{r,8}+y))\leq \gamma_{k+2\ell}((S-y)\cap\Gamma_{r,8})\exp(-\frac{\pi}{2}\sum_{j=1}^{2\ell}y_{k+j}^2+8\pi\sum_{j=1}^{2\ell}|y_{k+j}|).
\end{equation} Then we combine the above to get 
$$\begin{aligned}
\gamma_{k+2\ell}(S\cap B)&\leq \max_y \gamma_{k+2\ell}((S-y)\cap\Gamma_{r,8})\sum_{y_{k+1},\cdots,y_{k+2\ell}\in16\mathbb{Z}}e^{-\frac{\pi}{2}\sum_{j=1}^{2\ell}y_{k+j}^2+8\pi\sum_{j=1}^{2\ell}|y_{k+j}|}\\&\leq 2^{2\ell}\max_y\gamma_{k+2\ell}((S-y)\cap\Gamma_{r,8}).
\end{aligned}$$

Now we take $y$ the vector where the maximum was attained. When $S\cap(\Gamma_{r,8}+y)\neq\emptyset$ then we take $x\in S\cap(\Gamma_{r,8}+y)$ and define $h:=x-y\in\Gamma_{r,8}$. (If $S\cap(\Gamma_{r,8}+y)=\emptyset$ then there is nothing to prove). Then we must have
$$
(S-y)\cap\Gamma_{r,8}\subseteq (S-x)\cap\Gamma_{2r,16}+h.
$$ Therefore we conclude with the desired result
$$
\gamma_{k+2\ell}(S\cap B)\leq 2^{2\ell} \gamma_{k+2\ell}((S-x)\cap\Gamma_{2r,16}+h).
$$
\end{proof}

We shall also use a standard tail estimate for Gaussians:
\begin{fact}\label{tailgaussian} (\cite{campos2025singularity}, Fact 4.3)
    $$
\gamma_k(\{x\in\mathbb{R}^k:\|x\|_2^2\geq k\})\leq\exp(-k/8).
    $$
\end{fact}

We also need the following consequence of a Gaussian Burnn-Minkowski type theorem:
\begin{lemma}\label{lemma4.7equation}(\cite{campos2025singularity}, Lemma 4.5)
    Fix $k\in\mathbb{N}$ and let $A\subset\mathbb{R}^k$ be Borel measurable. Then
    \begin{equation}
        \gamma_k(A-A)\geq\gamma_k(A)^4.
    \end{equation} 
\end{lemma}

The next lemma is the key result of this section.  
First we define, for some $S\subset\mathbb{R}^{k+2\ell}$ and $y\in\mathbb{R}^{k+2\ell}$, the translated horizontal fiber
$$
F_y(S,a_1\cdots,a_{2\ell}):=\{\theta_{[k]}=(\theta_1,\cdots,\theta_k)\in\mathbb{R}^k:(\theta_1,\cdots,\theta_k,a_1,\cdots,a_{2\ell})\in S-y\}.
$$
Also let $\mathbf{a}$ denote a $2\ell$-dimensional vector with coordinates $\mathbf{a}=(a_{k+1},\cdots,a_{k+2\ell})\in\mathbb{R}^{2\ell}.$ 

\begin{lemma}\label{lemma4.8}For given $k,\ell\in\mathbb{N}$, integer tuples $\mathbf{s}\in\mathbb{R}_+^{2\ell}$ and some $t>0$, consider a measurable set $S\subset\mathbb{R}^{k+2\ell}$ satisfying
\begin{equation}2(\prod_{i=1}^{2\ell}s_i)t^{2\ell}(\frac{12}{\sqrt{2\ell}})^{2\ell}\left(\max_{y,\mathbf{a}}(\gamma_k(F_y(S;\mathbf{a})-F_y(S;\mathbf{a})))^\frac{1}{4}+e^{-k/8}\right)\leq\gamma_{k+2\ell}(S).
\end{equation}
    Then we can find some $x\in S$ such that
    \begin{equation}
        (\Gamma_{2\sqrt{k},16}\setminus\Gamma_{2\sqrt{k},\mathbf{s},t}+x)\cap S\neq\emptyset.
    \end{equation}
\end{lemma}

\begin{proof} In this proof we take $r=\sqrt{k}$.
We assume the contrapositive, so that for every $x\in S$,
\begin{equation}\label{contrapositiveassumption}
    (\Gamma_{2r,16}\setminus\Gamma_{2r,\mathbf{s},t}+x)\cap S=\emptyset.
\end{equation} We again take $B:=\{\theta\in\mathbb{R}^{k+2\ell}:\|\theta_{[k]}\|_2\leq r\}$.

   \textbf{ First step: upper bound for $\gamma_{k+2\ell}(S\setminus B)$.} For fixed $\theta_{[k]}\in\mathbb{R}^k$ we write 
    \begin{equation}
        \gamma_{k+2\ell}(S\setminus B)=\int_{\|\theta_{[k]}\|_2\geq r}\gamma_{2\ell}(S(\theta_{[k]}))d\gamma_k.
    \end{equation}
    By the contrapositive assumption \eqref{contrapositiveassumption} and Lemma \ref{lemma4.44}, we have 
    \begin{equation}
        \max_{\theta_{[k]}\in\mathbb{R}^k}\gamma_{2\ell}(S(\theta_{[k]}))\leq 2^{2\ell+1}(\prod_{i=1}^{2\ell}s_i)t^{2\ell}(\frac{3}{\sqrt{2\ell}})^{2\ell}.
    \end{equation}
    Then Fact \ref{tailgaussian} implies that
    $$
\gamma_k(\{\|\theta_{[k]}\|_2\geq r\})\leq\exp(-k/8),
    $$ so that combining both facts, we get 
    \begin{equation}
        \gamma_{k+2\ell}(S\setminus B)\leq\gamma_k(\{\|\theta_{[k]}\|_2\geq r\})\max_{\theta_{[k]}\in\mathbb{R}^k}\gamma_{2\ell}(S(\theta_{[k]}))\leq 2(\prod_{i=1}^{2\ell}s_i)t^{2\ell}(\frac{6}{\sqrt{2\ell}})^{2\ell}e^{-k/8}.
    \end{equation}

    \textbf{Second step: upper bound for $\gamma_{k+2\ell}(S\cap B)$.} Thanks to Lemma \ref{lemma4.5}, we can find some $x\in S$ and $h\in\Gamma_{r,8}$ such that
    \begin{equation}
        \gamma_{k+2\ell}(S\cap B)\leq 2^{2\ell}\gamma_{k+2\ell}((S-x)\cap\Gamma_{2r,16}+h).
    \end{equation} As we argue by contradiction in \eqref{contrapositiveassumption}, for this $x\in S$ we have 
    \begin{equation}
        (S-x)\cap\Gamma_{2r,16}\subseteq(S-x)\cap\Gamma_{2r,\mathbf{s},t}. 
    \end{equation} Thus we have, letting $y=x-h$,
    \begin{equation}\label{eq4.177}
        \gamma_{k+2\ell}(S\cap B)\leq 2^{2\ell}\gamma_{k+2\ell}((S-y)\cap (\Gamma_{2r,\mathbf{s},t}+h)).
    \end{equation}Then we only need to bound
    \begin{equation}
        \gamma_{k+2\ell}((S-y)\cap(\Gamma_{2r,\mathbf{s},t}+h))\leq\int_{\|\mathbf{a}-h_{[k+1,k+2\ell]}\|_\mathbf{s}\leq t}\gamma_k(F_y(S;\mathbf{a}))d\gamma_{2\ell},
    \end{equation} where we integrate over $\mathbf{a}\in\mathbb{R}^{2\ell}$. Then we upper bound the Gaussian integral over $\{\|a-h_{[k+1,k+2\ell]}\|_\mathbf{s}\leq t\}$ by Lebesgue integral and we apply Lemma \ref{lemma4.7equation} to $\gamma_k(F_y(S;\mathbf{a}))$ to get
    \begin{equation}\label{eq4.197}
        \gamma_{k+2\ell}((S-y)\cap(\Gamma_{2r,\mathbf{s},t}+h))\leq 2(\prod_{i=1}^{2\ell}s_i)t^{2\ell}(\frac{6}{\sqrt{2\ell}})^{2\ell}\max_{y,\mathbf{a}}(\gamma_k(F_y(S;\mathbf{a})-F_y(S;\mathbf{a})))^{1/4}.
    \end{equation}
    Then we combine \eqref{eq4.177} and \eqref{eq4.197} to deduce that 
    \begin{equation}
        \gamma_{k+2\ell}(S\cap B)\leq 2(\prod_{i=1}^{2\ell}s_i)t^{2\ell}(\frac{12}{\sqrt{2\ell}})^{2\ell}\max_{y,\mathbf{a}}(\gamma_k(F_y(S;\mathbf{a})-F_y(S;\mathbf{a})))^\frac{1}{4}.
    \end{equation}
    Combining both first step and second step, we conclude that
    $$
\gamma_{k+2\ell}(S)\leq  2(\prod_{i=1}^{2\ell}s_i)t^{2\ell}(\frac{12}{\sqrt{2\ell}})^{2\ell}\left(
\max_{y,\mathbf{a}}(\gamma_k(F_y(S;\mathbf{a})-F_y(S;\mathbf{a})))^\frac{1}{4}+e^{-k/8}\right),
    $$ which completes the proof via a contrapositive argument.
\end{proof}

\subsection{Applying the estimate to the level set}

Let $\zeta$ be a mean 0, variance 1 subgaussian random variable with $\zeta\in\Gamma_B$ for some $B>0$, and $\zeta'$ an independent copy of $\xi$. We define 
$$
\tilde{\zeta}=\zeta-\zeta'.
$$ Define $I_B=(1,16B^2)$ and let $\bar{\zeta}$ be $\tilde{\zeta}$ conditioned on the event $|\tilde{\zeta}|\in I_B.$

We let $p:=\mathbb{P}(|\widetilde{\xi}|\in I_B)$. Then by \cite{campos2024least}, Lemma II.1, we have \begin{equation}\label{whatdoesbhave?}p\geq\frac{1}{2^7B^4}.\end{equation}

For a given parameter $\nu\in(0,1)$ let $Z_\nu$ be an independent Bernoulli variable with mean $\nu$, and define 
$$
\xi_\nu:=\mathbf{1}\{|\tilde{\zeta}\in I_B\}\tilde{\zeta}Z_\nu.
$$
We write $X\sim\Phi_\nu(d;\zeta)$ to mean a $d$-dimensional random vector with i.i.d. coordinates of distribution $\tilde{\zeta}Z_\nu$, and we write $X\sim\Xi_\nu(d;\zeta)$ to mean a $d$-dimensional random vector with i.i.d. coordinates of distribution $\xi_\nu$.

Fix $\ell'\in\mathbb{N}$, for any $2d\times \ell'$ matrix $W$ we define the $W$-level set to be, for any $z\geq 0$,
$$
S_W(z):=\{\theta\in\mathbb{R}^{\ell'}:\mathbb{E}_{\bar{\xi}}\|\bar{\xi}W\theta\|_\mathbb{T}\leq\sqrt{z}\}.
$$

We transfer the Levy concentration function to measures on level sets:

\begin{lemma}\label{lemma4.9}(\cite{campos2024least}, Lemma IV.1) Given $\beta>0,\nu\in(0,\frac{1}{4})$ and let $W$ be a $2d\times \ell'$ matrix and $\tau\sim\Phi_\nu(2d;\xi)$. Then we can find some $m>0$ satisfying 
\begin{equation}
    \mathcal{L}(W^T\tau,\beta\sqrt{\ell'})\leq2\exp(2\beta^2\ell'-\nu p m/2)\gamma_{\ell'}(S_W(m)).
\end{equation}
    
\end{lemma}

\begin{lemma}\label{lemma4.10apply}(\cite{campos2024least}, Lemma IV.2) Given $\beta>0$, $\nu\in(0,\frac{1}{4})$, and let $W$ be a $2d\times \ell'$ matrix and $\tau\in\Xi_\nu(2d;\xi)$. Then we have for all $t\geq 0$
$$
\gamma_{\ell'}(S_W(t))e^{-32\nu p t}\leq \mathbb{P}_\tau(\|W^T\cdot\tau\|_2\leq\beta\sqrt{\ell'})+\exp(-\beta^2\ell').
$$
    \begin{fact}\label{fact4.11633}(\cite{campos2024least}, Fact VI.2)
        For any $2d\times (k+2\ell)$- matrix $W$ and any $t>0$, we have
        $$ S_W(t)-S_W(t)\subseteq S_W(4t),
        $$ and for any $y\in\mathbb{R}^{k+2\ell}$ and $\mathbf{a}\in\mathbb{R}^{2\ell}$, we have
        \begin{equation}
            F_y(S_W(t);\mathbf{a})-F_y(S_W(t);\mathbf{a})\subseteq F_0(S_W(4t);\mathbf{0})
        \end{equation} where $\mathbf{0}\in\mathbb{R}^{2\ell}$ is the all 0-s vector.
    \end{fact}
\end{lemma}

Before we state the main inverse Littlewood-Offord theorem, we define the notion of essential LCD we will use, which generalizes Definition \ref{definitiontwocomponentlcds} to multiple vectors.

\begin{Definition}\label{definitionlcd}
Fix some $L\geq 1$ and $\alpha\in(0,1)$.
Consider nonzero vectors $Y_1,\cdots,Y_\ell\in\mathbb{R}^d$ and a tuple of positive constants $\mathbf{t}=(t_1,\cdots,t_\ell)$ with $t_i>0$. We define the essential least common denominator (LCD) of the vector pair $\mathbf{Y}$ associated with these parameters:
$$
\operatorname{LCD}_{L,\alpha}^\mathbf{t}(\mathbf{Y}):=\inf\{\|\theta\|_2:\theta\in \mathbb{R}^\ell,\| \sum_{i=1}^\ell \theta_iY_i\|_\mathbb{T}\leq L\sqrt{\log_+\frac{\alpha\sqrt{\sum_{i=1}^\ell \theta_i^2/t_i^2}}{L}}\},
$$
where $\log_+(x)=\max(\log x,0)$.\end{Definition}  The symbols $L$ and $\alpha$ used here are consistent with the use of symbols in the standard literature \cite{rudelson2016no}, but can lead to confusion with the (same) symbols used in Section \ref{verification1}. We only use $L$ and $\alpha$ here for the role they play in Definition \ref{definitionlcd}, so that the symbols $L$ and $\alpha$ have a different meaning (the threshold function \eqref{thresholdsillus}, and in Lemma \ref{randmgenerationoflcd}) in 
Section \ref{verification1}.

Now we state our first result on conditioned inverse Littlewood-Offord inequalities:

\begin{Proposition}\label{proposition4.12}
    Fix $0<\nu\leq 2^{-15}$, $c_0\leq 2^{-50}B^{-4}$, $d\in\mathbb{N},\alpha\in(0,1)$. Consider $k\leq 2^{-17}B^{-4}\nu\alpha d$ and a sequence of positive real numbers $\mathbf{t}=(t_1,\cdots,t_\ell)$. Consider a tuple of vectors $Y_1,\cdots,Y_{\ell}\in\mathbb{R}^d$ satisfying, for some $\alpha\in(0,1)$, $$\operatorname{LCD}_{L,\alpha}^\mathbf{t}(\mathbf{Y})\geq 256B^2\sqrt{\ell},\quad\text{where } L=(\frac{8}{\sqrt{\nu p}}+\frac{256B^2}{\beta})\sqrt{\ell}.$$
   Let $W$ be an $2d\times k$ matrix satisfying $\|W\|_{op}\leq 2$ and $\|W\|_{HS}\geq\sqrt{k}/2$. We also consider two random vectors $\tau\sim\Phi_\nu(2d;\zeta)$ and $\tau'\sim\Xi_{\nu'}(2d;\zeta)$ where $\nu'=2^{-7}\nu$. Let $\beta\in[c_0/2^{10},\sqrt{c_0}]$ and $\beta'\in(0,1/2)$.

    Then we have the following estimate
    \begin{equation}
        \mathcal{L}(W_\mathbf{Y}^T\tau,\beta\sqrt{k+2\ell})\leq (\frac{R}{\alpha})^{2\ell}(\prod_{i=1}^\ell t_i^2)\exp(4\beta^2(k+\ell))\left(\mathbb{P}(\|W^T\tau'\|_2\leq\beta'\sqrt{k})+\exp(-\beta'^2k)\right)^{1/4}
    \end{equation} where we take 
    $R=2^{36}B^2\nu^{-1/2}c_0^{-2}(\frac{8}{\sqrt{\nu p}}+\frac{256B^2}{\beta})$.
\end{Proposition}

\begin{proof}
    We first apply Lemma \ref{lemma4.9} to obtain some $m>0$ such that the level set
    $$
S:=S_{W_\mathbf{Y}}(m):=\{\theta\in\mathbb{R}^{k+2\ell}:\|W_\mathbf{Y}\theta\|_\mathbb{T}\leq\sqrt{m}\}
    $$ satisfies 
    $$
e^{-\nu pm/2+2\beta^2(k+2\ell)}\gamma_{k+2\ell}(S)\geq\mathcal{L}(W_\mathbf{Y}^T\tau,\beta\sqrt{k+2\ell}).
    $$ This estimate combined with our hypothesis gives a lower bound
\begin{equation}\label{productk+2ell}
\gamma_{k+2\ell}(S)\geq e^{\nu pm/2+2\beta^2k}(\frac{R}{\alpha})^{2\ell}(\prod_{i=1}^\ell t_i^2)T^{1/4}
    \end{equation} where we take 
    $$
T:=\mathbb{P}(\|W^T\tau'\|_2\leq\beta'\sqrt{k})+\exp(-\beta'^2k).
    $$

 We claim that the following inclusion holds:

\begin{Claim}\label{claim4.13} We take $r_0=\sqrt{k}$ and $$s_0=2\alpha^{-1}e^{\nu pm/{8\ell}+\beta^2k/2\ell}\cdot\sqrt{\ell}\cdot(\frac{8}{\sqrt{\nu p}}+\frac{256B^2}{\beta}).$$ We also define $\mathbf{s}=(s_1,\cdots,s_{2\ell})$ via $s_{2i-1}=s_{2i}=t_i$ for each $i\in[\ell]$.

Then we can find some $x\in S\subset\mathbb{R}^{k+2\ell}$ so that 
    \begin{equation}
        (\Gamma_{2r_0,16}\setminus\Gamma_{2r_0,\mathbf{s},s_0}+x)\cap S=\emptyset.
    \end{equation}
\end{Claim}
\begin{proof}
    We prepare to apply Lemma \ref{lemma4.8}. First consider
    $$
M:=\max_{y,\mathbf{a}}\{\gamma_{k}(F_y(S;\mathbf{a})-F_y(S;\mathbf{a}))\},
    $$ where the maximum is over all $\mathbf{a}\in\mathbb{R}^{2\ell}$.
    By Fact \ref{fact4.11633}, we have 
    $$
F_y(S;\mathbf{a})-F_y(S;\mathbf{a})\subseteq F_0(S_{W_Y}(4m);\mathbf{0})
    $$ and we further observe that 
    $$
F_0(S_{W_Y}(4m);\mathbf{0})=\{\theta_{[k]}\in\mathbb{R}^k:\|W\theta_{[k]}\|_\mathbb{T}\leq\sqrt{4m}\}=S_W(4m),
    $$ which is a level set of the decoupled event $\mathbb{P}_{\tau'}(\|W^T\tau'\|_2\leq\beta'\sqrt{k})$. Then we apply Lemma \ref{lemma4.10apply} to obtain, using also $\nu'=2^{-7}\nu$, 
    \begin{equation}
        M\leq \gamma_k(S_W(4m))\leq  e^{128\nu' pm}T=e^{\nu pm}T.
    \end{equation}
We now claim that the following inequality holds:
\begin{equation}\label{claimtohold}
 e^{\nu pm/2+2\beta^2 k}(\frac{R}{\alpha})^{2\ell}\prod_{i=1}^\ell t_i^2\cdot T^{1/4}\geq 2\prod_{i=1}^\ell t_i^2s_0^{2\ell}(\frac{12}{\sqrt{2\ell}})^{2\ell}M^{1/4}+2\prod_{i=1}^\ell t_i^2s_0^{2\ell}(\frac{12}{\sqrt{2\ell}})^{2\ell}e^{-k/8}.
\end{equation} To check this, we verify that each term on the right hand side is bounded by half of the left hand side, but this follows immediately from our designation of $s_0$. Finally we take into account $p\geq\frac{1}{2^7B^4}$ and observe that $T^{1/4}\geq \exp(-\beta'^2k/4)\geq \exp(-k/8)$. Combining all these estimates verifies the validity of \eqref{claimtohold}. Then \eqref{productk+2ell} implies that
    \begin{equation} 
 \gamma_{k+2\ell}(S)\geq\operatorname{RHS \quad of} \eqref{claimtohold}.       
    \end{equation}
    Then we can apply Lemma \ref{lemma4.8} and finish the proof of Claim \ref{claim4.13}.
\end{proof}

Then Claim \ref{claim4.13} has the following immediate consequence:
\begin{Claim}\label{claim4.14}
    We have $S_{W_\mathbf{Y}}(4m)\cap(\Gamma_{2r_0,16}\setminus\Gamma_{2r_0,\mathbf{s},s_0})\neq\emptyset$.
\end{Claim}
\begin{proof}
    By Claim \ref{claim4.13}, we can find $x,y\in S=S_{W_\mathbf{Y}}(m)$ with $y\in(\Gamma_{2r_0,16}\setminus\Gamma_{2r_0,\mathbf{s},s_0} +x)$. Denote by $\phi:=y-x$, then $\phi\in S_{W_\mathbf{Y}}(4m)\cap(\Gamma_{2r_0,16}\setminus\Gamma_{2r_0,\mathbf{s},s_0})$ thanks to Fact \ref{fact4.11633}.
\end{proof}

A connection of our notion to the essential LCD is implied by the following claim:

\begin{fact}\label{fact4.15} For any $\phi\in S_{W_\mathbf{Y}}(4m)$, define the two vectors $$\mathbf{\phi}^1=(\phi_{k+1},\phi_{k+3},\cdots,\phi_{k+2\ell-1}),\quad \mathbf{\phi}^2=(\phi_{k+2},\phi_{k+4},\cdots,\phi_{k+2\ell}).$$ Then if there exists $\phi\in S_{W_\mathbf{Y}}(4m)\cap(\Gamma_{2r_0,16}\setminus\Gamma_{2r_0,\mathbf{s},s_0})$, then we can find some $i\in\{1,2\}$ and some $\bar{\zeta}_0\in(1,16B^2)$ such that 
$$
\|\bar{\zeta}_0\mathbf{\phi}^i\cdot\mathbf{Y}\|_\mathbb{T}\leq L\log_+\sqrt{\frac{\alpha\sqrt{\sum_{j=1}^\ell \frac{\psi_{k+2j-1}^2}{t_j^2}}}{L}},\quad L=(\frac{8}{\sqrt{\nu p}}+\frac{256B^2}{\beta})\sqrt{\ell}
$$
    where for $i\in\{1,2\}$ the notation $\phi^i\cdot \mathbf{Y}$ denotes $\sum_{j=1}^{\ell}\phi_{k+2j-2+i}Y_j$.

    Therefore, $S_{W_\mathbf{Y}}(4m)\cap(\Gamma_{2r_0,16}\setminus\Gamma_{2r_0,\mathbf{s},s_0})\neq\emptyset$ implies that $\operatorname{LCD}_{L,\alpha}^\mathbf{t}(\mathbf{Y})\leq 256B^2\sqrt{\ell}.$
\end{fact} 

\begin{proof}
    Since $\phi\in S_{W_\mathbf{Y}}(4m)$, we have $\mathbb{E}_{\bar{\xi}}\|\bar{\xi}W_\mathbf{Y}\phi\|_\mathbb{T}^2\leq 4m$. Then there exists some instance $\bar{\xi}_0\in(1,16B^2)$ satisfying
    \begin{equation}
        \|\bar{\xi}_0W_\mathbf{Y}\phi\|_\mathbb{T}^2\leq 4m.
    \end{equation} We denote by $\psi:=\bar{\xi}_0\phi$. Then we can find some $p\in\mathbb{Z}^{2d}$ satisfying $W_Y\psi\in B_{2d}(p,2\sqrt{m})$. Then from the expansion 
    $$
W_\mathbf{Y}\psi=W\psi_{[k]}+\sum_{j=1}^{\ell}\psi_{k+2j-1}\begin{bmatrix}Y_j\\\mathbf{0}_d\end{bmatrix}+\sum_{j=1}^{\ell}\psi_{k+2j}\begin{bmatrix}\mathbf{0}_d\\Y_j\end{bmatrix}
    $$
combined with the fact that $W_Y\psi\in B_{2d}(p,2\sqrt{m})$, we deduce that
\begin{equation}\label{blockporjectequation}
    \sum_{j=1}^{\ell}\psi_{k+2j-1}\begin{bmatrix}Y_j\\\mathbf{0}_d\end{bmatrix}+\sum_{j=1}^{\ell}\psi_{k+2j}\begin{bmatrix}\mathbf{0}_d\\Y_j\end{bmatrix}\in B_{2d}(p,2\sqrt{m})-W\psi_{[k]}\subseteq B_{2d}(p,2\sqrt{m}+64B^2\sqrt{k}),
\end{equation}
where the last inclusion follows from the fact that $\phi\in\Gamma_{2r_0,16}$ so that $\|\phi_{[k]}\|_2\leq 2\sqrt{k}$, and thus $\|\psi_{[k]}\|_2\leq 32B^2\sqrt{k}$. Finally, we use $\|W\|_{op}\leq 2$.

Now, the assumption that $\phi\in\Gamma_{2r_0,16}\setminus \Gamma_{2r_0,\mathbf{s},\mathbf{s}_0}$ means that $\|\psi_{[k+1,k+2\ell]}\|_\mathbf{s}\geq s_0$. Recall that we take $s_{2j-1}=s_{2j}=t_j$ for each $j\in[\ell]$ in Claim \ref{claim4.13}. Then at least one of the following two inequalities must be true:
$$
\sum_{j=1}^\ell\frac{\psi_{k+2j-1}^2}{t_{j}^2}\geq\frac{1}{2}s_0^2,\quad\text{ or   } \sum_{j=1}^\ell\frac{\psi_{k+2j}^2}{t_{j}^2}\geq\frac{1}{2}s_0^2.
$$
Assume without loss of generality that the first inequality holds. 
Then we project equation \eqref{blockporjectequation} onto the first $d$ coordinates, and we have 
$$
\sum_{j=1}^\ell \psi_{k+2j-1}Y_j\in B_d(p_{[d]},2\sqrt{m}+64B^2\sqrt{k}).
$$ This implies that $\|\sum_{j=1}^\ell \psi_{k+2j-i}Y_j\|_\mathbb{T}\leq 2\sqrt{m}+64B^2\sqrt{k}$. On the other hand, by our choice of $s_0$, we have that 
$$
L\log_+\sqrt{\frac{\alpha\sqrt{\sum_{j=1}^\ell \frac{\psi_{k+2j-1}^2}{t_j^2}}}{L}}\geq 2\sqrt{m}+64B^2\sqrt{k},\quad L=(\frac{8}{\sqrt{\nu p}}+\frac{256B^2}{\beta})\sqrt{\ell}.
$$  By assumption, $\|\phi_{[k+1,k+2\ell]}\|_\infty\leq 16$, so that $\|(\psi_{k+1},\psi_{k+3},\cdots,\psi_{k+2\ell-1})\|_2\leq 256B^2\sqrt{\ell}$.
 This completes the proof.
\end{proof} Combining Claim \ref{claim4.14} and Fact \ref{fact4.15} completes the proof of Proposition \ref{proposition4.12}.
\end{proof}

\subsection{Completing the proof of conditioned inverse Littlewood-Offord theorem}

We shall also use the following simple consequence of Talagrand's inequality:
\begin{lemma}(\cite{campos2024least}, Lemma VI.7)
    Fix $\nu\in(0,1)$ and $\beta'\in(0,2^{-7}B^{-2}\sqrt{\nu})$. Let 
     $\tau'\sim\Xi_\nu(2d;\zeta)$. Let $W$ be some $2d\times k$ matrix with $\|W\|_{HS}\geq\sqrt{k}/2$ and $\|W\|_{op}\leq 2$. Then 
    \begin{equation}
        \mathbb{P}(\|W^T\tau'\|_2\leq\beta'\sqrt{k})\leq4\exp(-2^{-20}B^{-4}\nu k).
    \end{equation}
\end{lemma}

Now we state and prove our main conditioned inverse Littlewood-Offord theorem.
\begin{theorem}\label{twolittlewoodofford}
    For given $0<\nu\leq 2^{-15}$, $c_0\leq 2^{-35}B^{-4}\nu$, $d\in\mathbb{N}$ and $\alpha\in(0,1)$. Consider an $\ell$-tuple of vectors $\mathbf{Y}=(Y_1,\cdots,Y_\ell)\in\mathbb{R}^d$ and a tuple of real numbers $\mathbf{t}=(t_1,\cdots,t_\ell)\in\mathbb{R}_+.$ Let $W$ be some $2d\times k$ matrix satisfying $\|W\|\leq 2$, $\|W\|_{HS}\geq\sqrt{k}/2$. Let $\tau\sim\Phi_\nu(2d;\xi)$. 

    Then if $\operatorname{LCD}_{L,\alpha}^\mathbf{t}(\mathbf{Y})\geq 256B^2\sqrt{\ell}$ with $L=(\frac{8}{\sqrt{\nu p}}+\frac{256B^2}{\sqrt{c_0}})\sqrt{\ell}$, then we must have
    \begin{equation}
        \mathcal{L}(W_\mathbf{Y}^T\tau,c_0^{1/2}\sqrt{k+2\ell})\leq (\frac{R}{\alpha})^{2\ell}(\prod_{i=1}^\ell t_i^2)\exp(-c_0k),
    \end{equation}where we take $R=2^{38}B^2\nu^{-1/2}c_0^{-2}(\frac{8}{\sqrt{\nu p}}+\frac{256B^2}{\sqrt{c_0}}).$
\end{theorem}

\begin{proof}
    By assumption, $c_0\leq 2^{-35}B^{-4}\nu$. The matrix $W$ satisfies $\|W\|_{HS}\geq \sqrt{k}/2$ and $\|W\|_{op}\leq 2$. Now we apply Proposition \ref{proposition4.12} with the choice $\beta'=64\sqrt{c_0}$ and $\tau'\sim\Xi_{\nu'}(2d,\xi)$ with $\nu'=2^{-7}\nu$ to deduce that 
    $$
\mathbb{P}(\|W^T\tau'\|_2\leq\beta'\sqrt{k})\leq 4\exp(-2^{-27}B^{-4}\nu k)\leq 4\exp(-32c_0k).
    $$ Now, if $\beta\leq \sqrt{c_0}$, then 
    $$
e^{4\beta^2k}\left(\mathbb{P}(\|W^T\tau'\|_2\leq\beta'\sqrt{k})+\exp(-\beta'^2k)\right)^{1/4}\leq2\exp(4c_0k-8c_0k)\leq 2\exp(-c_0k).
    $$Then we can apply Proposition \ref{proposition4.12} to deduce that
\begin{equation}
    \mathcal{L}(W_\mathbf{Y}^T\tau,c_0^{1/2}\sqrt{k+2\ell})\leq 2(\frac{R}{\alpha})^{2\ell}(\prod_{i=1}^\ell t_i^2)\exp(-c_0k).
\end{equation}
\end{proof}

\subsection{Upgrading to singular value of submatrices}\label{upgradingtos}
We now upgrade Theorem \ref{twolittlewoodofford} to an inverse Littlewood-Offord theorem conditioned on the matrix rank. 
We recall two facts on estimating the least singular value of a rectangular random matrix.

The first is the random-rounding technique, introduced by Livshyts \cite{livshyts2021smallest} to the random matrix setting, to find a net for a rectangular matrix.  Let $\mathcal{U}_{2d,k}$ consist of all $2d\times k$ matrices with orthogonal columns. Then we have a net for $\mathcal{U}_{2d,k}$ as a subset of $\mathbb{R}^{[2d]\times[k]}$:

\begin{lemma}\label{netofmatrices}(\cite{campos2025singularity}, Lemma 6.4) Fix $k\leq d$ and $\delta\in(0,\frac{1}{2})$. There exists a net $\mathcal{W}=\mathcal{W}_{2d,k}\subset\mathbb{R}^{[2d]\times[k]}$ satisfying $|\mathcal{W}|\leq (64/\delta)^{2dk}$, such that for any $U\in\mathcal{U}_{2d,k}$, any $r\in\mathbb{N}$ and any given $r\times 2d$ matrix $A$ we can find $W\in\mathcal{W}$ satisfying
\begin{enumerate}
    \item $\|A(W-U)\|_{HS}\leq\delta(k/2d)^{1/2}\|A\|_{HS}$,
    \item $\|W-U\|_{HS}\leq\delta\sqrt{k}$, and also
    \item $\|W-U\|_{op}\leq8\delta$.
\end{enumerate}
\end{lemma}

We shall also be using the following simple fact
\begin{fact}\label{factofcorank}(\cite{campos2025singularity}, Fact 6.5)
    Take $d,n\in\mathbb{N}$ with $3d\leq n$. Consider $H$ an $(n-d)\times 2d$ matrix. When $\sigma_{2d-k+1}(H)\leq x$ then we can find $k$ orthogonal unit vectors $v_1,\cdots,v_k\in\mathbb{R}^{2d}$ with $\|Hv_i\|_2\leq x,i\in[k]$. That is, we can find $W\in\mathcal{U}_{d,2k}$ satisfying $\|HW\|_{HS}\leq x\sqrt{k}$.
\end{fact}

We also need a control on $\|H\|_{HS}$:
\begin{fact}\label{largedeviationhilbertschmidt}[\cite{campos2024least},Fact VII.5] Let $H$ be a random matrix of size $(n-d)\times 2d$ with i.i.d. rows of distribution $\Phi_\nu(2d;\zeta)$. Then we have
\begin{equation}
    \mathbb{P}(\|H\|_{HS}\geq 2\sqrt{d(n-d)})\leq2\exp(-2^{-21}B^{-4}nd).
\end{equation}
    
\end{fact}
 The next lemma, specialized to the special case $
\ell=1$, is Lemma 6.3 of \cite{campos2024least}.
\begin{lemma}\label{tensorization2}
    Fix $d<n$ and $k\geq 0$. Consider $W$ a $2d\times(k+2\ell)$ matrix and $H$ a $(n-d)\times 2d$ random matrix with i.i.d. rows. Consider $\tau\in\mathbb{R}^{2d}$ a random vector having the same distribution as a row of $H$. Then if we take $\beta\in(0,\frac{1}{8})$, we have
    $$
\mathbb{P}_H(\|HW\|_{HS}\leq\beta^2\sqrt{(k+2\ell)(n-d)})\leq (32e^{2\beta^2(k+2\ell)}\mathcal{L}(W^T\tau,\beta\sqrt{k+2\ell}))^{n-d}.
    $$
\end{lemma}

\begin{proof}
    By Markov's inequality, 
    \begin{equation}\label{bymarkov}
        \mathbb{P}(\|HW\|_{HS}\leq \beta^2\sqrt{(k+2\ell)(n-d)})\leq\exp(2\beta^2(k+2\ell)(n-d))\mathbb{E}_He^{-2\|HW\|_{HS}^2/\beta^2}.
    \end{equation}
    Now we take $\tau_1,\cdots,\tau_{n-d}$ be an i.i.d. copy of the rows of $H$, then we have
    \begin{equation}\label{wehavecauchy}
        \mathbb{E}_He^{-2\|HW\|_{HS}^2/\beta^2}=(\mathbb{E}_\tau e^{-2\|W^T\tau\|_2^2/\beta^2})^{n-d}.
    \end{equation}
We can compute that 
$$
\mathbb{E}_\tau e^{-2\|W^T\tau\|_2^2/\beta^2}=\int_0^\infty 4u e^{-2u^2}\mathbb{P}(\|W^T\tau\|_2/\beta\leq u)du.
$$ Splitting the integral on the right hand side into two pars,
$$
\mathbb{E}_\tau e^{-2\|W^T\tau\|_2^2/\beta^2}=\int_0^{\sqrt{k+2\ell}}4ue^{-2u^2}\mathbb{P}(\|W^T\tau\|_2\leq\beta u)+\int_{\sqrt{k+2\ell}}^\infty 4ue^{-2u^2}\mathbb{P}(\|W^T\tau\|_2\leq\beta u).
$$
    We now use Fact \ref{fact4.21} to estimate
    $$\begin{aligned}
&\mathbb{E}_\tau e^{-2\|W^T\tau\|_2^2/\beta^2}\\&\leq\mathcal{L}(W^T\tau,\beta\sqrt{k+2\ell})(\int_0^{\sqrt{k+2\ell}}4ue^{-2u^2}du+\int_{\sqrt{k+2\ell}}^\infty(1+\frac{2u}{\sqrt{k+2\ell}})^{k+2\ell}4ue^{-2u^2}du)\\&\leq 9\mathcal{L}(W^T\tau,\beta\sqrt{k+2\ell}).
    \end{aligned}$$ To verify the last inequality, note that $\int_0^{\sqrt{k+2\ell}}4ue^{-2u^2}du\leq 1$ and the second integral is $\leq 8$ by Fact \ref{remainingcomp}. Combining this least estimate with \eqref{bymarkov} and \eqref{wehavecauchy} we conclude that
    $$
\mathbb{P}_H(\|HW\|_{HS}\leq\beta^2\sqrt{(k+2\ell)(n-d)})\leq\left(9\exp(2\beta^2(k+2\ell))\mathcal{L}(W^T\tau,\beta\sqrt{k+2\ell})\right)^{n-d}
,    $$
    which concludes the proof.
\end{proof}
We fill up the one remaining computation in this proof:
\begin{fact}\label{remainingcomp}
    For $k\geq 0,\ell\geq 0$ we have the following inequality
$$
\int_{\sqrt{k+2\ell}}^\infty(1+\frac{2u}{\sqrt{k+2\ell}})^{k+2\ell}ue^{-2u^2}du\leq 2.
$$\end{fact} 
\begin{proof}
    Applying the inequality $1+x\leq e^{x^2/3}$ for $x\geq 2$, we get 
    $$\int_{\sqrt{k+2\ell}}^\infty(1+\frac{2u}{\sqrt{k+2\ell}})^{k+2\ell}ue^{-2u^2}du\leq\int_{\sqrt{k+2\ell}}^\infty ue^{-2u^2/3}du\leq 2.$$
    
\end{proof}

Now we state and prove the main inverse Littlewood-Offord theorem of this paper.

\begin{theorem}\label{theorem12.234567}
    For given $n\in\mathbb{N}$, $0<c_0\leq 2^{-50}B^{-4}$, let $d\leq c_0^2n$ and fix $\alpha\in(0,1)$. Let $\mathbf{t}=(t_1,\cdots,t_\ell)$ be a tuple of positive real numbers. Consider an $\ell$-tuple of vectors $X_1,\cdots,X_\ell\in\mathbb{R}^d$ that satisfy $\operatorname{LCD}_{L,\alpha}^\mathbf{t}(\frac{c_0}{32\sqrt{n}}\mathbf{X})\geq 256B^2\sqrt{\ell}$ (where we denote by $\frac{c_0}{32\sqrt{n}}\mathbf{X}:=(\frac{c_0}{32\sqrt{n}}X_1,\cdots,\frac{c_0}{32\sqrt{n}}X_\ell)$) and where we take $L=(\frac{8}{\sqrt{\nu p}}+\frac{256B^2}{\sqrt{c_0}})\sqrt{\ell}$. Let $H$ be an $(n-d)\times 2d$ random matrix with i.i.d. rows of distribution $\Phi_\nu(2d;\zeta)$ with $\nu=2^{-15}$ (and we recall $p\geq\frac{1}{2^7B^4}$ from \eqref{whatdoesbhave?}.  Then whenever $k\leq 2^{-21}B^{-4}d$ and $\prod_{i=1}^\ell({Rt_i})\geq\exp(-2^{-23}B^{-4}d)$, we have the estimate
    \begin{equation}\begin{aligned}
        \mathbb{P}&_H\left(\sigma_{2d-k+1}(H)\leq c_02^{-4}\sqrt{n}\text{ and }\|H_1X_i\|_2,\|H_2X_i\|_2\leq n\quad\text{for all }i\in[\ell]\right)\\&\leq e^{-c_0nk/12}(\prod_{i=1}^\ell\frac{R t_i}{\alpha})^{2n-2d},
    \end{aligned}\end{equation}
    where $H_1=H_{[n-d]\times[d]}$, $H_2=H_{[n-d]\times[d+1,2d]}$ and $R=2^{44}B^2c_0^{-3}(\frac{8}{\sqrt{\nu p}}+\frac{256B^2}{\sqrt{c_0}})$.
\end{theorem}

\begin{proof}
    Let $Y_i:=\frac{c_0}{32\sqrt{n}}\cdot X_i$ for each $i\in[\ell]$. Then by Fact \ref{factofcorank}, 
    $$\begin{aligned}
&\mathbb{P}(\sigma_{2d-k+1}(H)\leq c_02^{-4}\sqrt{n}\text{ and }\|H_1X_i\|_2,\|H_2X_i\|_2\leq n\text{ for each }i\in[\ell])\\&\leq
\mathbb{P}(\exists U\in\mathcal{U}_{2d,k}:\|HU_\mathbf{Y}\|_{HS}\leq c_0\sqrt{n(k+2\ell)}/8).
    \end{aligned}$$ We now take $\delta:=c_0/16$ and let $\mathcal{N}$ be the net given in Lemma \ref{netofmatrices}.

For the given $H$, if $\|H\|\leq 2\sqrt{d(n-d)}$ and if there exists $U\in\mathcal{U}_{2d,k}$ with $\|HU_\mathbf{Y}\|_{HS}\leq c_0\sqrt{n(k+2\ell)}/8$, then Lemma \ref{netofmatrices} implies the existence of $W\in\mathcal{W}$ such that 
$$
\|HW_\mathbf{Y}\|_{HS}\leq \|H(W_\mathbf{Y}-U_\mathbf{Y})\|_{HS}+\|HU_\mathbf{Y}\|_{HS}\leq \delta(k/2d)^{1/2}\|H\|_{HS}+c_0\sqrt{n(k+2\ell)}/8
$$ and this is bounded from above by $c_0\sqrt{n(k+2\ell)}/4$. Thus we have
\begin{equation}\label{bigsums}\begin{aligned}
&\mathbb{P}_H(\exists U\in\mathcal{U}_{2d,k}:\|HU_\mathbf{Y}\|_{HS}\leq c_0\sqrt{n(k+2\ell)}/8)\\&\leq \mathbb{P}_H(\exists W\in\mathcal{W}:\|HW_\mathbf{Y}\|_{HS}\leq c_0\sqrt{n(k+2\ell)}/4)+\mathbb{P}_H(\|H\|_{HS}\geq 2\sqrt{d(n-d)})\\&\leq\sum_{W\in\mathcal{W}}\mathbb{P}_H(\|HW_\mathbf{Y}\|_2\leq c_0\sqrt{n(k+2\ell)}/4)+2\exp(-2^{-21}B^{-4}nd)
\end{aligned}\end{equation} where in the last step we used Fact \ref{largedeviationhilbertschmidt}.
    We can bound 
    $$
|\mathcal{W}|\leq (64/\delta)^{2dk}\leq \exp(32dk\log c_0^{-1})\leq\exp(c_0k(n-d)/6)
    $$ where we use the assumption $d\leq c_0^2n$. Then \begin{equation}
        \begin{aligned}&\sum_{W\in\mathcal{W}}\mathbb{P}_H(\|HW_\mathbf{Y}\|_2\leq c_0\sqrt{n(k+2\ell)}/4)\\&\leq\exp(c_0k(n-d)/6)\max_{W\in\mathcal{W}}\mathbb{P}_H(\|HW_\mathbf{Y}\|_2\leq c_0\sqrt{n(k+2\ell)}/4).
    \end{aligned}\end{equation}
    Now for this $W\in\mathcal{W}$, we apply Lemma \ref{tensorization2} where we set $\beta:=\sqrt{c_0/3}$ and get
    \begin{equation}\label{tensorizationmid}
        \mathbb{P}_H(\|HW_\mathbf{Y}\|_2\leq c_0\sqrt{n(k+2\ell)}/4)\leq(32e^{2c_0(k+2\ell)/3}\mathcal{L}(W_\mathbf{Y}^T\tau,c_0^{1/2}\sqrt{k+2\ell}))^{n-d}.
    \end{equation}

Now we prepare for the application of Theorem \ref{twolittlewoodofford}. Recall that $\nu=2^{-15}$. By assumption, $\operatorname{LCD}_{L,\alpha}^\mathbf{t}(\frac{c_0}{32\sqrt{n}}\mathbf{X})=\operatorname{LCD}_{L,\alpha}^\mathbf{t}(\mathbf{Y})\geq 256B^2\sqrt{\ell}$. Also, for each $W\in\mathcal{W}$, we have that $\|W\|_{op}\leq 2$ and $\|W\|_{HS}\geq\sqrt{k}/2$. Then applying Theorem \ref{twolittlewoodofford}, we get that
$$
\mathcal{L}(W_\mathbf{Y}^T\tau,c_0^{1/2}\sqrt{k+2\ell})\leq (\frac{R}{\alpha})^{2\ell}(\prod_{i=1}^\ell t_i^2)\exp(-c_0k).
$$

Substituting this into \eqref{tensorizationmid}, we deduce that 
$$
\max_{W\in\mathcal{W}}\mathbb{P}_H(\|HW_\mathbf{Y}\|_2\leq c_0\sqrt{n(k+2\ell)}/4)\leq\frac{1}{2}(\prod_{i=1}^\ell \frac{2R\cdot t_i}{\alpha})^{2n-2d}e^{-c_0k(n-d)/3}.
$$
Combining all the established bounds, we deduce that
$$\begin{aligned}&
\mathbb{P}(\sigma_{2d-k+1}(H)\leq c_0\sqrt{n}/16\text{ and }\|H_1X_i\|_2,\|H_2X_i\|_2\leq n\text{ for each }i\in[\ell])\\&\leq (\prod_{i=1}^\ell\frac{R\cdot t_i}{\alpha})^{2n-2d}e^{-c_0k(n-d)/6}\leq  (\prod_{i=1}^\ell\frac{R\cdot t_i}{\alpha})^{2n-2d}e^{-c_0nk/12},\end{aligned}
$$ where the restrictions $k\leq 2^{-21}B^{-4}d$ and $\prod_{i=1}^\ell Rt_i\geq\exp(-2^{-23}B^{-4}d)$ in the theorem are used to guarantee that the right hand side of the last line is larger than $2\exp(-2^{-21}B^{-4}nd)$
(and where we change the value of $R$ at the last step) and thus complete the proof.
    
\end{proof}

\subsection{Proof of simplicity of singular values and more}\label{lastbooks}

As Corollary \ref{theorem1.69} is a special case of Corollary \ref{theorem1.7}, we only prove the latter. The proof uses a simple covering argument.

\begin{proof}[\proofname\ of Corollary \ref{theorem1.7}]

    We start with the first claim. Let $\epsilon>0$ and consider any pair of eigenvalues $(x_1,x_2)\in I_\kappa$ satisfying $|\sum_{i=1}^2a_ix_i-c|\leq\epsilon n^{-3/2}$. We find an $\epsilon n^{-3/2}$-net $\mathcal{N}_\epsilon$ for $I_\kappa$ and choose a point $p_2$ from $\mathcal{N}_\epsilon$ minimizing $|p_2-x_2|$. Then $|a_1x_1+a_2p_2-c|\leq (|a_2|+1)\epsilon n^{-3/2}$. Since $a_1=1$ we can find at most $(|a_2|+4)$ points in $\mathcal{N}_\epsilon$ such that any $x_1$ in the possible eigenvalue pair $(x_1,x_2)$ should have distance at most $\epsilon n^{-3/2}$ to one of these points. For a given $x_1$ denote by $p_1$ the one of these points such that $|x_1-p_1|\leq \epsilon n^{-3/2}$. The fact that $x_1,x_2$ are eigenvalues imply that $\sigma_{min}(A_n-x_i I_n)=0$ for both $i=1,2$, and thus $\sigma_{min}(A_n-p_i I_n)\leq |x_i-p_i|\leq\epsilon n^{-3/2}$ for both $i=1,2$.
    
      Then the probability that $A_n$ has two distant eigenvalues in $I_\kappa$ satisfying the assumption is upper bounded by (where $|\mathcal{N}_\epsilon|$ denotes the cardinality of $\mathcal{N}_\epsilon$)
    $$\begin{aligned}&(|a_2|+4)|\mathcal{N}_\epsilon|\mathbb{P}(\sigma_{min}(A_n-p_i I_n)\leq\epsilon n^{-3/2},i=1,2)
    \\&\leq (|a_2|+4)
(\frac{2\sqrt{n}}{\epsilon n^{-3/2}})(C\epsilon^2n^{-2}+e^{-\Omega(n)})\leq2C(|a_2|+4)\epsilon +e^{-\Omega(n)},\end{aligned}
    $$ where we apply Theorem \ref{Theorem1.1} to get a bound for the second term, and in the last inequality we assume $\epsilon\geq e^{-cn}$ for a suitable $c>0$. As we only consider \textit{distant} eigenvalues, we can assume that $|p_i-p_j|\geq\Delta$ for any $i\neq j$ and we complete the proof. For the second claim, it suffices to use Theorem \ref{Theorem1.2} instead. 
\end{proof}

\section*{Funding}
The author is supported by a Simons Foundation Grant (601948, DJ).

\printbibliography

\end{document}